\documentclass{article}

\usepackage[utf8]{inputenc}
\usepackage{geometry}
\usepackage[textsize=footnotesize,textwidth=20ex,colorinlistoftodos]{todonotes}
    
\usepackage{amsmath, amssymb, amsthm}
\usepackage[capitalize]{cleveref}
\usepackage[affil-it]{authblk}

\numberwithin{equation}{section}

\newtheorem{theorem}{Theorem}[section]
\newtheorem{proposition}[theorem]{Proposition}
\newtheorem{corollary}[theorem]{Corollary}
\newtheorem{lemma}[theorem]{Lemma}

\theoremstyle{definition}
\newtheorem{definition}[theorem]{Definition}
\newtheorem{remark}[theorem]{Remark}
\newtheorem{example}[theorem]{Example}
\newtheorem{assumption}[theorem]{Assumption}
\newtheorem{construction}[theorem]{Construction}

\DeclareMathOperator{\Ima}{Im}
\DeclareMathOperator{\End}{End}
\DeclareMathOperator{\Lie}{Lie}

\begin{document}
	\title{Operator Kantor pairs}
	
	\author[1]{Sigiswald Barbier %
		\thanks{Electronic address: \texttt{Sigiswald.Barbier@UGent.be}}}
	\affil[1]{Department of Electronics and Information Systems\\ Faculty of Engineering and
		Architecture, Ghent University \\ Krijgslaan 281 \\ Building S8\\ 9000 Ghent, Belgium}
	\author[2]{Tom De Medts \thanks{Electronic address: \texttt{Tom.DeMedts@UGent.be}}}
	\affil[2]{Department of Mathematics: Algebra and Geometry, Ghent University\\ Krijgslaan 281 \\ Building S25\\ 9000 Ghent, Belgium}
	\author[1]{Michiel Smet \thanks{Electronic address: \texttt{Michiel.Smet@UGent.be}; Corresponding author}}
	
	\date{March 30, 2023}

	\maketitle
	
	\begin{abstract}
		Kantor pairs, (quadratic) Jordan pairs, and similar structures have been instrumental in the study of $\mathbb{Z}$-graded Lie algebras and algebraic groups.
		We introduce the notion of an operator Kantor pair, a generalization of Kantor pairs to arbitrary (commutative, unital) rings, similar in spirit as to how quadratic Jordan pairs and algebras generalize linear Jordan pairs and algebras. 
		Such an operator Kantor pair is formed by a pair of $\Phi$-groups $(G^+,G^-)$ of a specific kind, equipped with certain homogeneous operators.
		For each such a pair $(G^+,G^-)$, we construct a $5$-graded Lie algebra $L$ together with actions of $G^\pm$ on $L$ as automorphisms.
		Moreover, we can associate a group $G(G^+,G^-) \subset \operatorname{Aut}(L)$ to this pair generalizing the projective elementary group of Jordan pairs.
		If the non-$0$-graded part of $L$ is projective, we can uniquely recover $G^+,G^-$ from $G(G^+,G^-)$ and the grading on $L$ alone.
		We establish, over rings $\Phi$ with $1/30 \in \Phi$, a one to one correspondence between Kantor pairs and operator Kantor pairs. 
		Finally, we construct operator Kantor pairs for the different families of central simple structurable algebras.
	\end{abstract}
	
	\paragraph{2020 Mathematics Subject classification:} 17B60, 17B70, 17A30, 17C99, 16T05 
	
	\paragraph{Keywords} graded Lie algebras, Kantor pairs, structurable algebras, Hopf algebras
	
	\section*{Introduction}
	
	Structurable algebras were introduced over fields of characteristic different from $2$ and $3$ and classified over fields of characteristic $0$ by Bruce Allison \cite{ALL78}. Oleg Smirnov \cite{smirnov1990} generalized the classification of these algebras to fields with characteristic different from $2$, $3$ and $5$. More recently, Anastasia Stavrova \cite{STAV20} classified the structurable algebras over fields of characteristic different from $2$ and $3$ in a different fashion.
	These algebras are useful in the study of $5$-graded Lie algebras and algebraic groups \cite{Gar01,Kru07,BOE19, Cuy21}, which is the point of view we shall take throughout this paper. 
	Furthermore, the definition of structurable algebras can be extended to arbitrary rings \cite[Section 5]{ALLFLK93}, but not necessarily in a manner suitable for the study of $5$-graded Lie algebras. 
	However, a good definition for structurable algebras over arbitrary rings, which allows for a generalization of the links with (algebraic) groups, was still lacking.
	
	We will not only focus on structurable algebras, but also on Kantor pairs and Kantor triple systems, as introduced ---under the name ``generalized Jordan triple systems of second order''--- by Isaiah Kantor \cite{Kan72}. 
	Kantor pairs are useful for the study of $5$-graded Lie algebras and related groups as well \cite{ALLFLK99,ALLFLK17}.
	Just like for structurable algebras, there is no good theory yet for Kantor pairs over arbitrary rings.
	Since each structurable algebra induces a Kantor pair with the same $5$-graded Lie algebra, we will focus on generalizing Kantor pairs to ``operator Kantor pairs'' as a means to generalize both.
	
	We introduce the notion of an operator Kantor pair.
	The operator Kantor pairs relate to Kantor pairs, like quadratic Jordan pairs \cite{Loos74,Loos75} relate to linear Jordan pairs (``verbundene Paare'' in \cite{Mey70Kon}). To be specific, instead of the $3$-linear map $(a,b,c) \mapsto V_{a,b} c$ associated to a Kantor pair, we will consider certain operators $Q^\text{grp}$, $T$ and $P$ which can be constructed from $V$ over rings with\footnote{We will often talk about ``a ring $\Phi$ with $1/n$'' or write ``$1/n \in \Phi$'' when we mean that we require $n$ to be invertible in the ring $\Phi$.}~$1/6$.
	
	We construct, for each operator Kantor pair $(G^+,G^-)$, an associated $5$-graded Lie algebra. If $1/2 \in \Phi$, then this $5$-graded Lie algebra can, alternatively, be constructed from a genuine Kantor pair. We also construct a group $G(G^+,G^-)$ generalizing the projective elementary group. We show that the group $G(G^+,G^-)$ has the properties which are desired for the projective elementary group if $G^+$ and $G^-$ are projective modules. Namely, we prove that it is not only a group of automorphisms of a Lie algebra, but that the groups $G^+$ and $G^-$ correspond exactly to the ``positive" and ``negative" part of $G(G^+,G^-)$.

	For rings with $1/30 \in \Phi$, we show that each Kantor pair induces an operator Kantor pair.
	Moreover, we prove that all classes of central simple structurable algebras induce operator Kantor pairs.
	More precisely:
	\begin{enumerate}
		\item We show that each quadratic Jordan algebra becomes an operator Kantor pair with operators $Q^\text{grp} = Q,$ $T = 0,$ and $P_xy = Q_xQ_yx$, by establishing the link between the operators and parts of quasi-inverses.
		\item Similarly, we show that a Kantor pair associated to a hermitian form, as studied by John Faulkner and Bruce Allison \cite{ALLFLK99}, induces an operator Kantor pair over arbitrary rings using an appropriate generalization of quasi-inverses.
		In particular, each associative algebra with involution induces an operator Kantor pair.
		\item For each structurable algebra associated to a hermitian form \cite[8.iii]{ALL78}, we also construct an operator Kantor pair, by making use of a Peirce context in the sense of \cite{ALLFLK99} for this class of algebras.
		\item For the other classes, i.e.,
		\begin{enumerate}
			\item forms of the tensor products of composition algebras,
			\item forms of matrix structurable algebras, described in terms of hermitian cubic norm structures \cite{DeMedts2019},
			\item Smirnov algebras (if $1/2 \in \Phi$),
		\end{enumerate}
		we also prove that these induce operator Kantor pairs using certain algebras defined over $\mathbb{Z}$. However, for the tensor product of composition algebras and Smirnov algebras, we will not consider arbitrary forms.
	\end{enumerate}

	\paragraph{Outline.}
	
	In \cref{se:vectorgroups}, we introduce vector groups. These groups are $\Phi$-group functors $G$ equipped with a subgroup functor $G_2$, together with some sort of scalar multiplications. These vector groups can be identified with $\Phi$-groups that can be represented as formal power series $1 + tg_1 + t^2g_2 + \dotsm$. For each vector group, we construct a universal representation, which is a representation in a Hopf algebra $H$, in terms of such formal power series. We also define homogeneous maps on such groups and develop a method that will allow us to derive necessary equalities involving homogeneous maps using this universal representation.
	
	In \cref{se:pairs}, we introduce some notation to work with a pair of vector group representations, formulate Theorem \ref{thm main} which we prove in the appendix, and list a multitude of equations in Lemma~\ref{Lemma equations} that hold for pairs of vector group representations. Theorem \ref{thm main} and Lemma \ref{Lemma equations} are technical results that will be very useful when proving that certain pairs of vector group representations correspond to operator Kantor pairs.
	
	In \cref{se:okpairs}, we introduce operator Kantor pairs and prove the aforementioned results  on general operator Kantor pairs.
	Namely, we first introduce pre-Kantor pairs $(G^+,G^-)$, which are pairs of vector groups with homogeneous operators. For such a pair, we introduce a Lie algebra $L$, and we construct vector group representations to the endomorphism algebra of $L$ by making use of the operators.
	Secondly, we define operator Kantor pairs , which are pre-Kantor pairs satisfying some additional equations on those operators. We prove that, for operator Kantor pairs, the constructed vector group representations induce maps $G^\pm \longrightarrow \text{Aut}(L)$. Finally, in Theorem \ref{thm weights lie algebra}, we prove that the group $G \le \text{Aut}(L)$ generated by $G^+$ and $G^-$ is similar to the projective elementary group of quadratic Jordan pairs if $G^+$ and $G^-$ are projective.
	
	In \cref{se:lie}, we prove that there exists a precise link between Kantor pairs, pre-Kantor pairs and operator Kantor pairs. Namely, we show that these three classes are in a one to one correspondence if $1/30 \in \Phi$. Finally, in \cref{se:struct}, we use the different classes of central simple structurable algebras to define operator Kantor pairs. At the end of that section, we propose a definition of structurable algebras in our setting.
	
	\bigskip
	
	Throughout the paper, $\Phi$ denotes the base ring over which we work.
	Unless stated otherwise, $\Phi$~is an arbitrary commutative unital ring.
	The category $\Phi\textbf{-alg}$ is the category of associative, unital, commutative $\Phi$-algebras. 
	All $\Phi$-modules are assumed to be left modules unless stated explicitly otherwise (which will occur only in paragraph \ref{sss:herm}).
	
	\section{Vector groups}\label{se:vectorgroups}
	
	In this section, we introduce vector groups. These can be thought of as groups endowed with a scalar multiplication over $\Phi$, similar to modules. In order to make their likeness precise, we also need the $\Phi$-group functor (i.e., a functor\footnote{Such a group valued functor can also be described as a set valued functor endowed with certain natural transformations corresponding to the group operation, the unit, and taking inverses. A more precise formulation of this can be found in \cite[Section 1.4]{Wat79}.} $\Phi\textbf{-alg} \longrightarrow \mathbf{Grp}$) mapping $K \mapsto M \otimes_\Phi K$ for modules $M$. Namely, this functor can be endowed with a scalar multiplication $K \times M(K) 
	\longrightarrow M(K)$ which is natural in $K$. Vector groups are groups $G$ with a similar construction plan for a functor $K \mapsto G(K)$ for which there exists a natural transformation $K \times G(K) \longrightarrow G(K)$ coming from the scalar multiplication.
	
	In the first subsection, we introduce vector groups and construct functors from them. We also delineate the class of proper vector groups, which are the only ones we work with in the rest of the article.
	We prove that the Lie algebra of a proper vector group is preserved under scalar extensions. In the second subsection, we introduce homogeneous maps on proper vector groups and representations of proper vector groups, and introduce the corresponding universal objects. In the last subsection, we determine the primitive elements of the universal representation for projective vector groups and recover each projective vector group from its universal representation.
	
	\subsection{Definition and basic properties}
	
	Consider $\Phi$-modules $A$ and $B$ and a bilinear form $$\psi \colon A \times A \longrightarrow B.$$ We endow $A \times B$ with a group structure $$(a,b)(c,d) = (a + c , b + d +\psi(a,c)),$$
	a first type of scalar multiplication
	$$ \lambda \cdot_1 (a,b) = (\lambda a , \lambda^2 b),$$
	and a second one 
	$$ \lambda \cdot_2 (0,b) = (0, \lambda b),$$
	where the second scalar multiplication is only defined on the subgroup $0 \times B$. Since $\psi$ is a bilinear form, we see that $\lambda \cdot_1$ is an endomorphism of $A \times B$.
	
	\begin{definition}
		\label{definition vecgroup}
		Consider a subgroup $G$ of  $A \times B$. 
		We call $G$ an \textit{almost vector group} if the following two conditions hold.
		\begin{enumerate}
			\item $G$ is closed under $g \mapsto \lambda \cdot_1 g$ for all $\lambda \in \Phi$;
			\item $G_2 = G \cap 0 \times B$ is closed under $g \mapsto \lambda \cdot_2 g$ for all $\lambda \in \Phi$.
		\end{enumerate}
		Let $\hat{G}(K)$ be the smallest subgroup of $(A \times B) \otimes K$ containing all $g \otimes 1$ for $g \in G$ and closed under both $(g \mapsto \lambda \cdot_i g)$ for all $\lambda \in K$.
		If now, in addition,
		\begin{enumerate}
			\item[3.] $(G_2 \otimes_\Phi K) \longrightarrow \hat{G}(K)_2$ is a surjection\footnote{The tensor product is a tensor product of modules.} for all $\Phi$-algebras $K$ whenever $1/2 \notin \Phi$,
		\end{enumerate}
		holds, then we call $G$ a \textit{vector group}.
		In particular, if $1/2 \in \Phi$, then by definition, each almost vector group is a vector group.
	\end{definition}

	\begin{remark}
		An example where the third condition is not satisfied, is the group $G \subset \mathbb{Z}/4\mathbb{Z} \times \mathbb{Z}/2\mathbb{Z}$ over $\mathbb{Z}/4\mathbb{Z}$ with $\psi(1,1) = 1$, formed by the elements $(0,0)$ and $(2,1)$. Namely, if we set $K = \mathbb{Z}/2\mathbb{Z}$ we get $\hat{G}(K) = \{(0,0),(0,1)\}$, so that indeed $\hat{G}(K)_2$ is bigger than $G_2 \otimes K$.
		
		If $1/2 \in \Phi$, one can show that the condition stated in item 3 is automatically satisfied.
	\end{remark}

	\begin{lemma}
		\label{lemma G(K)}
		Let $G$ be an almost vector group.
		The group $\hat{G}(K)$ is generated as a subgroup of $(A \times B) \otimes K$ by the sets of generators $K \cdot_1 (g \otimes 1)$ for $g \in G$ and $K \cdot_2 (G_2 \otimes 1)$.
		\begin{proof}
			By definition, both  $K \cdot_1 (g \otimes 1)$ for $g \in G$ and $K \cdot_2 (G_2 \otimes 1)$ should be contained in $\hat{G}(K)$. So, we need to prove that the group generated by those sets is closed under scalar multiplications.
			This is the case since scalar multiplications are endomorphisms of $\hat{G}(K)$ such that
			$$ \lambda \cdot_i (\mu \cdot_i g) = (\lambda\mu \cdot_i g),$$
			for $i = 1,2$ and 
			$$ \lambda_i \cdot_i (\lambda_j \cdot_j g) = \lambda_i^j \lambda^i_j \cdot_2 g,$$
			for $i,j$ such that $\{i,j\} = \{1,2\}.$
		\end{proof}
	\end{lemma}
	
	From the previous lemma, we can conclude immediately that $K \longmapsto \hat{G}(K)$ is a functor $$\Phi\textbf{-alg} \longrightarrow \textbf{Grp}$$ and thus a $\Phi$-group.
	Consider the ring of \textit{dual numbers} $\Phi[\epsilon]$, i.e., the ring obtained by adding a generator $\epsilon$ satisfying $\epsilon^2 = 0$, and consider the projection $\pi : \Phi[\epsilon] \longrightarrow \Phi$.
	We use $\text{Lie}(G)$ to denote the Lie algebra associated to $G$: $$ \text{Lie}(G) = \ker \hat{G}(\pi) = \{(\epsilon a, \epsilon b) \in \hat{G}(\Phi[\epsilon])\}$$ with as bracket the restriction of the commutator on $A \times B$, or equivalently
	$$ [(a,b),(c,d)] = (0,\psi(a,c) - \psi(c,a)).$$
	This is the Lie algebra one typically associates to a linear algebraic group, cfr. \cite[the definition preceding Example 3.1]{Milne} (even though the Lie bracket is typically introduced in a different fashion).
	
	Now, we prove that each almost vector group can be seen as a quotient of an actual vector group.	
	
	\begin{lemma}
		\label{lem: almost vector group lifts}
		Suppose that $G \le A \times B$ is an almost vector group and that there exists $(a,b_a) \in G$ for each $a \in A$. Then $G$ is a vector group. Moreover, using the projection $\pi_A : G \longrightarrow A$, each almost vector group $G \le A \times B$ can be lifted to a vector group $G' \le \pi_A(G) \times B$ such that $\hat{G'}(K) \longrightarrow \hat{G}(K)$ is an isomorphism for all flat $K$.
		\begin{proof}
			We first remark that $q : a \mapsto b_a \in B/G_2$ is a quadratic map with linearisation $q(x,y) = \psi(x,y)$. So, if we look at $\hat{G}(K)$ we have elements $(a,b_a) \in (A \times B) \otimes K$ and we note that $q_K(a) = b_a \mod (G_2 \otimes K)$ since it holds for all generators of $\hat{G}(K)$ (where we use $(G_2 \otimes K)$ to denote its image in $B \otimes K$). Therefore, we see that $(0,b) \in \hat{G}(K)_2$ means that $b = 0 \mod (G_2 \otimes K)$, which is exactly what we had to prove.
			
			For the moreover part, we observe that $\pi_A(G)$ is a submodule of $A$ since it is closed under addition and scalar multiplication. For flat $K$ we note that
			\[ \pi_A(G) \otimes K \longrightarrow A \otimes K\]
			is an embedding.
		\end{proof}
	\end{lemma}

	In what follows, will just write $$(\lambda g_1,\lambda^2 g_2) = \lambda \cdot_1 (g_1,g_2),$$ etc., for elements of $\hat{G}(K)$ instead of using tensor products.
	
	\begin{lemma}
		\label{lemma 2g2 - g1sq}
		Let $G$ be an almost vector group. If $(g_1,g_2)$ is an element of $G$, then also $s_g = (0,2g_2 - \psi(g_1,g_1))$ is an element of $G$.
		Furthermore, for $K \in \Phi\textbf{-alg}$, $\lambda, \mu \in K$ and $g \in \hat{G}(K)$, we have 
		$$ (\lambda + \mu) \cdot_1 g = (\lambda \cdot_1 g)(\mu \cdot_1 g)(\lambda \mu \cdot_2 s_g).$$
		\begin{proof}
			We compute for $g = (g_1,g_2) \in \hat{G}(K)$ that
			$$ \hat{G} \ni g(-1 \cdot_1 g) = (0,2g_2 - \psi(g_1,g_1)) = s_g.$$
			For the second claim, we compute
			\begin{align*}
				(\lambda \cdot_1 g) (\mu \cdot_1 g)(\lambda \mu \cdot_2 s_g) 
				& = (\lambda g_1, \lambda^2 g_2) (\mu g_1, \mu^2 g_2)  (0, 2 \lambda \mu g_2 - \lambda \mu \psi(g_1,g_1)) \\ &= ((\lambda + \mu)g_1,(\lambda^2 + \mu^2)g_2 + \lambda\mu \psi(g_1,g_1)) (0, 2 \lambda \mu g_2 - \lambda \mu \psi(g_1,g_1)) \\
				& = ((\lambda + \mu)g_1, (\lambda^2 + \mu^2 + 2 \lambda\mu) g_2) \\
				& = (\lambda + \mu) \cdot_1 (g_1,g_2). \qedhere
			\end{align*}
		\end{proof}
	\end{lemma}
	
	\begin{remark}
		We will often write $-g$ for $(-1) \cdot_1 g$.
		In particular, the element $s_g$ of the previous lemma becomes $g(-g)$.
	\end{remark}
	
	\begin{lemma}
		\label{lemma Lie G}
		Let $G$ be a vector group. The underlying $\Phi$-module of $\Lie(G)$ is isomorphic to $$L = G/G_2 \oplus G_2,$$
		which we can identify with a submodule of $A \times B$ using $(a,b)G_2 \mapsto (a,0)$.
		\begin{proof}
			We note that $G/G_2 \subset A$ is an actual module since it is closed under the group operation, and thus addition, and since it is closed under $\cdot_1$, and thus scalar multiplication.
			
			We prove that $\hat{G}(\Phi[\epsilon])$ is a semi-direct product $G \ltimes (\epsilon L)$.
			This immediately proves that $\text{Lie}(G) \cong L$.
			First, note that
			$$ \epsilon L \subset \hat{G}(\Phi[\epsilon]),$$
			since $$\epsilon L = (\epsilon \cdot_1 G) + (\epsilon \cdot_2 G_2).$$
			We see that $\epsilon L$ is an abelian subgroup.
			Second, one can compute that $\epsilon L$ is stabilized by $G$ under the conjugation action.
			By applying Lemma \ref{lemma 2g2 - g1sq} on $(a + b \epsilon) \cdot_1 g$ we obtain that $\Phi[\epsilon] \cdot G \le G \ltimes \epsilon L$. Since $(a + b \epsilon)\cdot_1$ is an endomorphism, we see that $G \ltimes (\epsilon L)$ is closed under scalar multiplications.
		\end{proof}
	\end{lemma}
	
	\begin{lemma}
		Let $\rho_K : \hat{G}(K) \longrightarrow \hat{H}(K)$ be a natural transformation  between $\Phi$-group functors associated to vector groups such that 
		\[ \rho_K(\lambda \cdot_i g) = \lambda \cdot_i \rho_K(g)\]
		for all $g \in \hat{G}(K), \lambda \in K$.
		The groups $L = \ker \rho_\Phi$ and $I = \Ima \rho_\Phi$ are almost vector groups with $\hat{I}(K) = \Ima \rho_K$.
		Hence, both are always vector groups if $1/2 \in \Phi$.
		Moreover, $I$ is a vector group if and only if $\rho(G_2(K)) = I_2(K)$ and $L$ is a vector group if and only if $L_2 \otimes K \longrightarrow \hat{L}_2(K)$ is surjective.
		\begin{proof}
			
			It is trivial to see that $L$ and $I$ are closed under both scalar multiplications and thus almost vector groups. So, we conclude that $L$ and $I$ is a vector group if $1/2 \in \Phi$.
			Using the additional assumptions we immediately see that $I$ and $L$ are vector groups as well if $1/2 \notin \Phi$.
		\end{proof}
	\end{lemma}
	
	\begin{definition}
		We call a natural transformation $\rho : \hat{G} \longrightarrow \hat{H}$ between the $\Phi\textbf{-grp}$ functors associated to vector groups $G,H$, a \textit{vector group homomorphism} if 
		\[ \rho(\lambda \cdot_i g) = \lambda \cdot_i \rho(g),\]
		and $\rho^{-1}_K(\hat{H}_2(K)) \subseteq \hat{G}_2(K) \ker\rho_K$ for all $\Phi$-algebras $K$. If $1/2 \notin \Phi$ we also ask that $\ker \rho_\Phi$ is a vector group.
		We use $\textbf{VecGrp}$ to denote the category of vector groups over $\Phi$ with these morphisms. If there is any doubt about the base ring, we write $\Phi\textbf{-VecGrp}$.
	\end{definition}
	
	\begin{remark}
		All the conditions, except  $\rho^{-1}_K(\hat{H}_2(K)) \subseteq \hat{G}_2(K) \ker\rho_K$, are justified by the fact that $\ker \rho$ must be a vector group as well.
		This other condition translates into $\rho^{-1}(H_2) \subset G_2$ for injective maps, i.e., injective morphisms preserve the second component.
		The more general condition is obtained by requiring that this holds and
		that $G \longrightarrow H$ factors through some kind of vector group $L = G/\ker \rho$ with $L_2$ obtained by projecting $G_2$.
		Moreover, this condition guarantees that the image is a vector group if $1/2 \notin \Phi$.
		
		The condition on $\ker \rho_\Phi$ if $1/2 \notin \Phi$ can be thought of as asking that the lift $L$ of $\ker \Phi$ defined in Lemma \ref{lem: almost vector group lifts} has a homomorphism $\hat{L} \longrightarrow \hat{G}$ with image $\widehat{\ker \rho_\Phi}$.
		
		At this moment, we should be careful if we write $G/K$ for normal $K \le G$, since we did not prove that $G/K$ is a vector group for each normal vector subgroup $K$ in general. However, we can always consider it as a $\Phi$-group with well-defined scalar endomorphisms and a well-defined subgroup $(G/K)_2$. Later, we will see that $G/K$ is always a vector group whenever $G$ is proper (what it means to be proper is defined in Definition \ref{def: proper}) in Corollary \ref{cor: G(K) = Hat(H)(K)}.
	\end{remark}
	
	\begin{lemma}
		\label{Lemma reparam}
		Let $L_1 \oplus L_2$ be a $\mathbb{Z}$-graded Lie algebra over a ring $\Phi$ with $1/2 \in \Phi$. Then $L_1 \oplus L_2$ forms a vector group with operation $$(a,b)(c,d) = (a+c,b+d + [a,c]/2).$$
		Moreover, each vector group $G$ over $\Phi$ is isomorphic (as a group) to the vector group $L$ associated to $\Lie(G)$ by mapping 
		$$ (a,b) \longmapsto (a,b - \psi(a,a)/2).$$
		This mapping induces a $\Phi$-group epimorphism $\hat{L}(K) \longrightarrow \hat{G}(K)$. 
		Hence, the third axiom for vector groups will also hold if $1/2 \in \Phi$.
		\begin{proof}
			Note that $G = L_1 \oplus L_2$ so that all axioms for vector groups are trivially satisfied.
			
			Now, we prove that each vector group can be reparametrized in this fashion. Lemma \ref{lemma 2g2 - g1sq} proves that $2b - \psi(a,a) \in G_2 \subset \text{Lie}(G)$.
			So, we conclude that 
			$(a,b - \psi(a,a)/2) \in \text{Lie}(G).$ To observe the surjectivity of this map, take arbitrary $(a,b) \in G$ and note that the coset $(a,b) G_2$ maps to all $(a,s) \in \text{Lie}(G)$.
			On the other hand, the reparametrization is obviously injective.
			
			A direct computation shows that this reparametrization is a group homomorphism if one uses that $\psi(a,c) - \psi(c,a) = [a,c]$.
			Moreover, the scalar multiplications and $G_2$ are preserved, so that we have a vector group isomorphism.
			We note $L \otimes K \longrightarrow \hat{G}(K) : (a,b) \mapsto (a, \psi(a,a)/2 + b)$ is surjective.
		\end{proof}
	\end{lemma}
	
	\begin{remark}
		The fact that the operation in the previous lemma is given by
		$$(a,b)(c,d) = (a+c,b+d + [a,c]/2)$$
		shows that vector groups, if $1/2 \in \Phi$, correspond directly to groups of exponentials. Specifically, if we consider the Lie algebra
		$$ L = \Phi \zeta \oplus L_1 \oplus L_2,$$
		with $[\zeta ,l_i] = il_i$ for $l_i \in L_i$ and $[L_i,L_j] \subset L_{i + j}$ with $L_k = 0$ if $k > 2$, then the group 
		$$ \exp(L_1 \oplus L_2) \le \text{Aut}(L)$$ has the aforementioned operation.
	\end{remark}
	
	Let $G$ be a vector group over $\Phi$, then $\hat{G}(K)$ is a vector group over $K$. We write $L \longmapsto \hat{G}_K(L)$ to denote the $K\textbf{-Grp}$ functor associated to the vector group $\hat{G}(K)$.		
	
	\begin{lemma}
		\label{lemma G(K) is Lie(G)K}
		Let $G$ be a vector group and $K \in \Phi\textbf{-alg}$, then there exists a surjective homomorphism  $\Lie(G) \otimes K \longrightarrow \Lie(\hat{G}_K)$.
		\begin{proof}
			First, recall the surjection $G_2 \otimes K \longrightarrow \hat{G_2}(K)$ which is assumed in the definition of a vector group if $1/2 \notin \Phi$. We can also obtain such a surjection from Lemma \ref{Lemma reparam} if $1/2 \in \Phi$. Secondly, the possible first coordinates of elements of $G_K$ are $K$-linear combinations of possible first coordinates of $G$. We conclude that $\text{Lie}(G_K)$ is a homomorphic image of $\text{Lie}(G) \otimes K$ by Lemma \ref{lemma Lie G}.
		\end{proof}
	\end{lemma}	
	
	\begin{definition}
		\label{def free}
		We call a vector group $G$ \textit{free (resp. projective)} if $G_2$ and $G/G_2$ are free (resp. projective) as $\Phi$-modules.	
		Similarly, we say that $G$ is \textit{finitely generated} if $G/G_2$ and $G_2$ are finitely generated as $\Phi$-modules. 
	\end{definition}

Note that there exist non-canonical bijections between $G$ and $\text{Lie}(G)$, since $(a,b), (a,c)$ being elements of $G$ implies that $(a,b)(-a,c) = (0,b + c - \psi(a,a)) \in G_2$ so that for each possible first coordinate $a$ of $G$, the possible second coordinates stand in bijection with $\psi(a,a) - b + G_2 \cong G_2$.
So, Lemma \ref{lemma G(K) is Lie(G)K} shows that $\hat{G}$ can be identified with (a quotient of the) $\textbf{Set}$-functor defined by $$K \longmapsto \text{Lie}(G) \otimes K$$

If $G/G_2$ is projective, then this quotient is given by a natural transformation.
 Namely, since this module is projective, we have a generating set $(a_i)_{i \in I}$ of $G/G_2$ and dual generating set $(\alpha_i)_{i \in I}$ such that
$$ v = \sum_{i \in I} \alpha_i(v)a_i,$$
with finite $\alpha_i(v) \neq 0$. By using a total order on $I$ and taking $(a_i,b_i) \in G$ for each $i$, we obtain a natural transformation 
$$ \text{Lie}(G) \otimes K \longmapsto \hat{G}(K) : (v,s) \longmapsto \left(v, s + \sum_{i \in I} \alpha_i(v)^2b_i + \sum_{i < j} \alpha_i(v)\alpha_j(v) \psi(a_i,a_j)\right) \in \hat{G}(K).$$

If $G$ is finitely generated projective, then we know that $K \mapsto \Lie(G) \otimes K$ is an affine algebraic scheme. 
Specifically, set $V = \text{Lie}(G)$ then we know that $K \mapsto V \otimes K$ is an affine algebraic scheme underlying $G \mapsto \hat{G}(K)$. Later, in Lemma \ref{lem: proj is proper}, we shall also see that this is an affine algebraic group scheme.
To see that this is really affine and algebraic, use the module $V^*$ which is finitely generated and projective \cite[Chapter 2.6]{BouAlg1}, to consider the algebra $\text{Sym}(V^*)$ which is finitely generated as an algebra. 
Using that 
$$ \text{Hom}_{\Phi-\textbf{alg}}(\text{Sym}(V^*),K) \cong \text{Hom}_\Phi(V^*,K) \cong  V \otimes K,$$
for finitely generated projective modules $V$ \cite[Chapter 2.7, Corollary 4]{BouAlg1}, we obtain that $K \mapsto V \otimes K$ is an affine algebraic scheme.

	\begin{lemma}
		\label{Lemma quotient free}
		Let $S_1$ and $S_2$ be sets. There exists a free vector group $F(S_1,S_2)$ such that for every vector group $G$ and each pair of maps $f : S_1 \longrightarrow G$ and $g : S_2 \longrightarrow G_2$, there exists a unique $\Phi$-group homomorphism $\rho : \hat{F}(S_1,S_2) \longrightarrow \hat{G}$. Moreover, this is a vector group homomorphism if and only if the almost vector group $(\ker \rho_\Phi) \subset F(S_1,S_2)$ is a vector group and $\rho^{-1}_\Phi(G_2) = (\ker \rho_\Phi) F(S_1,S_2)_2$.
		\begin{proof}		First, fix a total order of $S_1$.
			Second, we define $T_2 = S_2 \cup \{[a,b] \in S_1 \times S_1 | a < b\} \cup \{ s(t) | t \in S_1\},$
			where we wrote $[a,b] \in S_1 \times S_1$ as they will be used to represent commutators, and use the symbol $s(t)$ to represent $t(-t)$ for $t \in S_1$.
			Fix a total order of $T_2$ as well.
			We will endow the functor
			$$ K \longrightarrow \hat{F}(S_1,S_2)(K) = \left\{ \prod_{s \in S_1} k_s \cdot s \prod_{t \in T_2} l_t \cdot t | \text{ only finitely many nonzero }  k_s,l_t \in K\right\}$$
			where the product is a symbolic product respecting the ordering of $S_1$ and $T_2$,
			with a vector group structure, and we wrote products as this will represent a product of group elements.
			
			If we can do that, we have a $\Phi$-group homomorphisms $\hat{F}(S_1,S_2) \longrightarrow \hat{G}$ for any pair of maps $f$ and $g$ described in the lemma.
			Note that this vector group, if it is a vector group, is free in the usual sense of Definition \ref{def free}, since the Lie algebra is freely generated by $S_1$ and $T_2$.
			
			We will define $\hat{F}(S_1,S_2)(K)$ as a vector group contained in the free $K$-module $F_1 \times F_2$ with $F_1$ free over $S_1$ and $F_2$ free over $U_2 = S_2 \cup (S_1 \times S_1) \cup t(S_1)$ (where we write $(s_1,s_2) \in S_1 \times S_2$ as usual for products of sets), endowed with the group operation
			$$ (s,t)(s',t') = (s + s',t + t' + (s,s')),$$
			where $(s,s')$ is the bilinear form defined from extending the identity map $$S_1 \times S_1 \longrightarrow S_1 \times S_1 \longrightarrow F(S_1 \times S_1)$$ bilinearly. In this module, we use $2t(l) - (l,l)$ to represent the element $s(l) \in T_2$ for $l \in S_1$, use $(a,b) - (b,a)$ to represent $[a,b] \in T_2$, and generators $s \in S_2$ are represented by themselves.
			
			The map which sends each element $g$ to its inverse $g^{-1}$ defined on $F_1 \times F_2$ can be projected to only act on the second component $F_2$; we obtain a map $i$ which acts as $l \mapsto -l, (a,b) \mapsto (b,a), t(k) \mapsto - t(k) + (k,k)$ for $a,b,k \in S_1, l \in S_2$. The submodule containing all $x$ in $F_2$ such that $x + i(x) = 0$ is precisely the free module generated by $T_2$, using the identifications made between generators.
			We also remark that the free module generated by $T_2$ is a direct summand of $F_2$, since we can extend $T_2$ with the set $\{(a,b) \in S_1 \times S_1 | a > b\} \cup \{ t(a) | a \in S_1\}$ to obtain a freely generating set of $F_2$. 
			We define
			$$ \hat{F}(S_1,S_2)(K) = \left\{ (a,b) \in (F_1 \times F_2) \otimes K | b + i(b) = (a,a)\right\}.$$
			Note that for each $s \in S_1$ the element $(ks_1,k^2 t(s_1))$ lies in $\hat{F}(S_1,S_2)(K)$ for $k \in K$, so that the underlying set functor $K \mapsto \hat{F}(S_1,S_2)(K)$ coincides with the earlier definition of $F$.
			Using the fact that $[a,b] = (a,b) - (b,a)$ and $s(a) = 2t(a) - (a,a)$, we see that the elements $[a,b] , s(a)$ of $F$ coincide with what they should be.
			
			Consider a vector group $G \le A \times B$. Given $f : S_1 \longrightarrow G$ and $g : S_2 \longrightarrow G_2$, we can define $\rho : F_1 \times F_2 \longrightarrow A \times B$. First write $f : S_1 \longrightarrow A \times B : s \mapsto (f_1(s), f_2(s))$. There exists a unique map $\tilde{f}_1 : F_1 \longrightarrow A$ extending $f_1$. There also exists a unique map $F_2 \longrightarrow B$ extending $g : S_2 \longrightarrow G_2$, $f_2 : S_1 \longrightarrow B$, and $\psi \circ (f_1 \times f_1) : S_1 \times S_1 \longrightarrow B$.
			This induces a homomorphism $\hat{F}(S_1,S_2) \longrightarrow \hat{G}$.
			
			Since both kernel and image are closed under scalar multiplications, the moreover statement follows as well.
		\end{proof}
	\end{lemma}

\begin{definition}
	\label{def: proper}
	We call a vector group $G$ \textit{proper} if there exist generating sets $S_1$ and $S_2$ of $G$ such that $\rho : \hat{F}(S_1,S_2) \longrightarrow \hat{G}$ constructed in Lemma \ref{Lemma quotient free} is a vector group homomorphism.
	This is the case if and only if the almost vector group $L = \ker \rho_\Phi$ is a vector group.
	
	So, we remark that all vector groups over rings containing $1/2$ are proper.
\end{definition}

	So, for a proper vector group, we can introduce a new functor $K \mapsto G(K)$ by making use of the previous lemma. Namely, we can construct a free vector group $F(S_1,S_2)$ from $G$ using elements $S_1 \subseteq G$ and $S_2 \subseteq G_2$ that generate the Lie algebra $\Lie(G)$ and note that there exists a vector group $L = \ker \rho_\Phi$ associated to the kernel $\rho : F(S_1,S_2) \longrightarrow G$.
	We define $G(K) = \widehat{F(S_1,S_2)}(K)/\hat{L}(K)$. In Lemma \ref{lem: sanity} we check that the defintion of $K \mapsto G(K)$ does not depend on the generating sets $S_1$ and $S_2$. 
	
	We remark that $G(K) \cong \hat{H}(K)$ for the image $H$ of $F(S_1,S_2)$ in $(F_1/L_1) \times (F_2/L_2)$ using the $F_1$, $F_2$ we introduced in Lemma \ref{Lemma quotient free}, while $L_1 \times L_2 = \ker F_1 \times F_2 \longrightarrow A \times B$.
	We remark that $H$ is obviously an almost vector group and thus a vector group by Lemma \ref{lem: almost vector group lifts}.

	We use $G_K$ to denote the functor $L \mapsto G(K \otimes_K L)$ for $K$-algebras $L$.
	
	\begin{lemma}
		\label{lem: Lie of G functorial}
		Let $G$ be a proper vector group, then $\Lie(G_K) \cong \Lie(G) \otimes K$.
		On the other hand, $G(K) \cong \hat{G}(K)$ whenever $\Lie(\hat{G}(K)) \cong \Lie(G) \otimes K$.
		Furthermore, $\Lie(G_K) \cong \Lie(G) \otimes K$ is a sufficient condition to be proper.
		\begin{proof}
			We use the notation introduced before this lemma.
			Take $g = (\epsilon f_1,\epsilon f_2) \hat{L}(K) \in G(K)$ and note that $g = 0$ if and only if $f_1$ lies in the image of $\Lie(L) \otimes K$ and $f_2$ lies in the image of $\Lie(L) \otimes K$.
			From the right exactness of the tensor product, we derive an exact sequence
			$\Lie(L)\otimes K \longrightarrow \Lie(F(S_1,S_2)) \otimes K \longrightarrow \Lie(G) \otimes K \longrightarrow 0$, which shows $\Lie(G)\otimes K \cong \Lie(G_K)$.
			
			Now, suppose that the Lie algebra of $\hat{G}(K)$ coincides with $\Lie(G) \otimes K$. 
			We remark that we have a natural transformation $G(K) \longrightarrow \hat{G}(K)$ by construction. The kernel of this map is always contained in $G_2(K)$ since $G(K)/G_2(K)$ and $\hat{G}(K)/\hat{G}_2(K)$ can be identified as parts of the Lie algebra. On the other hand, $G_2(K)$ and $\hat{G}_2(K)$ are also identified as parts of this Lie algebra.
			
			Lastly, for the sufficient condition, we note that $\Lie(G_K) \cong \Lie(G) \otimes K$ is sufficient for the kernel $\ker \rho_\Phi$ to be a vector group by the right exactness of the tensor product.
		\end{proof}
	\end{lemma}

Now, we do a little sanity check to see that we can vary the generating set employed in Lemma \ref{Lemma quotient free} without changing the functor $K \mapsto G(K)$.
	
\begin{lemma}
	\label{lem: sanity}
	The construction of $K \mapsto G(K)$ for a vector group does not depend on the generating sets $S_1$ and $S_2$ employed in Lemma \ref{Lemma quotient free}.
	\begin{proof}
		Suppose that $G$ is proper for a pair of generating sets $S_1$ and $S_2$ and consider the functor $G = \hat{H}$ constructed from it.
		Now, any pair of generating sets $U_1$ and $U_2$ induce a homomorphism $\theta : \hat{F}(U_1,U_2) \longrightarrow \hat{H}$ by using that $\hat{H}(\Phi) = G$. We conclude that $\ker \theta_\Phi$ is a vector group by Lemma \ref{lem: Lie of G functorial}.
		Hence, $\rho : \hat{F}(U_1,U_2) \longrightarrow \hat{G}$ is a vector group homomorphism as well.
	\end{proof}
\end{lemma}


%
%
%
	
	\begin{lemma}
		\label{lem: proj is proper}
		If $G$ is a projective vector group, then $G$ is proper. The functor $K \mapsto G(K)$ is an affine algebraic group for finitely generated projective $G$ with $G(K) \cong \Lie(G) \otimes K$ as set valued functors. 
		\begin{proof}
			Since $G_1 = G/G_2$ is projective and since the $F_1$ constructed in Lemma \ref{Lemma quotient free} is free, the epimorphism $\rho : F_1 \longrightarrow G_1$ shows that $\ker \rho_{|F_1} = K_1$ a direct summand of $F_1$. Now we see that $\widehat{\ker \rho_\Phi}(K) \le (K_1 \times F_2) \otimes K \subseteq (F_1 \times F_2) \otimes K$ so that $\ker \rho$ is a vector group by Lemma \ref{lem: almost vector group lifts}.
			
			We note that $K \mapsto \hat{F}(S_1,S_2)(K)$, with $\hat{F}(S_1,S_2)$ as in Lemma \ref{Lemma quotient free}, has coordinate algebra $A = \Phi[S_1,T_2]$ with $S_1$ a finite generating set of $G_1$ and $T_2$ a finite set defined in Lemma \ref{Lemma quotient free} from $S_1$ and a finite generating set $S_2$ of $G_2$. We remark that $\hat{H}$ has coordinate algebra $\mathcal{O}(\Lie(G))$.
			Applying the projection operator $A \longrightarrow \mathcal{O}(\Lie(G))$, yields a quotient of affine algebraic groups with $K \mapsto \hat{H}(K)$ as corresponding functor of points.
		\end{proof}
	\end{lemma}

	For the rest of this article, we are interested in proper vector groups and the associated functor $K \mapsto G(K)$ and not so much in the functor $K \mapsto \hat{G}(K)$.
	
	\subsection{Representations and homogeneous maps}
	
	In this section, we assume that all vector groups under consideration are proper.
	
	\begin{definition}Let $M$ be a module and $G$ a vector group.
		Identify $M$ with the functor $$\Phi\textbf{-alg} \longrightarrow \textbf{Set}: K \longmapsto M \otimes K.$$
		Similarly, identify $G$ with $G \longrightarrow G(K).$
		We shall define homogeneous maps of degree $n$ using recursion on the degree.	
		We call a natural transformation $f : G \longrightarrow  M$ a \textit{homogeneous map of degree $n$} if there exist linearisations $$f^{(i,j)} : G \times G \longrightarrow M,$$ for all $i,j \in \mathbb{N}_{>0}$ with $i + j = n$, which are homogeneous of degree $i$ in the first component, homogeneous of degree $j$ in the second component, such that
		$$ f_K((\lambda \cdot_a g)(\mu \cdot_b h)) = f_K(\lambda \cdot_a g) + f_K(\mu \cdot_b h) + \sum_{ai + bj = n} \lambda^{i}\mu^j f_K^{(ai,bj)}(g,h),$$
		for all $\lambda, \mu \in K$, in which we ask that 
		$$ f_K(\lambda \cdot_1 g) = \lambda^n f_K(g),$$
		and
		$$ f_K(\lambda \cdot_2 g) = \begin{cases}
			0 & n \text{ is odd }\\
			\lambda^{(n/2)} f_K(g) & \text{otherwise}
		\end{cases}.
		$$
		Remark that homogeneous maps of degree $0$ are constant maps\footnote{To prove this, use that $0^0$ means $0$ multiplied by itself $0$ times, i.e., $1$, in this context.}, and homogeneous maps of degree $1$ are maps that factor as linear maps through $G/G_2$.
	\end{definition}
	\begin{definition}Let $C$ be an associative unital algebra and $t$ a variable over which we consider formal power series. 
		We identify $(1 + tC[[t]])$ with the $\Phi\textbf{-Grp}$ functor
		$$K \mapsto 1 + t(C \otimes K)[[t]].$$
		Consider a sequence $(\rho_{i} : G \longrightarrow C)_i$ of natural transformations.
		Assume that these define a natural transformation of $\Phi$-groups $\rho_{[t]} : G \longrightarrow (1 + tC[[t]])$ using
		$$  \rho_{[t]}(g) = \sum \rho_i(g) t^i.$$ 
		Note that this satisfies, as it is a natural transformation of $\Phi$-groups, the equation
		$$ \rho_{[t]}(gh) = \rho_{[t]}(g)\rho_{[t]}(h)$$
		for all $g,h \in G(K)$.
		We call $(\rho_{i} : G \longrightarrow C)_i$ a \textit{vector group representation} if
		\begin{enumerate}
			\item $\rho_{[t]}(\lambda \cdot_1 g) = \rho_{[{\lambda t}]}(g)$, i.e., the first scalar multiplication corresponds to substituting $\lambda t$ for $t$ in the formal power series,
			\item $\rho_{[t]}(G_2) \le (1 + t^2C[[t^2]])$, so that we can take $$\sigma_{[t]} : G_2 \longrightarrow (1 + tC[[t]])$$ such that $\sigma_{[t^2]} = \rho_{[t]}$, 
			\item $\sigma_{[t]}(\lambda \cdot_2 g) = \sigma_{[\lambda t]}(g)$,
			\item $\rho_{[t]}(g) \in (1 + t^2C[[t^2]])$ means that 
			$$ g \in (\ker \rho) \cdot G_2,$$
			i.e., there exists $h \in G_2$ such that $gh \in \ker \rho$.
		\end{enumerate}	
	\end{definition}
	
	\begin{lemma}
		Suppose that $(\rho_i : G \longrightarrow C)_i$ is a vector group representation, then each $\rho_i$ is a homogeneous map of degree $i$.
		\begin{proof}
			This is obviously the case for $i = 0,1$.
			The rest follows easily by induction using that  \[ \rho^{(i,j)}_n(g,h) = \rho_i(g)\rho_j(h). \qedhere \]
		\end{proof}
	\end{lemma}
	
	\begin{lemma}
		\label{Lemma homogeneous map on commutator}
		Let $G$ be a vector group, and $M$ a module.
		Suppose that $f : G \longrightarrow M$ is a homogeneous map of degree $2$, then 
		$$ f^{(1,1)}(a,b) - f^{(1,1)}(b,a) = f([a,b]) = f(a^{-1}b^{-1}ab),$$
		holds for all $a,b \in G$.
		\begin{proof}
			Using the $(1,1)$-linearisation, one obtains 
			\begin{align*}
				f([a,b]) = &  f(a^{-1}) + f(b^{-1}) + f(a) + f(b) + f^{(1,1)}(a^{-1},a) + f^{(1,1)}(b^{-1},b) \\ &  + f^{(1,1)}(a^{-1},b^{-1}) + f^{(1,1)}(a^{-1},b) + f^{(1,1)}(a,b) + f^{(1,1)}(b^{-1},a).
			\end{align*}
			We use that $f(a^{-1}) + f^{(1,1)}(a^{-1},a) + f(a)  = f(0) = 0$, and that $f^{(1,1)}$ is linear if we interpret it as a map defined on $(G/G_2)^2$ so that $f^{(1,1)}(a,b^{-1}) = -f^{(1,1)}(a,b)$, etc., to obtain
			\[ f([a,b]) = f^{(1,1)}(a,b) - f^{(1,1)}(b,a). \qedhere \]
		\end{proof}
	\end{lemma}
	
	Note that $\rho_{[t]} : G \longrightarrow \Phi \oplus (A \times B)[[t]] : (g_1,g_2) \mapsto 1 + (tg_1,t^2g_2)$ can be seen as a vector group representation if we endow $\Phi \oplus (A \times B)$ with the multiplication $$(\lambda + (a,b))(\mu + (c,d)) = \lambda\mu + ((\mu a + \lambda c  ,\mu b + \lambda d + \psi(a,c)).$$
	
	\begin{remark}
		\label{Remark vecgroup repres}
		With a vector group representation $\rho_i : G \longrightarrow C$, one can associate the vector group $\Ima \rho \le C \times C$, where the latter vector group is associated to the bilinear form $(a,b) \longrightarrow ab$ and is formed by the elements $(\rho_1(g),\rho_2(g))$. A vector group representation induces a morphism\footnote{You can see this now as a $\Phi$-group homomorphism. Later if we see that $K \mapsto G(K)$ also comes from a vector group, we can see this as a vector group homomorphism} $$G \longrightarrow C \times C.$$
	\end{remark}

	\begin{lemma}
		Let $C$ be an associative unital algebra and $G$ a vector group.
		Consider a vector group representation $(\rho_i : G \longrightarrow C)_i$.
		For all $g,h \in G$ and $k \in G_2$ the following equations, in which we write
		$g_i$ for $\rho_i(g)$, hold:
		\begin{enumerate}
			\item $$\binom{i + j}{i}k_{2(i+j)} = k_{2i}k_{2j},$$
			\item $$\binom{i+j}{i} g_{i+j} = \sum_{\substack{a + c = i\\ b + c = j}} g_ag_b (g(-g))_{2c},$$
			\item $$g_i h_j = \sum_{\substack{a + c = i\\b + c = j}} h_bg_a [g,h]_c.$$
		\end{enumerate}
		\begin{proof}
			The last two equations are obtained by comparing the terms belonging to the scalar $\lambda^i\mu^{j}$ in
			\begin{align*}
				((\lambda + \mu) \cdot_1 g)_{i+j} & = ((\lambda \cdot_1 g)(\mu \cdot_1 g)(\lambda\mu \cdot_2 (g(-g)))_{i+j}, \\
				& \text{and} \\
				((\lambda \cdot_1 g)(\mu \cdot_1 h))_{i+j} & = ((\mu \cdot_1 h)(\lambda \cdot_1 g)(\lambda\mu \cdot_1 [g,h]))_{i+j},
			\end{align*}
			where the first of these equations must hold by Lemma \ref{lemma 2g2 - g1sq}, and the second follows from a direct computation.
			Similarly, one proves the first equation of this lemma.
		\end{proof}
	\end{lemma}
	
	\begin{definition}
		Let $G$ be a vector group.
		We say that $(\rho_i : G \longrightarrow U)_i$ is a \textit{universal vector group representation} if for each vector group representation $(\gamma_i : G \longrightarrow V)_i$, there exists a unique algebra homomorphism $f : U \longrightarrow V$ such that $f \circ \rho_i = \gamma_i$.
		Similarly, we call a homogeneous map $\rho : G \longrightarrow M$ of degree $n$ a \textit{universal homogeneous map of degree} $n$, if for each $\gamma : G \longrightarrow N$ of degree $n$ there exist a unique linear map $f : M \longrightarrow N$ such that $\gamma = f \circ \rho$.
	\end{definition}

	\begin{remark}
		Universal vector group representations and universal homogeneous maps are unique up to isomorphism.
	\end{remark}
	
	\begin{construction}
		\label{construction universal representation}
		We define $\mathcal{U}(G)$ as the unital associative algebra generated by symbols $g_i$ for $g \in G$ and $i \in \mathbb{N}$ and relations
		\begin{enumerate}
			\item $g_0 = 1$ for all $g$
			\item $g_{2i+1} = 0$ for all $g \in G_2$
			\item $\sum_{i + j = n} g_ih_j = (gh)_n$
			\item $(\lambda \cdot_1 g)_{j} = \lambda^j g_{j}$
			\item $(\lambda \cdot_2 g)_{2j} = \lambda^j g_{2j}$ for $g \in G_2$
			\item $\binom{i + j}{i}g_{2(i+j)} = g_{2i}g_{2j}$ for $g \in G_2$ \label{cons:g2 scalar}
			\item $\binom{i+j}{i} g_{i+j} = \sum_{\substack{a + c = i\\ b + c = j}} g_ag_b (g(-g))_{2c}$ \label{cons:g scalar}
			\item $g_i h_j = \sum_{\substack{a + c = i\\b + c = j}} h_bg_a [g,h]_c$
		\end{enumerate}
	\end{construction}
	
	Note that $\mathcal{U}(G)$ is an $\mathbb{N}$-graded algebra if we set $g_j$ to be $j$-graded. Observe, additionally, that all the relations imposed on $\mathcal{U}(G)$ are necessary to have a representation
	$$ \rho_i : G \longrightarrow \mathcal{U}(G): g \longmapsto g_i.$$
	So, if there exists a vector group representation to $\mathcal{U}(G)$, then it is obviously a universal one.
	We can write $g \mapsto \rho_{[t]}(g) = 1 + tg_1 + t^2g_2 + \ldots$, as if there exists a representation, for $g \in G(\Phi)$.
	
	We shall prove that $\mathcal{U}(G)$ is the universal representation for arbitrary $G$ by first proving it for free $G$.
	
	\begin{lemma}
		The algebra $\mathcal{U}(G)$ is an $\mathbb{N}$-graded Hopf algebra with operations
		\[ \Delta(g_n) = \sum_{i +j = n} g_i \otimes g_j, \quad \epsilon(g_i) = \delta_{i0}, \quad S(g_n) = (g^{-1})_n.\]
		\begin{proof}
			We note that all operations are compatible with the grading, so we only need to check whether we have a Hopf algebra.
			
			The comultiplication $\Delta$ is compatible with all the relations to which $\mathcal{U}(G)$ is subject as well.
			There are only $2$ relations that are not straightforwardly satisfied, as they involve binomial coefficients.
			We prove the compatibility with the comultiplication for the most difficult relation of them, namely
			$$ \binom{i+j}{i} g_{i+j} = \sum_{\substack{a + c = i\\ b + c = j}} g_ag_b (g(-g))_{2c}.$$
			We compute
			\begin{align*}
				\Delta \left(  \sum_{\substack{a + c = i\\ b + c = j}} g_ag_b (g(-g))_{2c} \right)			
				& = \sum_{\substack{a_1 + c_1 + a_2 + c_2 = i\\ b_1 + c_1 + b_2 + c_2 = j}} g_{a_1}g_{b_1} (g(-g))_{2c_1} \otimes g_{a_2}g_{b_2} (g(-g))_{2c_2}\\
				& = \sum_{\substack{k + l = i + j\\ m + n = i}} \binom{k}{m} g_{k} \otimes \binom{l}{n} g_{l} \\ 
				& = \binom{i+ j}{i}\sum_{k + l = i + j} g_{k} \otimes g_l,
			\end{align*}
			in which the second to last equality comes from applying the relation to the sum where we observe that the equations over which we sum are equivalent to 
			$$ \sum a_i + \sum b_i + 2 \sum c_i = i +j \quad \sum a_i + \sum c_i = i,$$
			and where the last equality corresponds to 
			$$  \sum_{\substack{ m + n = i}} \binom{k}{m} \binom{l}{n} = \binom{i + j}{i}$$
			for $k,l$ such that $k + l = i + j$.
			This equality between binomial coefficients corresponds to computing the term belonging to $a^ib^j$ in
			$$ (a + b)^k(a + b)^l = (a + b)^{i +j}$$
			over $\mathbb{Z}[a,b]$.
			All other relations are either trivial or depend on the same binomial formula.
			
			The antipode and counit are fine as well since $\rho_{[t]}(g)\rho_{[t]}(g^{-1}) = 1$ with $\rho_{[t]}(g) = \sum t^i g_i$.
		\end{proof}
	\end{lemma}

	\begin{lemma}
		\label{lemma free universal implies universal}
		If $\mathcal{U}(G)$ is the universal representation of $G$ for all free $G$ over $\mathbb{Z}$, then $\mathcal{U}(G)$ is the universal representation of $G$ for all (proper) $G$.
		\begin{proof}
			Suppose that $G$ is an arbitrary (proper) vector group. Write it as a quotient $H/K$ of a free vector group $H$ using Lemma \ref{Lemma quotient free} and note that $H = L_\Phi$ for some free vector group $L$ over $\mathbb{Z}$.
			Since $G$ is proper, we know that $K$ is a vector group as well.		
			We have a representation corresponding to $\rho_{i,\Phi} : L(\Phi) \longrightarrow \mathcal{U}(L) \otimes \Phi  \longrightarrow \mathcal{U}(G)$ of $L_\Phi$ which factors through $L_\Phi/K \cong G$ over all scalar extensions so that $g \mapsto g_i$ is a representation of $G$.
		\end{proof}
	\end{lemma}

	Now, we prove that $\mathcal{U}(G)$ is the universal representation for all free $G$ over $\mathbb{Z}$. To achieve that, we construct a similar algebra $\mathcal{X}(G)$ which will be isomorphic to $\mathcal{U}(G)$.
	
	\begin{construction}Suppose that $G$ is free, let $B_1$ be a set of free generators for $ G/G_2$ and let $B_2$ be a set of free generators for $G_2$. Take for each $b \in B_1$ an element $\hat{b} \in G$ so that $b = \hat{b}G_2$.
		We assume that $B_1$ and $B_2$ are totally ordered.
		Consider the unital associative algebra $F(G)$ with generators $b_i$ for $b \in B_1$ and $i \ge 1$ and $b_{2i}$ for $b \in B_2$ and $i \ge 1$. We use $b_0$ to denote $1$ for all $b \in B_1 \cup B_2$.
		For $g = \sum_{j = 1}^k \lambda_j \cdot_2 b_{j} \in G_2$ with $b_i < b_j$ for $i < j$, we write
		$$g_{2n} = \sum_{i_1 + \ldots + i_k = n} \prod_{j = 1}^k \lambda^{i_j}(b_j)_{2i_j}.$$
		
		We impose the following relations for $b < c \in B_1$, $d < e \in B_2$ on $F$ to obtain $\mathcal{X}(G)$.
		\begin{enumerate}
			\item $b_ib_j = \sum_{k + 2l = i + j} \binom{k}{i - l} b_k (-1 \cdot_2 (-b)b)_{2l}$.
			\item $d_{2i}d_{2j} = \binom{i + j}{i} d_{2(i+j)}$
			\item $c_jb_i = \sum_{\substack{k + m = i\\l + m = i}} b_k c_l [\hat{c},\hat{b}]_{2m}$
			\item $d_{2j}b_i = b_id_{2j}$
			\item $e_{2i}d_{2j} = d_{2j}e_{2i}.$
		\end{enumerate}
	\end{construction}
	
	All of these relations hold in $\mathcal{U}(G)$. This is only non-trivially the case for the first equation, which can be restated as
	$$ \rho_{[t]}(b)\rho_{[s]}(b) = \rho_{[t + s]}(b)\sigma_{[ts]}(-1 \cdot_2 (-b)b),$$
	where $\sigma_{[t^2]} = \rho_{[t]}$ on $G_2$. This relation holds in the algebra $\mathcal{U}(G)$ since we have imposed relations that ensure that
	$$  \rho_{[t]}(\lambda \cdot_1 b)\rho_{[t]}(\mu \cdot_1 b)\sigma_{[ t^2]}(\lambda \mu \cdot_2 (-b)b) = \rho_{[t]}((\lambda + \mu) \cdot_1 b),$$
	and 
	$$ \rho_{[t]}(g)\rho_{[t]}(h) = \rho_{[t]}(gh)$$
	hold in $\mathcal{U}(G)[\lambda,\mu][[t]]$.
	So, we know that there is a unique morphism $\mathcal{X}(G) \longrightarrow \mathcal{U}(G)$ which maps on generators as $b_i \mapsto b_i$ for $b \in B_1 \cup B_2$. 
	Furthermore, note that if there exists a vector group representation  $G$ in $\mathcal{X}(G)$ which maps the generators $b \in B_1 \cup B_2$ as expected, then we obtain $\mathcal{U}(G) \longrightarrow \mathcal{X}(G)$ mapping $b_i \mapsto b_i$. This is the case as each vector group representation $(\rho_i: G \longrightarrow A)_i$ induces a map $\mathcal{U}(G) \longrightarrow A$, since all relations imposed to obtain $\mathcal{U}(G)$ must hold in $A$ if one writes $g_i$ for $\rho_i(g)$.
	This will prove that $\mathcal{X}(G) \cong \mathcal{U}(G)$ and that both are universal representations.
	
	We will prove that $\mathcal{X}(G)$ is a universal representation by proving that the elements, in which the order of the products respects the orders we fixed for $B_1$ and $B_2$,
	$$ B_{f,g} = \prod_{b \in B_1} b_{f(b)} \prod_{b \in B_2} b_{g(b)},$$
	with $f: B_1 \longrightarrow \mathbb{N}$, $g : B_2 \longrightarrow \mathbb{N}$ so that $\sup(f) = \{ b \in B_1 | f(b) \neq 0\}$ and $\sup(g) =\{ b \in B_2 | g(b) \neq 0\}$ are finite sets, form a basis of $\mathcal{X}(G)$  utilizing the universal enveloping algebra of Lie algebras over fields of characteristic $0$.
	
	\begin{lemma}
		The module 
		$$ M = \langle B_{f,g} \rangle \subset F,$$
		is precisely the submodule of $F$ generated by monomials containing no expressions that are a left-hand side of a relation imposed on $F$ to obtain $\mathcal{X}(G)$. Furthermore, any element of $F$ is equivalent to an element of $M$ after applying a finite amount of relations.
		\begin{proof}
			We note that the elements $B_{f,g}$ are precisely the generating monomials of $F$ such that (1) no $b \in B_1 \cup B_2$ occurs multiple times as a $b_i$ in $B_{f,g}$ and (2) there are no $b,c \in B_1 \cup B_2$ which occur as $b_i, c_j$ which are not the same order as the order on $B_1 \cup B_2$, which is the order of $B_1$ and $B_2$ extended by $B_1 < B_2$. We note that the first two left-hand sides of relations correspond exactly with possible violations of (1) and that the last three left-hand correspond to possible violations of (2). This proves the first part of the lemma.
			
			Now, we prove the furthermore-part.
			First, we will associate a pair of natural numbers $(a,b)$ to a generating monomial which counts in a certain way to what degree the monomial violates (1) and (2).
			Second, we will prove that applying a relation (in the context of the lemma that is substituting a right-hand side for a left-hand side of a relation) to such a monomial creates a sum of monomials associated to pairs $(c,d) < (a,b)$ under the lexicographic order on $\mathbb{N}^2$.
			This will allow us to conclude that the furthermore part holds as those pairs $(a,b)$ are well-ordered, i.e., there cannot exist an infinite decreasing sequence.
			
			With a monomial $m = m_1\ldots m_n$ we associate the pair $(a,b)$ where $a$ counts the degree to which (1) and (2) are violated with at least one $b \in B_1$: $$a = |\{(i,j) | 1 \le i \le j \le n, m_i \text{ corresponds to } b \in B_1, m_j \text{ corresponds to } c < b \}|,$$
			while $b$ is the total degree in which (1) and (2) are violated:
			$$ b = |\{(i,j) | 1 \le i \le j \le n, m_i \text{ corresponds to } b, m_j \text{ corresponds to } c < b \}|.$$
			
			Observe that applying relations $1,2,4,5$ decreases $a + b$ while preserving $a$ or $b$. Applying relation $3$ decreases $a$. So, $(a,b)$ always decreases.
			Using the earlier remarked fact that the pairs $(a,b)$ are well-ordered, we can conclude that the lemma holds.
		\end{proof}
	\end{lemma}
	
	\begin{lemma}
		\label{lemma universal representation free over Z}
		Suppose that $G$ is a free vector group over $\mathbb{Z}$, the algebras $\mathcal{X}(G)$ and $\mathcal{U}(G)$ are isomorphic and are universal representations of $G$.
		\begin{proof}
			Consider the universal enveloping algebra $U$ of $\text{Lie}(G) \otimes \mathbb{Q}$. We know that $$\exp_{[t]} : \text{Lie}(G) \otimes \mathbb{Q} \longrightarrow U : (g,h) \mapsto \exp(tg,t^2h) = 1 + tg + t^2 (g^2/2 + h) + \ldots $$ is a vector group representation of $\text{Lie}(G) \otimes \mathbb{Q} \cong G(\mathbb{Q})$. This induces a vector group representation $$(\rho_i: G \longrightarrow U)_i.$$
			
			This means that there exists a map $$\mathcal{X}(G) \longrightarrow \mathcal{U}(G) \longrightarrow U,$$ using the earlier described map $\mathcal{X}(G) \longrightarrow \mathcal{U}(G)$ and the fact that there exists for each vector group representation a map with domain $\mathcal{U}(G)$ corresponding to that representation. Using the Poincarré-Birkhoff-Witt basis, we can conclude that this map is injective since $M$ embeds into $U$ and since each element of $\mathcal{X}(G)$ is contained in $M$ modulo the relations.
			Specifically, a generating monomial $B_{f,g} \in M$ is associated to a pair of functions $(f,g)$.
			Using the lexicographical order on $\mathbb{N}^{B_1 \cup B_2}_{\text{fin sup}}$, we can order these $B_{f,g}$.
			Similarly, one orders the (slightly modified) Poincarré-Birkhoff-Witt basis elements $$B'_{f,g} = \prod_{b \in B_1} b^{f(b)}/(f(b)!) \prod_{b \in B_2} b^{g(b)}/(g(b)!).$$
			Since $b < c$ for $b \in B_1,c \in B_2$, one can check that $B_{f,g} \mapsto B'_{f,g} + \sum_{(f,g) < (k,l)} c_{f,g,k,l} B'_{k,l}$ for certain $c_{f,g,k,l}$ using the relations involving binomial coefficients on $\mathcal{X}(G)$, and $b_i \mapsto b_i$ for $b_i \in B_i$ so that
			$$ (b_i)_{ni} \mapsto (b_i)^{n}/(n!) \mod \text{ ($B'_{f,g}$ with bigger $(f,g)$ than $(b_i)_{n_i}$)}$$
			for $b \in B_i$.
			The representation $(\rho_i : G \longrightarrow U)_i$ maps to the embedding of $\mathcal{X}(G)$ in $U$.
			Hence, $$(\rho_i: G \longrightarrow \mathcal{X}(G))_i$$ is a vector group representation. 
			As exposed earlier, this proves that $\mathcal{X}(G) \cong \mathcal{U}(G)$ is the universal representation. Specifically, we have a unique map from $\mathcal{U}(G)$ corresponding to the vector group representation $(\rho_i)_i$,  we have a map $\mathcal{X}(G) \longrightarrow \mathcal{U}(G)$ corresponding to the fact that all relations on $\mathcal{X}(G)$ also hold on $\mathcal{U}(G)$, and these maps interact nicely as they send $b_i \mapsto b_i$ for $b \in B_1 \cup B_2, i \in \mathbb{N}$.
		\end{proof}
	\end{lemma}
	
	\begin{theorem}
		Let $G$ be a proper vector group. There exists a universal vector group representation $$(\gamma_i : G \longrightarrow \mathcal{U}(G))_i.$$
		\begin{proof}
			For free vector groups $G$ over $\mathbb{Z}$, this is Lemma \ref{lemma universal representation free over Z}.
			For arbitrary $G$, this now follows from Lemma \ref{lemma free universal implies universal}.
		\end{proof}
	\end{theorem}
	
	\begin{corollary}
		\label{cor: G(K) = Hat(H)(K)}
		The category of proper vector groups is closed under quotients, i.e., suppose that $K \le G$ are vector groups with $K$ normal in $G$, then $L \longrightarrow G(L)/K(L)$ is a vector group. Moreover, if $G$ is proper and $K \subset G$ is normal but not necessarily proper, then $(G/K)$ is a proper vector group as well.
		\begin{proof}
			Consider $A = \mathcal{U}(G)/I$ with $I$ the ideal generated by $\mathcal{U}(K) \cap \ker \epsilon$ with $\epsilon$ the counit of the Hopf algebra. The universal representation $(\gamma_i: G \longrightarrow \mathcal{U}(G))_i$ induces a representation $(\rho_i: G/K \longrightarrow A)$. The representation $(\rho_i)_i$ is injective over all scalar extensions since it is injective on 
			$ \text{Lie}(G)/\text{Lie}(K)$ (since $K$ is normal).
			
			We remark that $K$ being proper does not matter, since $L \mapsto K(L)$ has a representation in the image of $\mathcal{U}(K) \otimes L$ contained in $\mathcal{U}(G) \otimes L$.
			
			So, this means that $G/K$ can be given a vector group structure in $A \times A$ as in Remark \ref{Remark vecgroup repres}.
		\end{proof}
	\end{corollary}
	
	\begin{theorem}
		\label{theorem universal homogeneous}
		Let $G$ be a proper vector group.
		Suppose that $\rho : G \longrightarrow M$ is a homogeneous map of degree $n$, then it factors uniquely through the $n$-th grading component of $\mathcal{U}(G)$.
		Hence, the mapping $g \longrightarrow g_n$ to $\mathcal{U}(G)_n$ defines the universal homogeneous map of degree $n$.
		\begin{proof}
			We proceed by induction.
			For $n = 0$, this is trivial as every homogeneous map $f$ of degree $0$ is constant, so it uniquely factors through $\Phi = \mathcal{U}(G)_0$, mapping $1$ to $f(G)$.
			
			Let $f : G \longrightarrow M$ be a homogeneous map of degree $n$.
			We prove by induction on $n$ that there exists a representation	
			$$ \rho_{[t]}(g) = 1 + tg_1 + \ldots t^{n-1}g_{n-1} + t^nf(g)$$	
			of $G$ in an algebra $A = \mathcal{U}(G)_{<n} \oplus M.$
			From this, we can obtain an algebra map $$\mathcal{U}(G) \longrightarrow A$$ using the universality of $\mathcal{U}(G)$ as a vector group representation. The restriction $\mathcal{U}(G)_n \longrightarrow M$ will map $g_n \mapsto f(g)$ which will prove the theorem.
			
			So, consider a homogeneous map $f : G \longrightarrow M$ of degree $n$ with linearisations $f^{i,j} : G \times G \longrightarrow M$, which inductively correspond to linear maps
			$g^{i,j}: \mathcal{U}(G)_i \otimes \mathcal{U}(G)_j \longrightarrow M.$
			These maps are associative in the following sense $g^{(i,j),k}(x,y,z) = g^{i,(j,k)}(x,y,z)$ where both expressions are $(i,j,k)$ linearisations of $f$ obtained by linearizing either $g^{i+j,k}$ or $g^{i,j+k}$, since both correspond to a certain term belonging to the coefficient $a^ib^jc^k$ in
			$$ f((a \cdot_1 x)(b \cdot_1 y)(c \cdot_1 z)).$$
			We endow $A$ with the restriction of the product of $\mathcal{U}(G)$ to $\mathcal{U}{(G)}_{<n}$ (where we throw away the result if it has too high a grading), and with the obvious left and right action of $\Phi$ on $M$, while using 
			$g^{i,j}$ for the rest of the products $\mathcal{U}(G)_i \otimes \mathcal{U}(G)_j \longrightarrow M$. Note that $A$ algebra is associative because of the associativity of the linearisations.
			
			It is easy to see that 
			$$ \rho_{[t]}(g) = 1 + tg_1 + \ldots + t^{n-1}g_{n - 1} + t^n f(g),$$
			forms a vector group representation.
			As argued earlier, this proves that $g \longmapsto g_n$ is the universal homogeneous map of degree $n$. 
		\end{proof}
	\end{theorem}
	
	\begin{remark}The following equality for homogeneous maps $f$ degree $3$ will be useful:
		\begin{equation}
			\label{equation homogeneous of degree 3}
			3(f(g) - f^{(1,2)}({g,g})) + f^{(1,1,1)}({g,g,g}) = 0.
		\end{equation}
		This equality can be obtained from the equality
		$$ 3g_3 - 3g_1g_2 + g_1^3 = 3g_3 - 2g_2g_1 + g_1(g^{-1})_2 = 0$$
		which holds in the universal representation. 
		This can either be proved using the relations on $\mathcal{U}(G)$, or one can use the following computation:
		$$ \epsilon(3g_3 - 2g_2g_1 + g_1(g^{-1})_2) = (((1 + \epsilon) \cdot_1 g)g^{-1})_3 = (\epsilon g_1, 2 \epsilon g_2 - \epsilon \psi(g_1,g_1))_3 = 0$$ 
		over the dual numbers $\Phi[\epsilon]$.
	\end{remark}
	
	\subsection{The primitive elements of the universal representation}
	
	In this section, we are interested in projective vector groups.
	By Lemma \ref{lem: proj is proper}, we know that projective vector groups are proper.

	\begin{definition}	Let $H$ be a Hopf algebra, we use $P(H)$ to denote the set of \textit{primitive} elements, i.e., the $x \in H$ such that 
		$$ \Delta(x) = x \otimes 1 + 1 \otimes x.$$
		We call an element $x \in H$ \textit{group-like} if $\Delta(x) = x \otimes x$ and $\epsilon(x) = 1$.
	\end{definition}

	We want to prove that $P(\mathcal{U}(G)) = \text{Lie}(G)$ for projective $G$ and be able to recover $G(K)$ from $\mathcal{U}(G) \otimes K$.
	In order to achieve that, we will consider concrete filtrations of $\mathcal{U}(G)$ which will coincide for projective vector groups $G$. These filtrations are chosen in such a way that we know the primitives of $\mathcal{U}(G)$ coincide with $\text{Lie}(G)$ if the first $2$ parts of the filtration are the same. Although we will not explicitly prove it, the question of whether these two filtrations coincide only depends on the module $\text{Lie}(G)$.
	
	\begin{definition}
		Consider the space
		\[ Y_m = \left\langle  \prod_{i = 1}^{k_1} (v_{1i})_{m_{1i}} \prod_{j = 1}^{k_2} (v_{2j})_{2m_{2j}}| v_{1i} \in G, v_{2j} \in G_2, \sum_{i=1}^{k_1} m_{1,i} + \sum_{j = 2}^{k_2} v_{2,j} \le m, k_i \in \mathbb{N} \right\rangle.\]
		We also define
		\[ Z_m =  \ker (\text{Id} - \epsilon)^{\otimes m + 1} \Delta^{m},\]
		writing $\text{Id} - \epsilon$ for $\text{Id}- \eta \circ \epsilon$ with $\eta : \Phi \longrightarrow \mathcal{U}(G)$ the structure morphism and using the operator  $\Delta^{m+1} = (\Delta^{m} \otimes \text{Id}) \circ \Delta$ with $\Delta^1 = \Delta$. 
		The coassociativity guarantees that $\Delta^{m+1} = (\Delta^{i} \otimes \Delta^{j}) \circ \Delta$ for all $i + j = m$.
		We will call vector groups $G$ with universal representation $\mathcal{U}(G)$ \textit{well-behaved} if $Y_m = Z_m$ for all $m$.
	\end{definition}
	
	\begin{remark}
		By construction, the spaces $Y_i$ satisfy
		\[ \Phi = Y_0 \subsetneq Y_1 \subsetneq Y_2 \subsetneq \dots \]
		and $\bigcup_m Y_m = \mathcal{U}(G)$.
	\end{remark}
	
	\begin{lemma}
		The spaces $Z_i$ satisfy $Z_i \subset Z_{i + 1}$ for all $i$. 
		\begin{proof}
			We must prove that
			\[ \ker (\text{Id} - \epsilon)^{\otimes m + 1} \Delta^{m} \subset \ker (\text{Id} - \epsilon)^{\otimes m + 2} \Delta^{m + 1}.\]
			For $p \in \mathcal{U}(G)$, we will use that
			$$ (\text{Id} - \epsilon)^{\otimes 2} \Delta(p) = \Delta((\text{Id} - \epsilon)(p)) - (1 - \epsilon)p \otimes 1 - 1 \otimes (1 - \epsilon)p = (\Delta - I_1 - I_2)(1 - \epsilon)(p),$$
			with $I_1(x) = x \otimes 1$ and $I_2(x) = 1 \otimes x$.
			This last equation is easily checked for $p$ such that $p = \epsilon(p)$, while one can use the $\mathbb{N}$-grading for $p$ such that $\epsilon(p) = 0$ to prove that $\Delta(p) = p \otimes 1 + 1 \otimes p + (\text{Id} - \epsilon)^{\otimes 2} \Delta(p)$.
			Hence, we obtain
			\begin{align*}
				(\text{Id} - \epsilon)^{\otimes m + 2} \Delta^{m + 1} (x) = & ((\Delta  - I_1 - I_2) \otimes \text{Id}^{\otimes m}) (\text{Id} - \epsilon)^{\otimes m + 1} \Delta^m (x).
			\end{align*} This proves the lemma.
		\end{proof}
	\end{lemma}
	
	\begin{lemma}
		\label{lemma filtration}
		Suppose that $G$ is a vector group such that $$Z_{m-1} \cap Y_m = Y_{m-1},$$ then $G$ is a well-behaved vector group.
		\begin{proof}
			For each $m$, induction on $m - k$ shows that
			$$Z_k \cap Y_m = Y_{k},$$
			for $k \le m$, using $Z_k \subset Z_{k+1}$.
			This proves the lemma since $\bigcup_m Y_m = \mathcal{U}(G)$.
		\end{proof}
	\end{lemma}
	
	Consider $$T_{n_1,\ldots,n_k} = \{ a \in \{1,\ldots,k\}^n | i \text{ appears } n_i \text{ times in } a\}. $$
	We consider a set of coset representatives $S_{n_1,\ldots,n_k} \subset S_{\sum n_i}$ corresponding to $S_{\sum n_i}/S_{n_1} \times \ldots \times S_{n_k}$, note that this set maps bijectively to $T_{n_1,\ldots,n_k}$ using permutation action of $S_m$ on the element $$(1,\ldots,1,2,\ldots,2,\ldots,k,\ldots,k)$$ formed by taking $n_i$ times $i$ in for each $i \in \{1,\ldots,k\}$.
	
	\begin{lemma}
		\label{Lemma well behaved}
		Each projective vector group is well-behaved.
		\begin{proof}
			Let $G$ be a projective vector group.
			Consider dual bases $(b_i^*)_{i \in I}$ and $(b_i)_{i \in I}$ of $G/G_2$ and $G_2$ so that each $v \in \text{Lie}(G)$ equals
			$$ v = \sum_{i \in I} b_i^*(v)b_i,$$
			with finitely $b_i^*(v) \neq 0$.
			Suppose that $I$ is totally ordered.
			
			We define a map $\text{Im}(\text{Id} - \epsilon)^{\otimes m} \Delta^{m-1}_{|Y_m} \longrightarrow Y_m/Y_{m-1}$ inverse to the map $Y_m/Y_{m-1} \longrightarrow \text{Im}(\text{Id} - \epsilon)^{\otimes m} \Delta^{m-1}_{|Y_m}$ induced by the map $(\text{Id} - \epsilon)^{\otimes m} \Delta^{m-1}_{|Y_m}$ defined on $Y_m$. 
			It is sufficient to construct a map $f$ which  maps as
			$$ \sum_{\sigma \in S_{n_1,\ldots,n_k}} \sigma \cdot b_{i_1} \otimes \cdots \otimes b_{i_1} \otimes \cdots \otimes b_{i_k} \otimes \cdots \otimes b_{i_k} \longmapsto (b_{i_1})_{n_1}\cdots(b_{i_k})_{n_k} \mod Y_{m-1},$$
			on linear generators 
			$$\sum_{\sigma \in S_{n_1,\ldots,n_k}} \sigma \cdot b_{i_1} \otimes \cdots \otimes b_{i_1} \otimes \cdots \otimes b_{i_k} \otimes \cdots \otimes b_{i_k}$$
			of $\text{Im}(\text{Id} - \epsilon)^{\otimes m} \Delta^{m-1}_{|Y_m}$ in which $n_i$ denotes the amount of times $b_i$ appears. To avoid redundant generators we suppose that $i_1,\ldots,i_k \in I$ form a strictly increasing sequence.
			We denote the values we want for $f$ on the set of generators as $f(b_{i_1},n_1,\ldots,b_{i_k},n_k)$.
			For arbitrary elements of $\Ima (\text{Id} - \epsilon)^{\otimes m} \Delta^{m-1}_{|Y_m}$ we define $f$ using
			$$ x \mapsto \sum_{\text{generators}} f(b_{i_1},n_1,\ldots,b_{i_k},n_k) (b_{i_1}^*)^{\otimes n_1} \otimes \ldots \otimes (b_{i_k}^*)^{\otimes n_k} (x) \mod Y_{m-1}.$$
			We remark, to justify the dual maps in the previous definition, that we can use the $b^*_i$ on $\mathcal{U}(G)$. For $b_i \in \text{Lie}(G)_1$ this is obvious, since we have a projection operator $\mathcal{U}(G) \longrightarrow \mathcal{U}(G)_1 \cong \text{Lie}(G)_1$. 
			If $b_i \in \text{Lie}(G)_2$, this is a bit trickier to see. We can use $\pi : \mathcal{U}(G)_2 \longrightarrow \text{Lie}(G)_2 : x \mapsto x - \hat{f} \circ (\text{Id} - \epsilon)^{\otimes 2} \circ \Delta^2 (x)$, with $\hat{f}$ the map with the same definition of $f$, except that it maps to $Y_2$ instead of $Y_2/Y_1$. The image of $\pi$ is $\text{Lie}(G)_2$ since it acts on linear generators as \[v_2 \mapsto v_2 - q(v) = ((v_1,v_2) \cdot (-v_1, - q(v) + v_1^2))_2\] and \[a_1b_1 \mapsto a_1b_1 - q(a,b) = ((a_1,q(a)) \cdot (b_1,q(b)) \cdot (-(a_1 + b_1),-q(a + b)+(a_1+b_1)^2))_2,\] with $q : \text{Lie}(G)_1 \longrightarrow \mathcal{U}(G)_2 : v \mapsto \hat{f}(v_1 \otimes v_1)$ a quadratic map such that $(a_1,q(a))$ in the image of $G$ for all $a \in G$ under the universal representation and $q(a,b) = q(a + b) - q(a) - q(b)$ its linearisation.
			
			Now, we prove that $f$ acts as expected.
			Inductively applying the relations on $\mathcal{U}(G)$ yields $(b_i)_{j_in} \equiv  \prod_{b_j \in I} (b^*_j(b_i)b_j)_{j_in} \mod Y_{m-1}$. Thus, we conclude that the function acts as expected on $b_i \otimes \ldots \otimes b_i$.
			Furthermore, one sees that
			\begin{align}
				\nonumber 
				& \left(\prod_i\binom{n_i + m_i}{n_i}\right)f(b_{i_1},n_1 + m_1, \ldots,b_{i_k},n_k + m_k) \\ \label{equation modulo} \equiv & f(b_{i_1},n_1, \ldots,b_{i_k},n_k)f(b_{i_1},m_1, \ldots,b_{i_k},m_k) \mod Y_{m-1}
			\end{align} using the relations on $\mathcal{U}(G)$.  
			Note that 
			$$ (1 - \epsilon)^{\otimes m+l} \Delta^{m+l-1}(ab) = \sum_{\sigma \in S_{l,m}} \sigma \cdot (1 - \epsilon)^{\otimes l}\Delta^{l-1}(a) \otimes (1 - \epsilon)^{\otimes m}\Delta^{m-1}(b)$$
			for elements $a \in Y_l, b \in Y_m$.
			A single term
			$$f(b_{i_1},n_1,\ldots,b_{i_k},n_k) (b_{i_1}^*)^{\otimes n_1} \otimes \ldots \otimes (b_{i_k}^*)^{\otimes n_k}(1 - \epsilon)^{\otimes k + l}\Delta^{k + l - 1}(ab)$$
			is computed by considering all possible sums $\sum_{i = 1}^k o_j = l$ with $o_j \le n_j$ and for each such sum considering all possible ways to choose of $o_j$ of the $b^*_{i_j}$ and evaluating the chosen $b^*$ on $(1 - \epsilon)^{\otimes l}\Delta^{l-1}(a)$ while evaluating the remaining $n_j - o_j$ of the $b^*_{i_j}$ on $(1 - \epsilon)^{\otimes m}\Delta^{m-1}(b)$. This yields the same result as directly evaluating
			$$ \sum_{\substack{\sum_j o_j = l\\ \sum_j p_j = m\\ o_j + p_j = n_j}} f(b_{i_1},o_1,\ldots)(b_{i_1}^*)^{\otimes o_1} \otimes \ldots \otimes (b_{i_k}^*)^{\otimes o_k} \bigotimes f(b_{i_1},p_1,\ldots)(b_{i_1}^*)^{\otimes p_1} \otimes \ldots \otimes (b_{i_k}^*)^{\otimes p_k}$$
			on 
			$$ (1 - \epsilon)^{\otimes l}\Delta^{l-1}(a) \otimes (1 - \epsilon)^{\otimes m}\Delta^{m-1}(b)$$
			by Equation (\ref{equation modulo}).
			This proves that $$f(1 - \epsilon)^{\otimes m + l}\Delta^{m + l - 1}(ab) \equiv f(1 - \epsilon)^{\otimes l}\Delta^{l-1}(a) f(1 - \epsilon)^{\otimes m}\Delta^{m-1}(b) \mod Y_{m + l - 1}$$
			for $a \in Y_{l}$ and $b \in Y_{m}$.
			So, $f$ takes the right values on all generators since $Y_nY_m \subset Y_{n+m}$ and since $Y_{n+m}$ is generated by all $Y_nY_m, nm \neq 0$ and the $g_{n+m}, h_{2(n+m)}$ for $g \in G, h \in G_2$. This proves that $$Y_{m-1} = \ker (1 - \epsilon)^{\otimes m} \Delta^{m-1} \cap Y_m.$$
			Therefore, Lemma \ref{lemma filtration} shows $G$ to be well-behaved.
		\end{proof}
	\end{lemma}
	
	\begin{lemma}
		\label{Lemma Y1}
		$Y_1 \cap \ker \epsilon \cong \Lie(G)$.
		\begin{proof}
			We know that $\Lie(G) \cong \langle g_1, h_2 | g \in G, h \in G_2\rangle$ and that $Y_1$ is generated by these elements and the $g_0$, which are equal to $1$.
		\end{proof}
	\end{lemma}
	
	\begin{lemma}
		\label{lemma Z1}
		$Z_1 \cap \ker \epsilon \cong \text{P}(\mathcal{U}(G))$
		\begin{proof}
			We know that $P(\mathcal{U}(G)) \subset \ker \epsilon$.
			For $x \in \ker \epsilon$ we know that $x$ is primitive if and only if
			\begin{align*}
				& \Delta(x) - x \otimes 1 - 1 \otimes x = 0 \\
				\iff&  \Delta(x) - (\text{Id}\otimes \epsilon) \Delta(x) - (\epsilon \otimes \text{Id}) \Delta(x) + (\epsilon \otimes \epsilon ) \Delta(x) = 0 \\
				\iff & (\text{Id}- \epsilon)^{\otimes 2} \Delta(x) = 0 \\
				\iff&  x \in Z_1,
			\end{align*}
			using that $\mu(\epsilon \otimes \text{Id}) \Delta = \text{Id} = \mu(\text{Id} \otimes \epsilon ) \Delta$ and $\epsilon^{\otimes 2} \Delta = \Delta \epsilon$.
		\end{proof}
	\end{lemma}
	
	\begin{theorem}
		\label{theorem PUG}
		If $G$ is a projective vector group, then $P(\mathcal{U}(G))$ is isomorphic to $\Lie(G)$. Consequently, the group $G(K)$ can be recovered from the universal representation, using
		$$ \{ g = 1 + tg_1 + t^2 g_2 + \ldots \in (\mathcal{U}(G) \otimes K)[[t]] \; | \; \Delta(g) = g \otimes g, g_i \in \mathcal{U}(G)_i \otimes K\} \cong G(K).$$

		\begin{proof}
			Lemmas \ref{Lemma well behaved}, \ref{Lemma Y1}, and \ref{lemma Z1} prove that $P(\mathcal{U}(G)) \cong \text{Lie}(G)$.
			
			Now, we prove that each group-like element 
			$$ g = \sum_{i = 1}^\infty t^ig_i,$$
			with $g_i \in \mathcal{U}(G)_i \otimes K$ is contained in the image of $G(K)$ under the universal representation.
			Remark that the first nonzero $g_i$ must be primitive, so that the first nonzero $g_i$ must be contained in $\text{Lie}(G) \otimes K$.
			We can use this to write
			$$ g = h_1h_2,$$
			with $h_1 \in G(K), \; h_2 \in G_2(K)$ such that $h_1$ and $g$ share the first coordinate and $h_2$ and $h_1^{-1}g$ share the second coordinate. 
		\end{proof}
	\end{theorem}
	
	\begin{remark}
		In this section, we defined vector groups using
		$$ G \le A_1 \times \ldots \times A_n,$$
		with $n = 2$ since these are the only vector groups we need in this article.
		It is possible to generalize what we have done here and give an exact way to construct group functors $K \mapsto G(K)$ from $G(\Phi)$ alone that are equivalent as a class of groups to groups of formal power series $ K \mapsto \{1 + tg_1 + t^2g_2 + \ldots | g \in G(K) \}$ closed under $$1 + t^ig_i + t^{2i}g_{2i} +  \ldots \mapsto 1 + \lambda t^i g_i + \lambda^2 t^{2i} g_{2i}\ldots$$ satisfying
		$$ G_i/G_{i+1} \otimes K \cong G_i(K)/G_{i+1}(K),$$
		with $G_i(K) = \{ g \in G(K) | g_{k} = 0 , k = 1,\ldots,i-1\}$ and for which there exists $i$ such that $G_i = 0$ (this restriction can be lifted by considering direct limits). 
		In that case, we have $\text{Lie}(G) = G/G_2 \times G_2/G_3 \times \ldots.$
	\end{remark}
	
\section{Pairs of vector group representations}\label{se:pairs}

As before, we assume that all vector groups under consideration are proper.

Now we introduce some notation that we will use during the following sections. We introduce some operators $\mu, \nu$ which depend on a pair of representations of vector groups. These operators will be useful if one wants, for example, to consider the group generated by two vector group representations.

We choose to focus on the conjugation $(g,h) \mapsto ghg^{-1}$ to express the relation between the groups, however using $(g,h) \mapsto gh, (g,h) \mapsto [g,h]$ one could define similar operators replacing the $\mu$ we shall define, to which everything else can be translated.

\begin{definition}
	Let $G^+,G^-$ be vector groups and let $$(\rho^+_i : G^+ \longrightarrow A)_i$$
	and
	$$ (\rho^-_i : G^- \longrightarrow A)_i$$ be vector group representations. 
	Consider the functions $$\mu_{i,j},\nu_{i,j} : G^+ \times G^- \longrightarrow A,$$ uniquely defined from
	$$\rho^+_{[s]}(x) \rho^-_{[t]}(y) \rho^+_{[s]}(x^{-1}) = \sum_{(i,j) \in \mathbb{N \times \mathbb{N}}} s^it^j \mu_{i,j}(x,y),$$ 
	$$\rho^+_{[s]}(x) \rho^-_{[t]}(y) \rho^+_{[s]}(x^{-1}) = \prod_{\substack{\gcd(i,j) = 1\\ (i,j) \in \mathbb{N} \times (\mathbb{N}_{>0})}} \left(1 + \sum_{k = 1}^\infty s^{ki}t^{kj}\nu_{ki,kj}(x,y)\right),$$ where the order of multiplication on the $(i,j),(k,l) \in \mathbb{N} \times (\mathbb{N}_{>0})$ corresponds to $il < jk$.
	Similarly, we define $\mu,\nu : G^- \times G^+ \longrightarrow A$.
\end{definition}

\begin{lemma}
	The maps $\mu_{i,j}, \nu_{i,j} : G^+ \times G^- \longrightarrow A$ are well-defined.
	\begin{proof}
		This is obviously the case for $\mu$.
		For $\nu$, one can use recursion on the degrees to compute the $\nu$'s uniquely.
		Specifically, we can compute $\nu_{l,m}$ using induction on $l + m$, by observing that $\nu_{l,m}$ is, recursively, the only undefined term of 
		$$\rho^+_{[s]}(x) \rho^-_{[t]}(y) \rho^+_{[s]}(x^{-1}) = \prod_{\substack{\gcd(i,j) = 1\\ (i,j) \in \mathbb{N} \times (\mathbb{N}_{>0})}} \left(1 + \sum_{k = 1}^\infty s^{ki}t^{kj}\nu_{ki,kj}(x,y)\right)$$
		when both expressions are considered in
		$ A[s,t]/(s^{l+1},t^{m+1})$.
	\end{proof}
\end{lemma}

\begin{remark}
	We remark that $\mu_{i,j}(x,y) = \sum_{a + b = i} x_a y_j (x^{-1})_b$.
	The maps $\nu_{i,j}$ can be computed recursively, starting from $\nu_{0,i} = \mu_{0,i}, \nu_{i,1} = \mu_{i,1}$ and then
	$\nu_{i,2} + \sum_{\substack{a + b = i\\a < b}} \nu_{a,1}\nu_{b,1} = \mu_{i,2}$.
	For higher $k$ we can also find formulas for $\nu_{i,k}$ corresponding to the previous lemma.	
\end{remark}

\begin{remark}
	The maps $\mu,$ and $\nu$ are homogeneous in both arguments. Namely, if $x \mapsto f(x)$ is homogneous, then $x \mapsto f(x^{-1})$ is homogeneous as well and if $f$ and $g$ are homogeneous of degree $i$ and $j$, then $x \mapsto f(x) \otimes g(x)$ is homogeneous of degree $i + j$. This is easily proved by using the Hopf algebra structure of the universal representation and the property that each homogeneous map factors through the universal representation as a linear map.
\end{remark}

We also use 
$$\exp(o_{i,j}(x,y),l) = \left(1 + \sum_{k = 1}^\infty l^k\nu_{ki,kj}(x,y)\right),$$
to denote the formal power series which are used to define the $\nu_{i,j}(x,y)$, where we use $l$ as the formal scalar instead of $s^it^j$.

\begin{lemma}
	\label{lemma grouplike}
	Consider a pair of vector group representations $(\rho^+_i:G^+ \longrightarrow A,\rho^-_i:G^- \longrightarrow A)$ to a Hopf algebra $A$.
	Suppose that $\rho^\pm_{[t]}(g)$ is group-like for all $g \in G^+ \cup G^-$, then 
	$\exp(o_{i,j}(g,h),t)$ is group-like as well for all $i,j$.
	\begin{proof}
		Parallel to the recursion used to define the $\nu$ one can prove inductively that 
		$ \exp(o_{i,j}(x,y),s^it^j)$ is a group-like element in $A[s,t]/(s^{ki+1}t^{kj+1})$ when one determines $\nu_{ki,kj}$.
		Specifically, in the inductive step one has
		$$\rho^+_{[s]}(x) \rho^-_{[t]}(y) \rho^+_{[s]}(x^{-1}) = \prod_{\substack{\gcd(i,j) = 1\\ (i,j) \in \mathbb{N} \times \mathbb{N}_{>0}}} \left(1 + \sum_{k = 1}^\infty s^{ki}t^{kj}\nu_{ki,kj}(x,y)\right) \mod (s^{ki+1}t^{kj+1}),$$
		where the left-hand side is group-like and the right-hand side is $ \exp(o_{i,j}(x,y),s^it^j)$ multiplied with group-like elements.
		Since the group-like elements form a group, we proved the induction step.	
	\end{proof}
\end{lemma}

\begin{remark}
	The definition of the $\nu_{i,j}$ is roughly the same as Ditter-Shay Bi-isobaric decomposition, cfr. \cite[Theorem 3.5]{Haz07}.
	We defined $\exp(o_{i,j}(x,y),l)$. It will be useful to abstract the $\exp$ and $l$ away.
	Specifically, we consider
	$ o_{i,j} $ as a sequence of maps
	$$ (\nu_{ki,kj} : G^+ \times G^- \cup G^- \times G^+ \longrightarrow A)_k,$$
	which will be able to play a similar role as $\rho^+, \rho^-$.
	The $o_{i,j}$ have the nice property that they can coincide with group elements of $G^+$ and $G^-$. Furthermore, by considering the $o_{i,j}$ as similar to $\rho^\pm$ it makes sense to speak about $\nu_{k,l}(x,o_{i,j}(y))$, etc..
\end{remark}

\begin{definition}
	Let $\rho_i : G \longrightarrow A$ be a vector group representation.
	We use $\hat{\rho}(G)$ to denote the subalgebra of $A$ generated by all $\rho_i(g).$ 
	Furthermore, we use $\mathcal{H}(G^+,G^-,\rho)$ to denote the algebra $$\langle \nu_{i,i}(x,y) | x \in G^+, y \in G^-, i \in \mathbb{N} \rangle.$$
\end{definition}

\begin{remark}
	Now, we will formulate a theorem involving seemingly random conditions. The conditions can be interpreted as requiring that the pair of representations corresponds to a representation of an operator Jordan-Kantor pair (which we will not define formally), similar to a divided power representation of Jordan pair as introduced by Faulkner \cite{FLK00}.
\end{remark}

\begin{theorem}
	\label{thm main}
	Suppose that $(\rho^\pm_i : G^\pm \longrightarrow A)_i$ are vector group representations of proper vector groups such that for $x \in G^\pm, y \in G^\mp$,
	\begin{itemize}
		\item $\exp(o_{2,1}(x,y),s) \in \rho^\pm_{[s]}(G^\pm)$, 
		\item $\exp(o_{3,1}(x,y),s^2) \in \rho^\pm_{[s]}(G^\pm)$ ,
		\item $\exp(o_{i,j}(x,y),s) = 1$ for $i > 2j$ and $i\neq 3j$,
		\item there exists $z \in G^\pm$ such that $\rho_{[s]}(z) = 1 + s\nu_{3,2}(x,y) + O(s^2)$,
	\end{itemize}
	over all scalar extensions,
	then $\hat{\rho}(G^-)\mathcal{H}(G^+,G^-,\rho)\hat{\rho}(G^+)$ is a subalgebra of $A$.
	
	Moreover, if $\rho^+$ and $\rho^-$ are injective on projective vector groups $G^\pm$, then the unique map from the algebra $F$ freely generated by $\mathcal{U}(G^+)$ and $\mathcal{U}(G^-)$ to $A$ corresponding to this representation factors through an algebra structure on the coalgebra $$H = \mathcal{U}(G^-) \otimes \hat{\mathcal{H}}(G^+,G^-,F) \otimes \mathcal{U}(G^+)$$ 
	in which $y \otimes \hat{h} \otimes x$ represents the element $yhx$ in $F$ and $\hat{\mathcal{H}}(G^+,G^-,F)$ is a quotient of $\mathcal{H}(G^+,G^-,F)$. This algebra is $\mathbb{Z}$-graded, if one combines the $\mathbb{N}$-grading of $\mathcal{U}(G^+)$, with the opposite of the $\mathbb{N}$-grading of $\mathcal{U}(G^-)$ and sets $\mathcal{H}(G^+,G^-,F)$ to be $0$-graded.
	\begin{proof}
		Proof in appendix \ref{section proof}.
	\end{proof}
\end{theorem}

\begin{remark}
	If the vector groups $G^\pm$ are not projective, the moreover part of the previous theorem still holds if one replaces $\mathcal{U}(G^\pm)$ by a suitable quotient in the definition of $H$.
\end{remark}

Now, we prove a technical lemma to obtain the necessity of certain linearisation formulas which will be useful in the following sections.
We extend the domains of $\mu_{i,j}$, $\nu_{i,j}$, so that they can also take $o_{k,l}(x,y)$ as input. 
We use $\nu_{(a,b),(c,d)}$ to express certain $((a,b),(c,d))$-linearisations of $\nu_{a+b,c+d}$, corresponding to linearizing the homogeneous maps $G^+ \longrightarrow \text{Hom}(G^-,A),G^- \longrightarrow \text{Hom}(G^+,A).$

\begin{lemma}
	\label{lemma linearizations mu}
	Consider a pair of vector group representations $(\rho^+_i:G^+ \longrightarrow A,\rho^-_i:G^- \longrightarrow A)$. 
	The linearisations of $\mu$ can be computed as
	$$ \mu_{(i,j),k}(x,z,y) = \sum_{a + b = i} x_a \mu_{j,k}(z,y) (x^{-1})_b, \quad \mu_{i,(j,k)}(x,z,y) = \sum_{a + b = i} \mu_{a,j}(x,z)\mu_{b,k}(x,y).$$
	\begin{proof}
		This follows easily if one computes both sides from conjugations and products of $\rho_{[s]}(x), \rho_{[t]}(z),$ and $\rho_{[u]}(y)$.
	\end{proof}
\end{lemma}

\begin{lemma}
	\label{Lemma equations}
	Suppose that $\rho^\pm_i : G^\pm \longrightarrow A$ are vector group representations.
	The following formulas hold for $x,g \in G^+, y \in G^-$
	\begin{enumerate}
		\item $\sum t^{k+l} \nu_{ki,kj}(x^{-1},y^{-1})\nu_{li,lj}(y,x) = 1,$
		\item \label{REP EQ V1}$[\nu_{1,1}(y^{-1},x^{-1}),\rho_1(g)] = \nu_{(1,1),1}(g,x,y)$,
		\item \label{REP EQ Tau1}$\nu_{2,1}(o_{1,1}(y^{-1},x^{-1}), g) = \nu_{(1,2),2}(g,x,y)$,
		\item \label{REP EQ V 2}$[\nu_{1,1}(y^{-1},x^{-1}),\rho_2(g)] = \nu_{(2,1),1}(g,x,y) + \nu_{(1,1),1}(g,x,y)\rho_1(g)$,
		\item $\nu_{(2,1),2}(g,x,y) = \nu_{2,1}(g,o_{2,1}(y^{-1},x^{-1})) - [\mu_{1,1}(x,y),\mu_{2,1}(g,y)],$
		\item $\nu_{(3,1),2}(g,x,y) = \nu_{3,1}(g,o_{2,1}(y^{-1},x^{-1})) - [\mu_{1,1}(x,y),\mu_{3,1}(g,y)] + \nu_{2,1}(g,y)\nu_{(1,1),1}(g,x,y)$,
		\item $\nu_{(1,3),2}(x,g,y) = [\rho_1(x),\nu_{3,2}(g,y)] + \nu_{(1,1),1}(x,g,y)\nu_{2,1}(g,y).$			
		\item 
		Furthermore, if $\nu_{4,1}(x,y) = 0$ for all $x \in G^+,y \in G^-$, then we also have
		\begin{enumerate}				
			\item \label{REP EQ Tau2}$\nu_{2,2}(o_{1,1}(y^{-1},x^{-1}), g) = \nu_{(2,2),2}(g,x,y) - \nu_{2,1}(g,y)\nu_{2,1}(x,y)).$
			\item $\nu_{3,(1,1)}({y,g,x}) = [\nu_{3,1}(y,g),\rho_1(x)] - [\nu_{1,1}(y,x),\nu_{2,1}(y,g)],$
			\item $\nu_{4,(1,1)}({y,g,x}) = [\nu_{3,1}(y,g),\nu_{1,1}(y,x)] + \nu_{2,1}(y,g)\nu_{2,1}(y,x).$
		\end{enumerate}
	\end{enumerate}		
	\begin{proof}
		The first equation follows from
		$$ (\rho^-_{[t]}(y^{-1})\rho^+_{[s]}(x)\rho^-_{[t]}(y)\rho^+_{[s]}(x^{-1}))^{-1} = \rho^+_{[s]}(x)\rho^-_{[t]}(y^{-1})\rho^+_{[s]}(x^{-1})\rho^-_{[t]}(y),$$
		since this proves, using the defining expressions for $o_{i,j}$ with slightly different indices so that they correspond to commutators, that
		$$ \left(\prod_{\gcd(i,j) = 1, (i,j) \in (\mathbb{N}_{>0})^2} \exp(o_{i,j}(x,y),s^it^j) \right)^{-1} = \prod_{\gcd(i,j) = 1, (i,j) \in (\mathbb{N}_{>0})^2} \exp(o_{i,j}(y^{-1},x^{-1}),s^jt^i).$$
		In the rest, we write $g_i$ for $\rho_i(g)$.
		
		The rest are just calculations involving linearisations. 
		These equations can be proved by rewriting the $\nu$ on the left-hand side in terms of $\mu$, and smartly using
		$$ \mu_{(a,b),c}(g,x,y) = \sum_{i + j = a}g_i \mu_{b,c}(x,y) (g^{-1})_j$$
		and 
		$$ \mu_{a,(b,c)}(y,g,x) = \sum_{i + j = a} \mu_{i,b}(y,g)\mu_{j,c}(y,x). $$
		
		We prove, as an example, equation $8.a$, since the proof of this equation showcases all techniques needed to prove the other cases.
		From the decomposition of $$\exp(o_{1,1}(y^{-1},x^{-1}),l)\rho_{[s]}(g)\exp(o_{1,1}(x,y),l)\rho_{[s]}(g^{-1})  = \exp(o_{1,1}(y^{-1},x^{-1}),l) \sum s^il^j \mu_{i,j}(g,o_{1,1}(x,y))$$
		as a product
		$$ \prod_{\gcd(i,j) = 1, (i,j) \in (\mathbb{N}_{>0})^2}
		\exp(o_{i,j}(o_{1,1}(y^{-1},x^{-1}),g),l^is^j),$$ we learn that 
		\begin{equation}
			\label{equation lemma equations 1}
			\nu_{2,2}((o_{1,1}(y^{-1},x^{-1}), g) = \mu_{2,2}(g,o_{1,1}(x,y)) - \mu_{1,1}(x,y)\mu_{2,1}(g,o_{1,1}(x,y))
		\end{equation}
		by comparing the terms belonging to $l^2s^2$ and using that $\nu_{l,1} = \mu_{l,1}$.
		The ordinary definitions of $\nu$ and $\mu$ yield \[\nu_{2,2}(x,y) = \mu_{2,2}(x,y) - y_1\mu_{2,1}(x,y),\] so that
		\begin{equation}
			\label{equation lemma equations}
			\mu_{2,2}(g,o_{1,1}(x,y)) = \mu_{(2,2),2}(g,x,y) - \sum_{i + j = 2} g_i y_1 \mu_{2,1}(x,y) (g_j)^{-1}
		\end{equation}
		must hold since $\nu_{2,2}(x,y)$ is the second coordinate of $o_{1,1}(x,y)$, using the definition of $\mu_{2,\cdot}(g,\cdot)$.
		Equation (\ref{equation lemma equations}) can be rewritten further as
		\begin{equation}
			\label{equation lemma equations 2}
			\mu_{2,2}(g,o_{1,1}(x,y)) = \mu_{(2,2),2}(g,x,y) - \sum_{k + l= 2} \mu_{k,1}(g,y) \mu_{(l,2),1}(g,x,y),
		\end{equation}
		using $\sum_{i + j = 2} g_i y_1 \mu_{2,1}(x,y) (g_j)^{-1} = \sum_{k + l= 2} \mu_{k,1}(g,y) \mu_{(l,2),1}(g,x,y)$ by Lemma \ref{lemma linearizations mu}.
		Similarly, we have
		\begin{equation}
			\label{equation lemma equations 3}
			- \mu_{1,1}(x,y)\mu_{2,1}(g,o_{1,1}(x,y)) = - \mu_{1,1}(x,y)\mu_{(2,1),1}(g,x,y).
		\end{equation}
		Combining Equations (\ref{equation lemma equations 1}), (\ref{equation lemma equations 2}), and (\ref{equation lemma equations 3}), we obtain
		$$ \nu_{2,2}(g,o_{1,1}(y^{-1},x^{-1})) = \nu_{(2,2),2}(g,x,y) - \mu_{2,1}(g,y)\mu_{2,1}(x,y), $$
		by using
		$$ \nu_{(2,2),2}(g,x,y) = \mu_{(2,2),2}(g,x,y) - \mu_{1,1}(g,y)\mu_{(1,2),1}(g,x,y) - \mu_{1,1}(x,y)\mu_{(2,1),1}(g,x,y)$$
		and $\mu_{(2,2),1}(g,x,y) = 0$.
		Using $\mu_{2,1} = \nu_{2,1}$ yields the equation we wanted to prove.	     
	\end{proof}	
\end{lemma}

\begin{remark}
	Once one understands how to work with the linearisations, the previous equations are not that hard to prove. They are just very tedious to write down. All of the computations can be expressed relatively compactly, however.
	For example, if one schematically writes $(\cdot,\cdot)$ for $\mu_{\cdot,\cdot}$ and $[\cdot,\cdot]$ for $\nu_{\cdot,\cdot}$, the previous computation can be restated as
	\begin{align*}
		[2,2]_{o,g} & = (2,2)_{g,o} - (1,1)_{x,y}(2,1)_{g,o} = ((2,2),2) - ((0,1),1)((2,1),1) -\sum_{i + j = 2} ((i,0),1)((j,2),1) \\
		& = [(2,2),2] - ((2,0),1)((0,2),1) - ((0,0),1)((2,2),1) = [(2,2),2] - [(2,0),1][(0,2),1], 
	\end{align*}
	where we did not write evaluations in $(g,x,y)$, used the linearisation properties of $\mu$ and used expressions for $\nu_{i,j}$'s in terms of $\mu_{i,j}$'s. For this computation, one can express the $\nu$'s in terms of $\mu$'s using
	$$ [2,2] = (2,2) - (0,1)(2,1), \quad [4,2] = (4,2) - (1,1)(3,1).$$
	In this last expression we dropped the term $(0,1)(4,1)$, since we assume that $(4,1) = 0$ in the lemma.
	Another relation that is useful in proving the previous lemma is
	$$ [3,2] = (3,2) - (0,1)(3,1) - (1,1)(2,1).$$
\end{remark}
	
\section{Operator Kantor pairs}\label{se:okpairs}

\subsection{Pre-Kantor pairs}

Let $G^+ \le A^+ \times B^+$ and $G^- \le A^- \times B^-$ be a pair of proper vector groups with associated bilinear forms $\psi : A \times A \longrightarrow B$ (we will not write uninformative $\pm$ signs). For  ease of notation, we assume that 
$$ \text{Lie}(G^+) = A^+ \oplus G^+_2,$$
i.e., we assume that $A^+$ is as small as possible (one can always modify $A^+ \times B^+$ so that this is the case). Furthermore, we can assume that $\hat{G} \cong G$ as functors, which is possible since $G$ is proper.
We assume the same for $G^-$ and $A^-$.

\begin{lemma}
	Let $G$ be a vector group.
	The space $[G,G] \le G_2$ is a $\Phi$-submodule.
	\begin{proof}
		By definition $[G,G] \le G_2$ is the subgroup generated by all the commutators of $G$, so it is an additive subspace.
		We prove that it is closed under $\cdot_2$ as well.
		This is the case, since the equation
		$$ \lambda \cdot_2 [a,b] = [\lambda \cdot_1 a, b]$$
		holds for all $a, b \in G$.
	\end{proof}
\end{lemma}

Let $M$ be a module. We call a natural transformation $f : G^+ \times G^- \longrightarrow M$ homogeneous of \textit{bidegree} $[i,j]$ if it is homogeneous of degree $i$ in the first component and of degree $j$ in the second component.
We often use $f_x y$ to denote $f(x,y)$ for a map that is homogeneous in both arguments.

\begin{assumption}
	Throughout this section we will be working with certain operators $Q,R,T,P, Q^\text{grp}$ defined for $G^+,G^-$. We assume, in this section, that these are of the following form: 
	\begin{itemize}
		\item $ Q^{\text{grp}} : G^+ \times G^- \longrightarrow G^+ : (g,h) \mapsto Q^{\text{grp}}_gh = (Q_g h, R_g h)$
		with $Q$ of bidegree $[2,1]$ and $R$ of bidegree $[4,2]$,
		\item $ T : G^+ \times G^- \longrightarrow [G^+,G^+] : (g,h) \mapsto (0,T_g h)$
		of bidegree $[3,1]$,
		\item $ P : G^+ \times G^- \longrightarrow A^+: (g,h) \mapsto P_g h,$ 
		of bidegree $[3,2]$.
	\end{itemize}
	We also use $Q,R,T,P,Q^\text{grp}$ to denote operators $G^-\times G^+ \longrightarrow M$ for the appropriate $M$.
	These operators $Q^\text{grp},T$ will in a precise sense play the role of $o_{2,1}$ and $o_{3,1}$ of Theorem \ref{thm main} and $P$ will play the role of $\nu_{3,2}.$
\end{assumption}

\begin{definition}
	Let $C$ be an associative algebra with involution $a \mapsto \overline{a}$ and consider a left $C$-module $M$. We call a $\Phi$-bilinear map $h : M \times M \longrightarrow C$ a (left-)hermitian form if it satisfies $h(cm,n) = ch(m,n)$ and $\overline{h(m,n)} = h(n,m)$.
\end{definition}

\begin{example}
	\label{example hermitian form}
	Now, we will consider an example of such a set of operators that will form an operator Kantor pair. This will be proved after we define operator Kantor pairs.
	Consider a hermitian form $h : M \times M \longrightarrow C$ and assume that $C$ acts faithfully on $M$.
	The Kantor triple system $(M,V_{x,y})$, with $V_{x,y} : M \longrightarrow M$ for $x,y \in M$ (for a definition of a Kantor triple system, see \cite[section 3]{ALLFLK99}), associated to this hermitian form is given by
	$$ V_{x,y} z = h(x,y)z + h(z,y)x - h(z,x)y.$$

	Set $G = \{ (m,a) \in M \times C | a + \bar{a} = h(m,m), a - \bar{a} \in \langle  h(x,y) - h(y,x) |x,y \in M \rangle\}$ with operation $$(m,a)(n,b) = (m + n, a + b +h(m,n)).$$
	We cannot guarantee that $G$ is proper. So, we assume that $C$ and $M$ are chosen in such a way that $G$ is proper. We remark that setting $M'= M \otimes \Phi[s]/(s^2 - s)$ and $C' = C \otimes \Phi[s]/(s^2 - s)$ with involution $c \otimes s \mapsto \bar{c} \otimes (1 - s)$ guarantees that $G'$ is proper. Namely, we have $G'_2 = \langle h(x,y) \otimes s - h(y,x) \otimes (1 - s) \rangle$ so that $G'_2 \otimes K \longrightarrow \langle h(M,M) \rangle \otimes \Phi[s]/(s^2 - s) \otimes K$ is injective for all $\Phi$-algebras $K$. Therefore, we know there exists a realisation of $\hat{G}'$ of $G$ such that $\Lie(\hat{G}'(K)) \cong \Lie(G') \otimes K$, which proves that $G'$ is proper by Lemma \ref{lem: Lie of G functorial}.
	Define 
	\begin{enumerate}
		\item $Q_{(m,a)}(n) = - h(m,n)m + an,$
		\item $T_{(m,a)}(n) =  h(m,an) -h(an,m),$
		\item $P_{(m,a)}(n,b) =  - h(m,n)Q_{(m,a)}(n) - abm,$
		\item $R_{(m,a)}(n,b) = ab\bar{a} + h(m,n)ah(n,m) - h(m,n)h(m,an).$
	\end{enumerate}
	A direct computation shows that $Q^{\text{grp}}_g h = (Q_g h,R_g h)$ is an operator mapping to $G$. At the end of this section we will prove that $(G,G,Q^\text{grp},T,P)$ forms an operator Kantor pair if $1/2 \in \Phi$ or if
	\begin{equation}
		\label{condition kantor pair hermitian form}
		\langle c - \bar{c} | (a,c) \in G \rangle \subseteq [G,G].
	\end{equation}
	Once again, we remark that this assumption holds whenever we consider $M'$ and $C'$ instead of $M$ and $C$.
	One can recover the original Kantor triple system using
	$$ V_{x,y} z = - Q^{(1,1)}_{z,x} y,$$
	with $Q^{(1,1)}$ the $(1,1)$-linearisation of $Q$.
\end{example}

\begin{definition}
	\label{definition V tau}
	We want to define certain new operators $V, \tau : G^+ \times G^- \longrightarrow \text{End}(G^+) \times \text{End}(G^-)$, and also define them on $G^- \times G^+$.
	In this definition we write $V^+_{g,h}, V^-_{g,h}$ for the projections of $V_{g,h}$ as endomorphisms of $G^+, G^-$ (outside the scope of this definition, we will just write $V$). We do the same for $\tau$.
	Use $f^{(i,j)}$ to denote the $(i,j)$-linearisation $f$ to the first component.
	For $x,g \in G^\epsilon$ and $y \in G^{-\epsilon}$, $\epsilon \in \{\pm\}$, we put
	$$ V^\epsilon_{y^{-1},x^{-1}} g = (Q^{(1,1)}_{g,x} y, T^{(2,1)}_{g,x} y + \psi(Q^{(1,1)}_{g,x} y,g))$$
	and 
	$$ \tau^\epsilon_{y^{-1},x^{-1}} g = (P^{(1,2)}_{g,x} y, R^{(2,2)}_{g,x} y - \psi(Q_g y,Q_x y) + \psi(P^{(1,2)}_{g,x} y, g)).$$	
	So far we have defined $V^{\epsilon},\tau^{\epsilon}$ on $G^{-\epsilon} \times G^\epsilon$. 
	Now, we want to define $V^{\epsilon}, \tau^{\epsilon}$ when these take arguments in $G^{\epsilon} \times G^{-\epsilon}$.
	
	By working over $\Phi[\eta]$ with $\eta^3 = 0$ and taking elements in $G^\pm \le G^\pm(\Phi[\eta])$ under the canonical embedding, we can rewrite the definition of $V^+,\tau^+$ as
	\begin{align*}
		& \left(1 + \eta V^{+}_{y^{-1},x^{-1}} + \eta^2 \tau^{+}_{y^{-1},x^{-1}}\right)(g) = \left(\eta Q^{(1,1)}_{g,x} y + \eta^2P^{(1,2)}_{g,x} y,  \eta T^{(2,1)}_{g,x} y + \eta^2 R^{(2,2)}_{g,x} y - \eta^2\psi(Q_gy,Q_xy)\right) \cdot g,
	\end{align*}
	for $x,g \in G^+$, $y \in G^-$.

	We use $$(1 + \eta V^\epsilon_{x,y} + \eta^2 \tau^\epsilon_{x,y})(1 + \eta V^\epsilon_{y^{-1},x^{-1}} + \eta^2 \tau^\epsilon_{y^{-1},x^{-1}}) = 1$$
	to define $V^\epsilon_{x,y}$ and $\tau^{\epsilon}_{x,y}$ from the already defined $V^\epsilon_{y^{-1},x^{-1}}, \tau^\epsilon_{y^{-1},x^{-1}}$.
\end{definition}

\begin{remark}
	Note that the definition of $V$ and $\tau$ is not that easy to give, although it could be easier for $V$ since $V_{x,y} = - V_{y^{-1},x^{-1}}$. We extended the definition using
	$$(1 + \eta V_{x,y} + \eta^2 \tau_{x,y})(1 + \eta V_{y^{-1},x^{-1}} + \eta^2 \tau_{y^{-1},x^{-1}}) = 1,$$
	which encodes the relationship between the very difficult to define operator $\tau$ and the more easily defined $V$.
	This relationship corresponds to the fact that $(G^+,G^-) \longrightarrow (G^-,G^+):(x,y) \longmapsto (y^{-1},x^{-1})$ corresponds to $o_{i,j}(x,y) \mapsto o_{j,i}(y^{-1},x^{-1})^{-1}$ for any pair of vector group representations.
\end{remark}

\begin{remark}
	If one were to use $Q = \nu_{2,1}, T = \nu_{3,1}, P = \nu_{3,2}, R = \nu_{4,2}$ for some vector group representations (later we shall see that the operators can always be understood in this way), then one can check that the definition of $V_{y^{-1},x^{-1}}$ and $\tau_{y^{1},x^{-1}}$ is derived from the conjugation action of $\exp(o_{1,1}(y^{-1},x^{-1},\eta)).$
	We can uniquely recover the definition of $V$ from equations $2$ and $4$ of Lemma \ref{Lemma equations}, where one can think of the left-hand side of those equations as representing 
	$$ [V_{y^{-1},x^{-1}},\rho_i(g)].$$
	Similarly, one can uniquely recover the definition of $\tau$ from equations $3$ and $8.a$, where we gave expressions $\nu_{2,i}(o_{1,1}(y^{-1},x^{-1}),g)$ instead of the usual conjugation action $\mu_{2,i}(o_{1,1}(y^{-1},x^{-1}),g)$ (which is no problem since one can easily go back and forth between $\nu_{i,j}$'s and $\mu_{i,j}$'s).
\end{remark}

From now onward we do not make a distinction between homogeneous maps $f :G^\pm \longrightarrow M$ of degree $1$, and linear maps $A^\pm \longrightarrow M$.
In particular, we can write $V_{y,x} = V_{-y,-x} = V_{y^{-1},x^{-1}}$.
We use the same convention to identify the commutator on $\text{Lie}(G)$ and the commutator on $G$, using $[a,b] = [(a,\cdot),(b,\cdot)].$

\begin{definition}
	\begin{enumerate}
		\item Let $f : G^\pm \times G^\mp \longrightarrow M$ be a map defined on a product of vector groups which is homogeneous of degree $n$ in the first component and of degree $m$ in the second component. 
		We use $f^{(i,j)}$ to denote the linearisation of $f$ in the first component, $f^{(n,(i,j))}$ to denote the linearisation of $f$ in the second component, and $f^{((i,j),(k,l))}$ if we use linearisations in both components.
		The reason for this asymmetry in notation is that the first component in $f$ plays conceptually a more important role. The used notation $g \mapsto f_g$ can be interpreted as corresponding to a natural transformation $G^\pm \longrightarrow \text{Nat}(G^\mp,M)$.
		\item 
		Consider a pair $(G^+,G^-)$ of proper vector groups with the assumed operators $Q,T,P,R$.
		A pair of automorphisms $(a,b)$ of $G^+$ and $G^-$ is called an \textit{automorphism} if it preserves the operators, i.e.,
		$a(O_g h) = O_{ag} (bh), \quad b(O_hg) = O_{bh}(ag),$
		for $O = Q^{\text{grp}},T,P$ and $g \in G^+, h \in G^-$. To say that $Q^\text{grp}$ is preserved can be restated as $a_1(Q_gh) = Q_{(ag)}(bh)$ and $a_2(Q^\text{grp}_gh) = R_{a(g)}(bh)$, if one writes $a(g) = (a_1(g),a_2(g))$.
	\end{enumerate}
\end{definition}
\begin{definition}
	\label{Definition pre-Kantor pair}
	Let $(G^+,G^-)$ be a pair of vector groups such that
	\begin{equation}
		\label{Kantor pair equation}
		g(-g) \in [G^\pm,G^\pm] \text{ for } g \in G^\pm.
	\end{equation}
	We also assume that
	\begin{equation}
		\label{Kantor pair equation 2}
		G_2^\pm \longrightarrow \text{Hom}_\Phi(A^\mp,A^\pm): s \longmapsto Q_s \quad \text{is injective.}
	\end{equation}  
	We call $(G^+,G^-,Q^\text{grp}, T, P)$ a \textit{pre-Kantor} pair if
	\begin{enumerate}
		\item $1 + \epsilon V_{x,y}$ is an automorphism over $\Phi[\epsilon]$ with $\epsilon^2 = 0$
		\item \label{PKP LIN T} $T^{(1,2)}_{a,g} y = [a, Q_g y]$,
		\item \label{PKP LIN P} $P^{(2,1)}_{g,a} y = Q_gQ_{y^{-1}} a - V_{a,y} Q_g y$,
		\item $P^{(3,(1,1))}_g (y_1,y_2) = Q_{T_g y_1} y_2 - V_{g,y_2} Q_g y_1 $,
		\item $R^{(3,1)}_{g,a} y = T_g Q_{y^{-1}} a - V_{a,y} T_g y + \psi(Q_gy,Q^{(1,1)}_{g,a} y)$,
		\item $R^{(1,3)}_{a,g} y = [a, P_g y] + \psi(Q^{(1,1)}_{a,g} y,Q_g y)$,
		\item $R^{(4,(1,1))}_g(a,b) = - V_{g,b} T_g a + \psi(Q_g a,Q_g b)$.		
	\end{enumerate}
	We assume that these equalities hold over all scalar extensions.
\end{definition}

\begin{remark}
	If $1/2 \in \Phi$, then Condition (\ref{Kantor pair equation}) ensures that $G^\pm_2 = [G^\pm,G^\pm]$ since $$(0,2s) = (0,s)(0,s) = (0,s)(-1 \cdot_1 (0,s)) \in [G^\pm,G^\pm].$$
	Later we will see that this equation ensures the existence of a Lie algebra with grading element associated to a Kantor pair such that both $G^\pm$ are groups of exponentials (in the usual sense if $1/6 \in \Phi$ and in a more general sense if only $1/2 \in \Phi$).
	Conditions (\ref{Kantor pair equation}), (\ref{Kantor pair equation 2}) correspond thus, if $1/6 \in \Phi$, precisely to the conditions \cite[Proposition 7.5]{BENSMI03} for a Jordan-Kantor pair to come from a Kantor pair (using that the action of the Jordan part will be faithful if and only if $s \mapsto Q_s$ is injective on $G_2$).
\end{remark}

\begin{lemma}
	\label{lemma uniqueness P}
	Let $(G^+,G^-)$ be a pre-Kantor pair.
	Let $P'$ be a homogeneous map $G^+ \times G^- \longrightarrow A^+$ with the same linearisations as $P$, 
	then 
	$$ P_xy = P'_xy$$
	for all $x \in G^+, y \in G^-$.
	Furthermore, $P$ is determined uniquely by $Q$ and $T$.
	\begin{proof}
		Recall Theorem \ref{theorem universal homogeneous}, i.e.,   $g \mapsto g_n \in \mathcal{U}(G)_n$ is the universal homogeneous map of degree $n$, so that equalities in the universal representation induce equalities for all homogeneous maps.
		First, we use that $P$ is homogeneous of degree $3$ in the first argument, so that
		$$ 3 P_xy = 3 P^{(1,2)}_{x,x} y - P^{(1,1,1)}_{x,x,x}y,$$
		since
		$$ 3x_3 = 3x_1x_2 - x_1^3$$
		holds in the universal representation.
		Second, we use that $P$ is homogeneous of degree $2$ in the second component and that there exists $a_i,b_i$ such that $y(-y) = \sum_{i = 1}^n [a_i,b_i]$ by (\ref{Kantor pair equation}), to obtain
		$$ 2P_xy = P^{(3,(1,1))}_x(y,y) +  \sum_{i = 1}^n P^{(3,(1,1))}_x(a_i,b_i) - P^{(3,(1,1))}_x(b_i,a_i),$$
		since $2y_2 - y_1^2 = \sum_{i = 1}^n [a_i,b_i]$ in the universal representation.
		These two combined yield that
		$$ P_xy = (3 - 2)P_xy = 3 P^{(1,2)}_{x,x} y - P^{(1,1,1)}_{x,x,x}y - P^{(3,(1,1))}_x(y,y) - \sum_{i = 1}^n P^{(3,(1,1))}_x(a_i,b_i) - P^{(3,(1,1))}_x(b_i,a_i).$$
		Hence, $P$ is uniquely determined by its linearisations. So, we conclude that $P_xy = P'_xy$.
		For the furthermore part, note that the linearisations of $P$ given in the definition of a pre-Kantor pair are expressions only involving $Q, T,$ and linearisations of these operators ($V$ is defined in terms of these linearisations).
	\end{proof}
\end{lemma}

\begin{definition}
	Consider an automorphism $ 1 + \epsilon \delta = (1 + \epsilon \delta^+, 1 + \epsilon \delta^-)$ over $\Phi[\epsilon]$ with $\epsilon^2 = 0$, of a pre-Kantor pair. We call $\delta = (\delta^+,\delta^-)$ a \textit{derivation} of the pre-Kantor pair.  
	For derivations $\delta_1, \delta_2$ we define
	$[\delta_1,\delta_2]$ using
	$$ [1 + \epsilon_1 \delta_1, 1 + \epsilon_2 \delta_2] = 1 + \epsilon_1\epsilon_2 [\delta_1,\delta_2],$$
	over $\Phi[\epsilon_1,\epsilon_2]$ with $\epsilon_i^2 = 0$ for $i = 1,2$.
	On the universal representations of the vector groups $G^+, G^-$, we have
	$$ [\delta_1,\delta_2] = \delta_1\delta_2 - \delta_2\delta_1.$$
\end{definition}

\begin{lemma}
	\label{Lemma tau commutator}
	Let $(G^+,G^-)$ be a pre-Kantor pair.
	Take, $x,u,c,d \in G^+,y,v,a,b \in G^-$.
	The equations
	\begin{enumerate}
		\item $[V_{x,y},V_{u,v}] = V_{V_{x,y} u,v} - V_{u,V_{y,x} v},$
		\item $ \tau_{x,[a,b]} = V_{Q_x a, b} + [V_{x,a},V_{x,b}] - V_{Q_x b, a}$,
		\item $ \tau_{[c,d],y^{-1}} = V_{Q_y d, c} + [V_{y,d},V_{y,c}] - V_{Q_y c,d}$
	\end{enumerate}
	hold.
	\begin{proof}
		The first equation holds since $V$ is a derivation by item (3.a) of Definition \ref{Definition pre-Kantor pair}, $V_{x,y} = - V_{y,x}$ and since $V$ is entirely defined using linearisations of the operators.
		We note that the second and third equations are equivalent, since $\tau_{x,[a,b]} = - \tau_{[a,b],x^{-1}} = \tau_{[b,a],x^{-1}}$ holds by definition.
		We prove the second (and also third equation) in Appendix \ref{Lemma linerizations tau} by proving that $\tau^{(2,(1,1))}_{x,a,b} = V_{Q_xa,b} + V_{x,a}V_{x,b}$, which is sufficient since $\tau_{x,[a,b]} = \tau^{(2,(1,1))}_{x,a,b} - \tau^{(2,(1,1))}_{x,b,a}$ as can be proved by applying Lemma \ref{Lemma homogeneous map on commutator} to the homogeneous map $g \mapsto \tau_{x,g}$.
	\end{proof}
\end{lemma}

\begin{definition}
	Let $(G^+,G^-)$ be a pre-Kantor pair. For each invertible $\lambda$ in $K$ we have an automorphism of $(G^+_K,G^-_K)$, namely $(g \mapsto \lambda \cdot_1 g, g \mapsto \lambda^{-1} \cdot_1 g)$ by the assumptions on the degrees of the operators.
	Consider the automorphism associated to $(1 + \epsilon)$, and write its action on $G^+$ and $G^-$ additively as $(1 + \epsilon\zeta,1 + \epsilon \zeta)$. We see that $\zeta (g_1,g_2) = (\pm g_1, \pm 2 g_2)$ for $g \in G^\pm$.
	We call $$ \text{InStr}(G) = \langle V_{x,y} | x \in G^\pm, y \in G^\mp \rangle + \langle \zeta \rangle,$$
	with operation $[a,b]$ the \textit{inner structure algebra}\footnote{We do not call it the inner derivation algebra, as there could be some more morphisms which should fall under that name  if $1/2 \notin \Phi$, namely if $g$ or $t$ in $G_2$ then $\tau(g,t)$ should correspond to a derivation.} and $\zeta$ the \textit{grading element}.
	The previous lemma proves that derivations $\tau_{x,[a,b]}, \tau_{[a,b],x}$ are contained in InStr$(G)$. 
\end{definition}

\begin{lemma}
	The module $\text{InStr}(G)$ forms a Lie algebra with operation $[\cdot, \cdot]$.
	\begin{proof}
		For derivations $D,D'$ of $(G^+,G^-,Q^\text{grp},T,P)$ and arbitrary $x \in G^+, y \in G^-$, we compute
		$$ [D,[D', V_{x,y}]] - [D',[D,V_{x,y}]] = V_{(DD' - D'D)x,y} + V_{x,(DD' - D'D)y} = V_{[D,D']x,y} + V_{x,[D,D']y} = [[D,D'],V_{x,y}]$$
		using that $[E,V_{x,y}] = V_{Ex,y} + V_{x,Ey}$ for all derivations $E$.
		The Jacobi identity now follows for arbitrary triples $a + \lambda_a \zeta, b + \lambda_ b \zeta, c + \lambda_c \zeta$ of elements, with $a,b,c$ linear combinations of $V_{x,y}$'s and $\lambda_a, \lambda_b, \lambda_c \in \Phi$, since the bracket is alternating and linear.
	\end{proof}
\end{lemma}

We want to define a $5$-graded Lie algebra structure on
$$ L(G^+,G^-) = [G^-,G^-] \oplus A^- \oplus \text{InStr}(G^+,G^-) \oplus A^+ \oplus [G^+,G^+].$$
To achieve that, we endow $L(G^+,G^-)$ with the unique alternating operation $[\cdot,\cdot]$ such that
\begin{itemize}
	\item $A^+ \oplus [G^+,G^+]$ is a Lie subalgebra of $\text{Lie}(G^+)$, i.e.,
	$[(a,b),(c,d)] = [a,c] = \psi(a,c) - \psi(c,a),$
	\item $A^- \oplus [G^-,G^-]$ is a Lie subalgebra of $\text{Lie}(G^-),$
	\item $\text{InStr}$ embeds into $L(G^+,G^-)$ as a Lie algebra,
	\item $ [\delta, g] = \delta(g)$ for $g \in [G^\pm,G^\pm], \delta \in \text{InStr}(G)$,
	\item $[\delta,a] = \delta_1(a),$ for $ a \in A^\pm,\delta \in \text{InStr}(G)$ i.e.,
	we compute the action of $\delta$ by only looking at the first component so that 
	$$ [V_{x,y} ,a ] = -Q^{(1,1)}_{a,x} y,$$
	and 
	$$ [\zeta,a] = \pm a,$$
	with the sign corresponding to whether $a \in A^+$ or $A^-$,
	\item $[a,b] = V_{a,b}$ for $a \in A^+,b \in A^-$, 
	\item $[g,h] = \tau_{g,h}$ for $g \in [G^+,G^+]$, $h \in [G^-,G^-]$,
	\item $[g,b] = Q_{g} b$ for $g \in [G^\pm,G^\pm], b \in A^\mp$.
\end{itemize}
We remark that $[b,a] = V_{b,a} = -V_{a,b}$ and $[g,h] = \tau_{g,h} = -\tau_{h,g}$, so that $L(G^+,G^-) \cong L(G^-,G^+)$.

\begin{lemma}
	Consider a pre-Kantor pair $(G^+,G^-)$. The algebra $L(G^+,G^-)$ is a Lie algebra.
	\begin{proof}
		Set $L_+ = A^+ \oplus [G^+,G^+], L_0 = \text{InStr}(G^+,G^-)$ and $L_- = A^- \oplus [G^-,G^-].$
		Since $L_0$ is a Lie algebra that acts as derivations on the Lie algebra $L_+$, we know that $$ L_0 \oplus L_+$$
		forms a Lie algebra. 
		
		We only need to check the Jacobi identity.	As we check this identity we can assume without loss of generality that $a,b \in L_+ \cup L_0$ and that $c \in L_-$.
		Suppose that $a, b \in L_0$, then we are already finished since we know that $L_- \oplus L_0$ forms a Lie algebra.
		So, suppose that $b \in L_0$ and $a \in L_+$.
		Note that all automorphisms of the pre-Kantor pair preserve the Lie bracket of $L$, so that 
		$$ (1 + \epsilon b)[a,c] = [(1 + \epsilon b)a,(1 + \epsilon b)c],$$
		for $b \in L_0$. This proves the Jacobi identity if $b \in L_0$.
		
		So, we can assume that $a,b \in L_+$, $ c \in L_-$.
		
		Suppose first that $c \in A^-$
		For $a,b \in A^+$, the Jacobi identity follows from  $$[[a,b],c] = Q_{[a,b]} c = Q^{(1,1)}_{a,b}c - Q^{(1,1)}_{b,a}c = - V_{b,c} a + V_{a,c} b = [[a,c],b] + [a,[b,c]]$$
		by using Lemma \ref{Lemma homogeneous map on commutator} for the second step.
		If $a, b \in A^+ \cup [G^+,G^+]$ but not both in $A^+$, then we need to check that
		$$ [a,[b,c]] = [b,[a,c]].$$
		Assume that $a \in A^+, b \in [G^+,G^+]$. So, must check \[[a,Q_b c] = - V_{a,c} b.\]		
		This is equivalent to checking that
		\[ T^{(1,2)}_{a,b} c = [a,Q_b c] = - V_{a,c} b = T^{(2,1)}_{b,a},\]
		using Definition \ref{Definition pre-Kantor pair}.\ref{PKP LIN T} for the linearization of $T$ and Definition \ref{definition V tau} for $V$.
		The previous equation holds since $T_{(t \cdot_1 g)b} = T_{b(t \cdot_1 g)}$ for all $g \in G(\Phi[t])$ with first coordinate $a$ since $b \in G_2$.
		If both $a,b \in [G^+,G^+]$, then the Jacobi-identity is trivially satisfied by the grading.
		
		Assume now that $c \in [G^-,G^-]$. If both $a,b \in A^+$, then Lemma \ref{Lemma tau commutator} shows that the Jacobi identity holds, since
		$$ [c,[a,b]] = \tau_{c,[a,b]} = V_{Q_ca,b}- V_{Q_c b,a} = [[c,a],b] - [[c,b],a].$$
		Lastly, if both $a,b \in A^+ \cup [G^+,G^+]$ and if it is not the case that both are contained in $A^+,$ then we must show that
		$$ [a,[b,c]] = [b,[a,c]],$$
		which is the case if both $a,b \in [G^+,G^+]$ since the previous equation becomes
		$$ - \tau_{b,c} a = R^{(2,2)}_{a,b} c = R^{(2,2)}_{b,a} c = - \tau_{a,c} b,$$
		using Definition \ref{definition V tau} for $\tau$, which holds since $\Phi[t]$-scalar multiples of $a$ and $b$ commute. 
		So, suppose that $a \in A^+$ and $b \in [G^+,G^+]$, then the equation reduces to
		$$ P^{(1,2)}_{a,b} c = -Q_bQ_ca,$$
		which is the case since $P^{(1,2)}_{a,b} c = P^{(2,1)}_{b,a} c = - Q_bQ_c a$ using Definition \ref{Definition pre-Kantor pair}.\ref{PKP LIN P}.
	\end{proof}
\end{lemma}

We are almost finished with our exposition on pre-Kantor pairs. Before we end the exposition, we first establish a link with Kantor pairs as they appear in the literature and secondly construct a pair of vector group representations of pre-Kantor pair in the endomorphism algebra of the associated Lie algebra. 

\begin{definition}
	Suppose that $(A^+,A^-)$ is a pair of modules over $\Phi$ and that $W : A^+ \times A^- \longrightarrow \End(A^+) \times \End(A^-)$ is a linear map. We write $V_{a,b} c$ for $W_{a,b} c$ whenever $a, c \in A^+$ and $b \in A^-$ and write $V_{b,a} c$ for $- W_{a,b} c$ whenever $a \in A^+$ and $b, c\in A^-$. Furthermore, we write $K_{a,b} c$ for $V_{a,c} b - V_{b,c} a$.
	We call $(A^+,A^-, W)$ a \textit{Kantor pair} if
	\begin{enumerate}
		\item $[ V_{a,b}, V_{c,d}] = V_{V_{a,b} c,d} - V_{c, V_{b,a} d}$ if $a, c \in A^\epsilon$ and $b, d\in A^{-\epsilon}$,
		\item $V_{a,b} K_{c,d} + K_{c,d} V_{b,a} = K_{K_{c,d} b, a}$ if $a,c,d \in A^\epsilon$, $b \in A^{-\epsilon}$,
	\end{enumerate}
using $\epsilon$ to denote an arbitrary sign.
	This definition is equivalent with the one appearing in \cite[section 3]{ALLFLK99}. 
\end{definition}

\begin{lemma}
	\label{lemma preKantor induces Kantor pair}
	Consider a pre-Kantor pair $(G^+,G^-,Q^\text{grp},T,P)$. Define $A^\pm = G^\pm/G_2^\pm$, the triple $(A^+,A^-,V_{|A^+ \times A^-})$ forms a Kantor pair. Moreover, if $1/6 \in \Phi$ then $(G^+,G^-,Q^\text{grp},T,P)$ can be uniquely recovered from $(A^+,A^-,V_{|A^+ \times A^-})$.
	\begin{proof}
		This immediately follows from the fact that $L = L(G^+,G^-)$ is a Lie algebra. More precisely, since $L$ is a $\mathbb{Z}/2\mathbb{Z}$-graded algebra with as $1$-graded part $A^+ \oplus A^-$, we can conclude that $A^+ \oplus A^-$ is a Lie triple system with operation $[a^\epsilon,b^{-\epsilon},c^\epsilon] = V_{a,b}c,\; [a^\epsilon,b^{\epsilon},c^\epsilon] = 0$ where $x^\epsilon$ is contained in $A^\epsilon$ for $x = a,b,c$. Allison and Faulkner proved \cite[Equations (2) and (3)]{ALLFLK99}\footnote{There is an (implicit) additional sign there, which is no problem since $W$ corresponds to a Kantor pair if and only if $-W$ does. The authors assume that $1/2 \in \Phi$ in the beginning of the section in which they prove the theorem. This assumption is irrelevant for the interpretation of these equations.} that this implies that $(A^+,A^-,V)$ is a Kantor pair.
		
		If $1/6 \in \Phi$, then we know that $G^+ \cong A^+ \oplus \{ a \mapsto V_{x,a}y - V_{y,a} x | x,y \in A^+\}$ with group operation determined by $$\psi(a,b) = c \mapsto (V_{a,c}b - V_{b,c} a)/2$$ since $G_2 = [G,G]$, $G_2 \cong Q_{G_2}$ as modules, and since $Q$ must satisfy $$Q_{[a,b]} c = Q^{(1,1)}_{a,b} c - Q^{(1,1)}_{b,a} c = V_{a,c} b - V_{b,c} a.$$
		The operators $Q,T,P$ and $R$ are determined by their linearisations.
		To be precise, if $x(-x) = \sum_{i = 1}^n [a_i,b_i],$ $ y(-y) = \sum_{i = 1}^m [c_i,d_i]$, then we know that
		\begin{enumerate}
			\item $2 Q_x y = Q^{(1,1)}_{x,x} y + \sum_{i = 1}^n Q^{(1,1)}_{a_i,b_i}y - Q^{(1,1)}_{b_i,a_i} y,$
			\item $ 3 T_x y = 3 T^{(1,2)}_{x,x} y - T^{(1,1,1)}_{x,x,x} y,$
			\item $ 2 P_x y =  P^{(3,(1,1))}_x(y,y) + \sum_{i = 1}^m P^{(3,(1,1))}_x(c_i,d_i) - P^{(3,(1,1))}_x(d_i,c_i),$
			\item $ 2 R_x y = R^{(4,(1,1))}_x(y,y)+ \sum_{i = 1}^m R^{(3,(1,1))}_x(c_i,d_i) - R^{(3,(1,1))}_x(d_i,c_i)$
		\end{enumerate}
		since $2x_2 = x_1^2 + \sum_{i = 1}^n (a_i)_1(b_i)_1 - (b_i)_1(a_i)_1,$ $2y_2 = y_1^2 + \sum_{i = 1}^m (c_i)_1(d_i)_1 - (d_i)_1(c_i)_1$, and $3x_3 = 3x_1x_2 - x_1^3$ hold for the universal homogeneous maps of degree $2$ and $3$ on vector groups.
		Using $V_{x,y} z = - Q^{(1,1)}_{z,x} y$ and the linearisations of the operators for pre-Kantor pairs ($T^{(1,1,1)}$ can be obtained by linearising $T^{(1,2)}$), these equations uniquely determine the pre-Kantor pair.
	\end{proof}
\end{lemma}

We define an action of $G^\pm$ on $L(G^+,G^-)$.
Namely, for $g = (a,b) \in G^\pm $ we define
$$ \exp(g) = 1 + g_1 + g_2 + g_3 + g_4,$$
using the following endomorphisms of $L(G^+,G^-)$:
\begin{align*}
	g_1 \cdot z & =  [a,z] \\
	g_2 \cdot z & = \begin{cases} \tau_{g,z} & z \in [G^\mp,G^\mp]  \\
		Q_g z & z \in A^\mp \\
		\hat{z}(g) & z \in \text{InStr}(G^+,G^-) \\
		0 & \text{otherwise}
	\end{cases}, & \text{where } \hat{z}(g) = - (z(g))_2 + \psi(z(g)_1,a)\\
	g_3 \cdot z & = \begin{cases} P_g z & z \in [G^\mp,G^\mp] \\
		T_g z & z \in A^\mp \\
		0 & \text{otherwise} 
	\end{cases}, \\
	g_4 \cdot z & = \begin{cases} R_g z  & z \in [G^\mp,G^\mp]\\
		0 & \text{otherwise} 
	\end{cases}.
\end{align*}
Remark that $(\epsilon z(g)_1,-\epsilon \hat{z}(g)) = ((1 + \epsilon z)(g))g^{-1}$ for all derivations $z$.
\begin{remark}
	Observe that $g_2 \cdot \zeta \in [G^+,G^+]$ for $g \in G^+$ since $g(-g) = (\zeta(g))_2 - \psi(\zeta(g),g)$. We also note that $R_g z \in [G^+,G^+]$ for $z \in [G^-,G^-]$. This can be observed by considering
	$$ R_g [a,b] = R^{(4,(1,1))}_{g}(a,b) - R^{(4,(1,1))}(b,a)$$
	and the $(4,(1,1))$ linearisation of $R$ given in the definition of a pre-Kantor pair.
\end{remark}

In what follows, we will often write $L$ to denote the Lie algebra $L(G^+,G^-)$ if $G$ is clear from the context.

\begin{lemma}
	The maps $(G \longrightarrow \text{End}_\Phi(L) : g \mapsto g_i)_i$ form a vector group representation, with $g_i = 0$ for $i > 4$.
	\begin{proof}
		It is sufficient to prove that $g \mapsto g_i$ is a homogeneous map with $(k,l)$-linearisation $(g,h)_{k,l} = g_kh_l$.
		Since the $g_i$ are defined using homogeneous maps, we only need to prove that the $(k,l)$-linearisations are what they should be.
		
		We already know that $g \mapsto g_1$ is linear.
		For $g_2$ we check whether the $(1,1)$-linearisation evaluated in $(a,b) \in A^+ \times A^+$ applied to the definition coincides with $a_1b_1$. So, we check whether
		\begin{align*}
			\tau^{((1,1),2)}_{a,b,z}&  = - V_{a,Q_z b}, \\
			Q^{(1,1)}_{a,b} z & = - V_{b,z} a, \\
			-(z(a,b))_{(1,1)} + \psi(z(a),b) + \psi(z(b),a) & = [a,z(b)],
		\end{align*}
		i.e., we check all the nontrivial cases in the definition of $g_2 \cdot z$.
		The first case holds by Lemma \ref{Lemma linerizations tau}, which is proved in the appendix, and since $z \in G_2$.
		The second case holds by definition. The last case follows from $(z(a,b))_{(1,1)} = \psi(z(a),b) + \psi(a,z(b))$, which can be proved using that
		$ (1 + \epsilon z)$ is an automorphism.
		
		Most linearisations of $g_3$ and $g_4$ are directly observed to be what they should be, as can be seen from the linearisations of $T, \; P$ and $R$ that are fixed in the definition pre-Kantor pairs. The other linearisations play a role in the definition of $V$ and $\tau$ involving the linearisations of $T,\;P$ and $R$. In these cases a direct computation shows that the linearisations are what they should be.
		For example, we can compute for $T$ that
		$$ T^{(2,1)}_{g,x}(y) = (V_{y,x} g)_2 - \psi(V_{y,x}g,g) = - g_2 \cdot V_{y,x} = g_2x_1 \cdot y.$$
		The homogeneous maps $g_n = 0$ for all $n > 4$ have trivial linearisations, as implied by the grading of the Lie algebra. 
	\end{proof}
\end{lemma}

\subsection{Operator Kantor pairs}

\begin{definition}
	We call a pre-Kantor pair $(G^+,G^-)$ an \textit{operator Kantor pair} if
	\begin{enumerate}
		\item $ \tau_{x,T^{(2,1)}_{y,a}(x)} +  V_{x,Q_yQ_{x^{-1}}a} + \tau_{y,T_{x^{-1}}a} = V_{P_xy,a} + V_{Q_{T_xy}a,y} - [V_{x,y},V_{Q_xy,a}],$
		\item   \label{operator Kantor pair uniqueness R}	$ P_xT^{(2,1)}_{y,a}(x) + Q_xQ_yQ_{x^{-1}}a - \tau_{y,T_{x^{-1}}a}x = Q_{Q^{\text{grp}}_xy} a + V_{x,y}Q_{T_xy} a,$
		\item
		$ Q^\text{grp}_x T^{(2,1)}_{y,a}(x) + T_xQ_yQ_{x^{-1}}a - (\tau_{y,T_{x^{-1}}a}(x))_2 + \psi(\tau_{y,T_{x^{-1}}a}(x),x)  =  [Q_x y, Q_{T_xy} a],$
		\item $ T_{Q^{\text{grp}}_x y} a + [P_x y , Q_{T_x y} a ] + V_{x,y} [Q_x y, Q_{T_xy} a] = R_x T_y Q_{x^{-1}} a+ T_x P_y T_{x^{-1}} a,$
	\end{enumerate}
	hold for all $x \in G^\epsilon(K), y \in G^{-\epsilon}(K), a \in A^{-\epsilon} \otimes K$.
	The first three equations express, using the representations $g \mapsto g_i \in \text{End}_\Phi(L)$ and associated operators $\mu_{i,j}, \nu_{i,j}$, precisely that
	\begin{enumerate}
		\item $ \mu_{3,2}(x,y)$ acts as $a \mapsto V_{P_xy,a}  + V_{{Q_{T_xy}a},y} - [V_{x,y},V_{Q_xy,a}] $ on $A^{-\epsilon}$,
		\item  $ \mu_{4,2}(x,y)$ acts as $a \mapsto Q_{Q^{\text{grp}}_xy} a + V_{x,y}Q_{T_xy} a$  on $A^{-\epsilon}$,
		\item $\mu_{5,2}(x,y)$ acts as $a \mapsto [Q_x y, Q_{T_xy} a]$ on $A^{-\epsilon}$.
	\end{enumerate}
	The fourth equation expresses that $\nu_{6,3}(x,y)$ acts as $a \mapsto T_{Q^\text{grp}_x y} a$ on $A^{-\epsilon}$, if $\mu_{4,1}(x,y) = 0,\; \nu_{2,1}(x,y) = \text{ad} \; Q_xy, \;\nu_{1,1}(x,y) = \text{ad} \; V_{x,y},  \;\nu_{3,1}(x,y) = \text{ad} \; T_x y,$ and $\nu_{3,2}(x,y) = \text{ad} \; P_x y$ hold.
\end{definition}

\begin{remark}
	If $1/30 \in \Phi$, then each pre-Kantor pair is an operator Kantor pair. We prove this in Corollary \ref{bijection Kp preKP opKP}. 
\end{remark}

\begin{remark}
	\label{remark sufficient condition}
	A sufficient condition, which will be easier to use in practice, for the previous equations to hold, is given by the following list of properties:
	\begin{enumerate}
		\item $\exp(g)$ is an automorphism for $g \in G^+$ and $G^-$,
		\item $ \nu_{3,2}(x,y) = \text{ad} \; P_x y,$
		\item $ \nu_{4,2}(x,y) = (Q^{\text{grp}}_x y)_2,$
		\item $ \nu_{5,2}(x,y) = 0,$
		\item $ \nu_{6,3}(x,y) = (Q^{\text{grp}}_x y)_3.$
	\end{enumerate}
	These conditions are sufficient since one obtains the axioms of an operator Kantor pair by (1) expressing the $\nu_{i,j}$ as a sum of products $\mu_{k,l}$ and (2) evaluating the endomorphisms on $A^-$.
	The $\mu_{i,j}$ that must be of a specific form to obtain the precise equation of the definition, are $\mu_{1,1}(x,y) = \text{ad} \; V_{x,y},\; \mu_{2,1}(x,y) = \text{ad} \; Q_x y$, $\mu_{3,1}(x,y) = \text{ad} \; T_x y$ and $\mu_{4,1}(x,y) = 0$.
	Using the assumption that $\exp(g)$ is an automorphism one sees that the equation $\mu_{i,1}(x,y) = \text{ad} \; \exp(x)_i(y)$ always holds.
	We remark that these conditions are not only sufficient but necessary as well. In the next lemma we prove that $\exp(g)$ is an automorphism and we argue in the second part of Theorem \ref{thm weights lie algebra}, in which we prove that we can apply Theorem \ref{thm main}, that the other equations hold.
\end{remark}

\begin{remark}
	\label{Remark uniqueness P and R}
	For operator Kantor pairs, the operators $P$ and $R$ can be determined from $Q$, $T$ and $V$.
	We already proved this for $P$ in Lemma \ref{lemma uniqueness P}.
	We want to prove the it for $R$. We determine $R$ for all flat $\Phi$-algebras $K$ and note that this lets us determine the value for all $R_x y$ over all $\Phi$-algebras using naturality.
	 So, consider the set $S = \{ (Q_gh, t) \in G(K)\}$ for flat $K$.
	There exists a $k \in S$, namely $ (Q^{\text{grp}}_x y)$, which satisfies
	\[ k_2 \cdot a = \nu_{4,2}(x,y) \cdot a = (\mu_{4,2}(x,y) - \mu_{1,1}(x,y) \mu_{3,1}(x,y)) \cdot a\]
	for all $a \in A^-$, using that $ \nu_{4,2}(x,y) = (Q^{\text{grp}}_x y)_2$.
	This $k$ is unique by Equation (\ref{Kantor pair equation 2}) and the flatness of $K$, since
	\[ (k(0,s))_2 \cdot a = k_2 \cdot a + Q_{(0,s)} \cdot a.\]
	So, if
	\[ (\mu_{4,2}(x,y) - \mu_{1,1}(x,y) \mu_{3,1}(x,y)) \cdot a \]
	is an element of $A^+$ which can be expressed using $Q$, $T$, $P$, $V$, then we know that $R$ is uniquely determined by the other operators.
	Evaluating this normally yields such an expression if one substitutes $ - \tau_{y,T_{x^{-1}} a} x = P^{(1,2)}_{x, T_{x^{-1}} a} y^{-1},$ for the contribution of
	$ x_1y_2x^{-1}_3 \cdot a$ in $\mu_{4,2}(x,y) \cdot a$.
\end{remark}

\begin{lemma}
	Let $(G^+,G^-)$ be an operator Kantor pair and take $g \in G^\pm$. The map $\exp(g)$ is an automorphism. 
	\begin{proof}
		We note that $\exp(g)$ is an automorphism if 
		$$ \exp(g)[a,b] = [\exp(g)a,\exp(g)b],$$
		for all $a,b \in L$.
		Suppose that $g \in G^+$, the previous identity holds if either one of $a$ and $b$ is an element of $L_+$. We illustrate this if $a \in L_+$. In that case, the equation $\exp(g)(\text{ad} \; a )\exp(g^{-1}) = \text{ad} \; \exp(g) a$ holds since we have a vector group representation, which proves that $ \exp(g)[a,b] = [\exp(g)a,\exp(g)b].$
		
		Similarly, using that $1 + \epsilon l$, with $\epsilon^2 = 0$, is an automorphism of the operator Kantor pair for $l \in L_0$ and that $\exp(g) \cdot \epsilon l = \epsilon l - ( (1 + \epsilon l)(g)\cdot g^{-1})$ holds by definition of the action, we conclude that $\exp(g)$ interacts nicely with brackets involving an $l \in L_0$ by observing that
		$$ \epsilon l - \text{ad} \; \exp(g) \cdot \epsilon l= \exp((1 + \epsilon l) g) \exp(g^{-1}) - 1= (1 + \epsilon l)\exp(g)(1 - \epsilon l)\exp(g^{-1}) - 1 = \epsilon l - \epsilon \exp(g) l \exp(g^{-1}),$$
		as maps on $L$ (remark that we can write $l$ for $\text{ad} \; l$ if we interpret $l$ as an element of $\text{InStr}(G)$).
		
		So, the only thing left to prove is that
		$$ \exp(g)[a,b] = [\exp(g)a,\exp(g)b],$$
		for $a,b \in L_-$.
		
		If $a,b \in A^-$ this follows from the Jacobi identity, the $(n,(1,1))$-linearisations $P,R$ of Definition \ref{Definition pre-Kantor pair} using Lemma \ref{Lemma homogeneous map on commutator}, and the equation
		$$ \tau_{x,[a,b]} = V_{Q_xa,b} + V_{a,Q_xb} + [V_{x,a},V_{x,b}]$$
		proved in Lemma \ref{Lemma tau commutator}.
		
		Now, we assume that $a \in [G^-,G^-]$ and $b \in A^-$.
		We will prove that 
		\begin{equation}
			\label{eq ad a}
			\exp(g)[a,\exp(g^{-1})(b)] = [\exp(g)(a),b],
		\end{equation} which proves that 
		\begin{equation}
			\label{eq ad b}
			\exp(g)(\text{ad} \; \exp(g^{-1})b)\exp(g^{-1}) = \text{ad} \; b
		\end{equation} holds. Using that we already proved that $\exp(g)$ conjugates as expected with $$\exp(g^{-1})(b) - b \in \text{InStr}(G) \oplus A^+ \oplus [G^+,G^+]$$
		we will be able to conclude from Equation (\ref{eq ad b}) that $$\exp(g)(\text{ad} \; b)\exp(g^{-1}) = \text{ad} \; \exp(g)(b).$$
		So, if we prove Equation (\ref{eq ad a}), then we are able to conclude that
		$\exp(g)(\text{ad} \; b)\exp(g^{-1}) = \text{ad} \; \exp(g)(b)$ for all $b \in A^- \cup A^+ \cup \zeta$, i.e., a generating set of $L$. This proves that $\exp(g)(\text{ad} \; c)\exp(g^{-1}) = \text{ad} \; \exp(g)(c)$ for all $c \in L$, since we can write $\text{ad} \; c$ as a polynomial function of elements $\text{ad} \; b$ for which it holds.
		So, proving Equation (\ref{eq ad a}) proves the lemma, without needing to consider the case where $a,b \in [G^-,G^-]$.
		So, we try to prove Equation (\ref{eq ad a}).
		The first $3$ axioms for operator Kantor pairs show that $\exp(g)a_2\exp(g^{-1}) = \text{ad} \; \exp(g)(a)$ on $A^-$ if and only if
		$$ \tau_{g,a} = g_2a_2 - g_1a_2g_1 + a_2(g^{-1})_2$$ on $A^-$.
		Applying the definition of $\tau$ and the action of the elements $g_i$ and $a_2$ on the Lie algebra, shows us that we need to prove that
		\[ -P^{(1,2)}_{b,a}(g^{-1})  = \tau(g,a)(b) = (g_2a_2 - g_1a_2g_1 + a_2(g^{-1})_2) \cdot b = Q_aQ_{g^{-1}}b + Q_{V_{b,g}a} g.\]
		Using that $P^{(2,1)}_{a,b}(g^{-1}) = Q_aQ_g b - V_{b,g}Q_a g$ by Definition \ref{Definition pre-Kantor pair}.\ref{PKP LIN P} and that $P^{(1,2)}_{b,a} = P^{(2,1)}_{a,b} + P^{(2,1)}_{[b,a],a}$ as any homogeneous map of degree $3$, transforms what we need to show into
		\[ Q_a(Q_g + Q_{g^{-1}})b - V_{b,g}Q_ag + Q_{V_{b,g}a}g + Q_{[b,a]}Q_ga - V_{a,g}Q_{[b,a]}g = 0.\]
		Using that $(Q_g + Q_{g^{-1}})(b) = V_{b,g} g$ since $g_2 + (g^{-1})_2 = - g_1^2$, and that $[Q_a,V_{b,g}]g = -Q_{V_{b,g}a}g$ for $a \in G_2$ since $V_{b,g}$ is a derivation, reduces the equation to
		$$ V_{a,g}Q_{[a,b]}g = Q_{[a,b]}Q_ga$$
		which holds since $a \in [G^-,G^-]$ so that $V_{a,g} = 0$ and $Q_{[a,b]} = 0$. This proves that Equation \ref{eq ad a} holds and thus the lemma.
	\end{proof}
\end{lemma}

\begin{theorem}
	\label{thm weights lie algebra}
	Let $(G^+,G^-)$ be an operator Kantor pair and let $G^+$ and $G^-$ be projective. The $\Phi$-group functor $K \longrightarrow G(K) = \langle G^+(K),G^-(K) \rangle \le \text{End}_\Phi(L(G^+,G^-)) \otimes K$ has as corresponding Lie algebra $$L = \text{Lie}(G^-) \oplus L_0 \oplus \text{Lie}(G^+)$$ for some Lie algebra $L_0$. This Lie algebra is $5$-graded by the weights of the action of $\Phi_m(K) = K^\times$ defined by $\lambda \cdot (g_+,g_-) = (\lambda \cdot_1 g_+, \lambda^{-1} \cdot_1 g_-)$ of $\Phi_m$  on $(G^+,G^-)$ and $L_0$ is the $0$ graded component of this action. Moreover, if $E = 1 + e_1 + e_2 + e_3 + e_4 \in G(K)$ with $e_i(L_j) \subset L_{i+j}$ for all $i,j$, then $E$ must be an element of $G^+(K)$. 
	\begin{proof}
		We want to apply Theorem \ref{thm main}. 
		Suppose that we can use the conclusions of the theorem.
		This proves that the vector group representations of $G^+,G^-$ factor through some $\mathbb{Z}$-graded bialgebra $H = \mathcal{U}(G^-)\mathcal{H}\mathcal{U}(G^+)$ such that each $h \in H$ can be uniquely written as $\sum_{i = 1}^n u_{i}h_iv_i$ with all $u_i \in \mathcal{U}(G^-), v_i \in \mathcal{U}(G^+)$ and all $h_i$ contained in a $0$-graded Hopf subalgebra $\mathcal{H}$.
		Furthermore, $G(K)$ can be constructed by a quotient of the group $I(K) = \langle \rho^+_{[t]}(G^+(K)), \rho^-_{[t]}(G^-(K)) \rangle$ defined by embedding the universal representations $\mathcal{U}(G^\pm)$ in $H$.
		Specifically, take the unique map defined by $$1 + tg_1 + t^2g_2 + \ldots \longmapsto 1 + g_1 + g_2 + g_3 + g_4$$ on the generators $\rho^+_{[t]}(G^+(K))$ and $\rho^-_{{t}}(G^-(K))$ which maps $I(K) \longrightarrow G(K)$.
		Suppose that $1 + \epsilon d \in \text{Lie}(G)$. We can lift this to an element $F$ of $I(K[\epsilon])$. Furthermore, we can assume that $F$ is of the form
		$$ 1 + \epsilon f_1 + \epsilon f_2 + \epsilon f_3 + \ldots,$$
		using the fact that the image of $F$ under $\epsilon \mapsto 0$ lies in $I(K).$
		Since all the generators of $I(K)$ are group like, i.e., satisfy $\Delta(i) = i \otimes i, \epsilon(i) = 1$, we conclude that all $f_i$ satisfy $\Delta(f_i) = f_i \otimes 1 + 1 \otimes f_i$. Lastly, note that $g_k \mapsto 0$ for all $k > 4$ and all $g \in G^+,G^-$, i.e., $1 + ti_1 + t^2i_2 + \ldots \in I(K)$ implies that there exists some $N$ such that $i_k \mapsto 0 $ if $k \ge N$.
		So, we conclude that there exists a primitive element 
		$ \sum_{i = 1}^N f_i$ which maps to $d$.
		Since each primitive element of $H$ can be uniquely written as a sum $\sum u_{i}h_iv_i$, it is not hard to check that each primitive element can be uniquely decomposed as $p = u_1 + h_2 + v_3$ with $u_1,h_2,v_3$ primitive using $(\pi_{\mathcal{U}(G^-)} \otimes \pi_\mathcal{H} \otimes \pi_{\mathcal{U}(G^+)})\Delta^2(yhx) = y \otimes h \otimes x$.
		If we assume that $d$ has no $0$-graded component, then we can conclude that $d \in \text{Lie}(G^+) \oplus \text{Lie}(G^-)$.
		
		If the conditions of Theorem \ref{thm main} apply, then one also proves the moreover part in a similar fashion. Namely, one can find a lift for $E$ in $I(K)$ using the fact that all group-like elements $h$ in $H$ such that $h - 1$ has only positively graded components, must lie in $\mathcal{U}(G^+)$ and then apply Theorem \ref{theorem PUG} to conclude that $h \in G^+(K)$.
		
		So, now we prove that the conditions of Theorem \ref{thm main} apply.
		It is sufficient to prove that 
		$$ \exp(O^{i1}_xy) = 1 + \sum_{k = 1}^4 \nu_{(ki),k}(x,y),$$
		for $O^{21} = Q^{\text{grp}}$ and $O^{31} = T$, that
		$$ \nu_{4,1}, \nu_{5,1}, \nu_{5,2}, \nu_{7,3} = 0$$
		and that
		$$ \nu_{3,2}(x,y) =\text{ad} \; P_x y$$
		for $x \in G^+$ and $y \in G^-$ (or with the roles of $x$ and $y$ reversed) and $l \in A^+$.
		To simplify this task, note that there are no nontrivial derivations $d: L(G^+,G^-) \longrightarrow L(G^+,G^-)$ so that $d(L(G^+,G^-)_i) \subset L(G^+,G^-)_{i + k}$ with $|k| \ge 4$ since $L(G^+,G^-)$ is generated by $A^+,A^-$ and $\zeta$.
		Furthermore, one computes, using the fact that $\exp(g)$ is an automorphism for all $g$, that 
		$$ \exp(o_{i,j}(x,y)) = \sum_{k = 0}^\infty \nu_{ki,kj}(x,y),$$
		is an automorphism as well (similar to Lemma \ref{lemma grouplike}), proving, in particular, that the $\nu_{i,j}$ with $|i - j| \ge 4$ are uniquely determined by the $\nu_{k,l}$ with $|k - l| < 4$ and $i/j = k/l$, as all $\nu_{i,j}$ are determined by those smaller $\nu_{k,l}$ up to a derivation.
		The previous observations reduce what we need to prove to
		\begin{enumerate}
			\item $\nu_{2,1}(x,y) = \text{ad} \; Q_x y$
			\item $\nu_{4,2}(x,y) = (Q^\text{grp}_x y)_2$,
			\item $\nu_{6,3}(x,y) = (Q^\text{grp}_xy)_3$
			\item $\nu_{3,1}(x,y) = \text{ad} \; T_xy$
			\item $\nu_{4,1}(x,y) = \nu_{5,2}(x,y) = 0$
			\item $\nu_{3,2}(x,y) = \text{ad} \; l = \text{ad} \; P_x y$.
		\end{enumerate}
		For $\nu_{i,1}(x,y) = \mu_{i,1}(x,y)$ this follows from the fact that $\exp(x)$ is an automorphism. 
		Note that $\nu_{5,2}(x,y)$ is a derivation on $L(G^+,G^-)$ which acts as $+3$ on the grading. So, the action given in the definition of an operator Kantor pair determines $\nu_{5,2}(x,y)$ uniquely since that the third axiom for operator Kantor pairs implies that $\mu_{5,2}(x,y) - \nu_{2,1}(x,y)\nu_{3,1}(x,y) = 0$ on $A^-$. Similarly, the fourth axiom guarantees that $\nu_{6,3}(x,y) = (Q^\text{grp}_xy)_3$ if we can prove that the other equations hold.
		So, we only need to consider $\nu_{4,2}(x,y)$ and $\nu_{3,2}(x,y)$.
		The element $\nu_{4,2}$ acts trivially on $A^+$ and acts as expected on $A^-$. Evaluating on the grading derivation gives us an element that acts as an inner derivation of the Lie algebra. This evaluation yields $\nu_{4,2}(x,y) \cdot \zeta$ which is an element that acts like the endomorphism $-2\nu_{4,2}(x,y) + \nu_{2,1}(x,y)^2$. This last endomorphism acts exactly as $(0,-2R_xy + \psi(Q_xy,Q_xy))_2$ on $A^-$, which proves that $\nu_{4,2}(x,y)$ acts as $(Q^\text{grp}_x y)_2$ since $G^+_2$ acts faithfully on $A^-$.
		
		Note that 
		$$ \nu_{3,2}(x,y) \cdot \zeta = - l \in A^+.$$
		Using the fact that $\nu_{3,2}$ is a derivation, we conclude $\text{ad} \; l = \nu_{3,2}(x,y)$.
		We know that $(x,y) \mapsto  - \nu_{3,2}(x,y) \cdot l$ is a homogeneous map with the same linearisations, by Lemma \ref{Lemma equations}, as $P$, hence $l = P_xy$ by Lemma \ref{lemma uniqueness P}. 
		This finishes the proof.
	\end{proof}
\end{theorem}

\begin{remark}
	It could very well be true that the previous theorem holds without the assumption that $G^+$ and $G^-$ are projective.
	One can, however, also drop the projective assumption in order to get a weaker result. Namely, if there exists an element $E \in G(K)$ of the form of the previous theorem which is not contained in $G^+(K)$, then it still corresponds to a group-like element in a Hopf quotient $U^+$ of the Hopf algebra $\mathcal{U}(G^+)$. Similarly, the positively graded elements of the Lie algebra $\text{Lie}(G)$ that are not contained in $\text{Lie}(G^+)$, correspond to primitive elements of $U^+$.
\end{remark}

\begin{theorem}
	\label{thm main 2}
	Suppose that $\rho^\pm : G^\pm \longrightarrow A$ are vector group representations of proper vector groups such that for all $x \in G^\pm, y \in G^\mp$,
	\begin{itemize}
		\item $\exp(o_{2,1}(x,y),s) \in \rho^\pm_s(G^\pm)$, 
		\item $\exp(o_{3,1}(x,y),s^2) \in \rho^\pm_s([G^\pm,G^\pm])$ ,
		\item $\exp(o_{i,j}(x,y),s) = 1$ for $i > 2j$ and $i \neq 3j,$
		\item there exists $z \in G^\pm$ such that $\rho_s(z) = 1 + s\nu_{3,2}(x,y) + O(s^2)$.
	\end{itemize}
	Set $K^\pm = \{(\rho^\pm_1(g),\rho_2^\pm(g)) \in A \times A | g \in G^\pm\}.$
	Then $$(K^+,K^-, Q^\text{grp} = o_{2,1}, T = o_{3,1}, P = \nu_{3,2})$$ forms an operator Kantor pair if and only if $\rho_2(G_2) \ni (0,s) \mapsto Q_{(0,s)}(\cdot)$ is injective and $\rho^\pm(g(-g)) \in \rho^\pm[G^\pm,G^\pm]$.
	\begin{proof}
		The conditions that the $(0,s) \in G$ act faithfully and that $$\rho^\pm(g(-g)) \in \rho^\pm[G^\pm,G^\pm]$$ are necessary.
		Furthermore, the equations of Lemma \ref{Lemma equations} prove that we have a pre-Kantor pair since these express firstly, in equations \ref{REP EQ V1}, \ref{REP EQ V 2}, \ref{REP EQ Tau1} and \ref{REP EQ Tau2}, that the definition of $V_{x,y},$ and $\tau_{x,y}$ is such that the action of $1 + \eta V_{x,y} + \eta^2 \tau_{x,y}$ coincides with the conjugation action of $\exp(o_{1,1}(x,y),\eta)$ over $\Phi[\eta]/(\eta^3)$. These equations also prove that $1 + \epsilon V_{x,y}$ is an automorphism since the operators $Q,$ $T,$ $R,$ and $P$ are defined using multiplications in $A$. Thirdly, the rest of the equations also express that the linearisations of the operators $Q,$ $R,$ and $P$ are as required; the equations can be matched to axioms for pre-Kantor pairs by matching the degrees. The linearization for $T$ follows from Lemma \ref{lemma linearizations mu} and $\nu_{k,1} = \mu_{k,1}$.
		
		We will prove that it is an actual operator Kantor pair in Appendix \ref{section proof}.
	\end{proof}
\end{theorem}

We use the previous theorem to prove that what we constructed in Example \ref{example hermitian form} is in fact an operator Kantor pair.

\begin{definition}
	We call a vector group $G$ with operators $Q^\text{grp}, T, P$  such that $(G,G,Q^\text{grp},T,P)$ forms an operator Kantor pair, an \textit{operator Kantor system}.
\end{definition}

\begin{lemma}
	The quadruple $(G,Q^{\text{grp}},T,P)$ of Example \ref{example hermitian form} forms an operator Kantor system if either condition (\ref{condition kantor pair hermitian form}) holds and $G$ is proper or if $1/2 \in \Phi$.
	\begin{proof}
		We will apply Theorem \ref{thm main 2}. If we construct vector group representations so that the operators of Example \ref{example hermitian form} correspond to the operators of Theorem \ref{thm main 2}, then it is sufficient to check that $g(-g) \in [G,G]$ in order to prove that we have an operator Kantor pair.
		Condition (\ref{condition kantor pair hermitian form}) says precisely this, namely $(m,a) \in G$ implies that
		$$ (m,a)(-m,a) = (0,2a - h(m,m)) = (0,a - \bar{a}) \in [G,G].$$
		We do not need this condition if $1/2 \in \Phi$ since $(m,a) \in G$ implies that $a - \bar{a} \in \langle h(x,y) - h(y,x) | x,y \in M \rangle$, so that $$(0, a - \bar{a}) \in \langle (0,h(x,y) - h(y,x))  \rangle = \langle [(x,h(x,x)/2),(y,h(y,y)/2)] \rangle.$$
		Furthermore, by definition $G$ needs to be proper. However, each vector group over $\Phi$ containing $1/2$ is proper.
		
		We work with the algebra $\mathcal{A}$ associated with a Kantor pair associated to a hermitian form in \cite[section 8]{ALLFLK99}. This algebra is of the form $\begin{pmatrix}
			C & M & C \\
			\bar{M} & \mathcal{E} & \bar{M} \\
			C & M & C\\
		\end{pmatrix}$ as a $\Phi$ module with $\bar{M}$ and $M$ equal as sets but opposite $C$-module structures and $\mathcal{E}$ a subalgebra of pairs of endomorphisms of $M,\bar{M}$; the product of $\mathcal{A}$ corresponds to a matrix product and product that not involve $\mathcal{E}$ are given by the actions of $C$ on $M$ and the bilinear map $M \times \bar{M} \longrightarrow C$. We do not need the products involving $\mathcal{E}$ in this proof.
		
		We also consider the vector group representations
		$$(m,a) \mapsto \begin{pmatrix}
			1 & m & a \\ 
			0 & (\text{Id},\text{Id}) & \bar{m} \\
			0 & 0 & 1 \\
		\end{pmatrix} \in \mathcal{A}, \quad 
		(n,b) \mapsto \begin{pmatrix}
			1 & 0 & 0 \\ 
			\bar{n} & (\text{Id},\text{Id}) & 0 \\
			b & \bar{n} & 1 \\
		\end{pmatrix} \in \mathcal{A}.$$
		One can compute that this gives us a pre-Kantor pair with the given operators $Q,T,P,R$ corresponding to $\nu_{ij}$ using Theorem \ref{thm main 2}.
		The decomposition of $\exp_+(sx)\exp_-(ty)$ as $\exp_-(tY)h\exp_+(sX)$ for certain $X,Y$with $h$ diagonal, can be used to compute certain formal power series in $s$ and $t$. Note that for such a decomposition
		$\exp_-(ty^{-1})\exp_+(sx^{-1}) = \exp_+(sX^{-1})h^{-1}\exp(tY^{-1})$ 
		always holds.
		The first decomposition yields that $\exp_+(tX)$ equals
		$$\begin{pmatrix}
			1 & tm + t^2s Q_xy + t^3s^2P_xy + O(s^3) & t^2a + t^3s(T_xy + (Q_xy) \bar{x}) + t^4s^2(R_xy + (P_xy) \bar{x} )  + O(s^3)\\
			& (\text{Id},\text{Id}) & \ldots \\
			& & 1 \\			
		\end{pmatrix},$$
		which shows the necessity of the formulas for $Q,T,P,R$.
		With the second decomposition, one can easily show that the omitted (and difficult) expression is what it should be, which proves that the formulas hold.
		Theorem \ref{thm main 2} shows that we have an operator Kantor pair.
	\end{proof}
\end{lemma}

\begin{remark}	
	The fact that we have an operator Kantor pair, can be used to prove that the element $sX$ of the previous theorem is a well-defined element of $s \cdot_1 G^+[[st]]$ which denotes the inverse limit of the groups $s \cdot_1 G^+(\Phi[st]/[(st)^n])$.
	This $sX$ can be thought of as the quasi inverse $(sx)^{ty}$ which one often considers in the context of Jordan pairs.
\end{remark}
	
\section{5-graded Lie algebras over rings $\Phi$ with $1/6$ and $1/30$}\label{se:lie}

We assume throughout this section that $1/6 \in \Phi$.
We use the definition of a Kantor pair $(P^+,P^-)$ with operators $V$ and associated Lie algebra as given by Allison and Faulkner \cite[section 3]{ALLFLK99}.

We remark that each vector group over $\Phi$ is proper.

\begin{lemma}
	\label{lemma kantor pair implies pre kantor pair}
	Let $(P^-,P^+)$ be a Kantor pair with associated Lie algebra $$ L  = [P^-,P^-] \oplus  P^- \oplus (\Phi\zeta \oplus V_{P^+,P^-}) \oplus P^+ \oplus [P^+,P^+].$$
	Then the following hold
	\begin{itemize}
		\item $\exp(P^+ \oplus [P^+,P^+]) = G^+$ is a vector group\footnote{This is the usual exponential $\exp(l) = \sum_{i = 0}^4 (\text{ad} \; l)^i/(i!)$.}. This vector group can be coordinatized by $\exp(e) = 1 + e_1 + e_2 + e_3 + e_4 \mapsto (e_1,e_2) \in \End_\Phi(L)^2$ with $e_i$ the part which acts as $+i$ on the grading.
		The similarly defined $G^-$ is also a vector group.
		\item On these pairs $(e_1,e_2)$ the operators $Q^\text{grp},T,P$ are given by
		\begin{itemize}
			\item 
			$Q^{\text{grp}}_g h = (\nu_{2,1}(g,h), \nu_{4,2}(g,h)),$
			\item $T_g h = (0,\nu_{3,1}(g,h))$,
			\item $P_gh = \nu_{3,2}(g,h)$
		\end{itemize}
		and satisfy the definition of a pre-Kantor pair.
	\end{itemize}
	\begin{proof}
		The map  $\exp : P^+ \oplus [P^+,P^+] \longrightarrow \text{End}_\Phi(L)$ is injective since $1/6 \in \Phi$, as can be observed by evaluating the exponentials on the grading derivation $\zeta$. Mapping $\exp(g) = 1 + e_1 + e_2 + \ldots \mapsto (e_1,e_2)$ is injective, since it is injective on $\text{Lie}(G^+)$. 
		From
		$$\exp(a)\exp(b) = \exp(a+b+[a,b]/2),$$
		which holds by the Baker-Campbell-Hausdorff Theorem, we conclude that $P^+ \oplus [P^+,P^+]$ is a vector group with representation $\exp$. The description as pairs $(e_1,e_2)$ is exactly the image of $P^+ \oplus [P^+,P^+]$ under this representation.
		We choose to work with this second coordinatization since this allows us to define the operators more easily.
		
		We want to check that the mentioned operators $Q^{\text{grp}}, T$ and $P$ map to $G^+, G^+_2$ and $G^+/G^+_2$. If this is the case, then these operators automatically satisfy the part of Definition \ref{Definition pre-Kantor pair} that determines the linearisations of these operators by Lemma \ref{Lemma equations} and Lemma \ref{lemma linearizations mu}. For a more precise correspondence between those lemmas and the definition, see the proof of Theorem \ref{thm main 2}.

		In order to prove that the operators map to the right spaces, we assume $g = (a,b) \in G^+$ and $h = (c,d) \in G^-$.
		We will use that $g_i [a,b] = \sum_{k + l = i}[g_k a,g_l b]$ for $i \le 4$, which must hold since we have $1/6 \in \Phi$.
		Note that this implies $\nu_{i,1}(g,h) = \mu_{i,1}(g,h) = \text{ad}(g_i \cdot h_p),$ where we use $h_p$ to denote the part in $P^-$ without the $\text{ad}$.
		In particular, this implies that $Q$ and $T$ already map to the right spaces.
		We note that $\mu_{1,1}(g,h)$ corresponds to the usual $V_{g,h}$ operator for Kantor pairs and that $\mu_{4,1} = 0$.
		
		For the $\nu_{i,2}$ we will use that $h_2$ is of the form
		$$ (h_1)^2/2 + \text{ad} \; v,$$
		for some $ v \in [P^-,P^-].$
		It is easy to check that 
		$$\mu_{(n,2)}(g,h) = \sum_{i + j = n} \mu_{(i,1)}(g,h)\mu_{(j,1)}(g,h)/2 + \text{ad} \; (g_n \cdot v),$$
		for $n \le 4$.
		Using this, we obtain that
		$$ \nu_{(4,2)}(g,h) =  \mu_{(2,1)}(g,h)^2/2 + \text{ad} \; (g_4 \cdot v) - \text{ad} (V_{g,h} T_g h)/2,$$
		which shows that $G^\text{grp}$ maps to $G^+$.
		Similarly, one proves that $P_g h$ is an inner derivation.
		
		Note that $1 + \epsilon V_{x,y} = 1 + \epsilon \text{ad} \; [x,y]$, which we let act using conjugation on the groups, is a vector group automorphism that interacts nicely with the defined operators.
		Lastly, observe that $G^+_2 = [G^+,G^+]$ and that $G^+_2$ acts faithfully on $P^-$ under the adjoint action, proving that it acts faithfully under $Q$ (since this coincides with adjoint action).
	\end{proof}
\end{lemma}

\begin{definition}
	We call the pre-Kantor pair associated to a Kantor pair, as constructed in the previous lemma, the \textit{associated pre-Kantor pair}.
\end{definition}

\begin{theorem}
	\label{theorem kantor implies operator kantor}
	Suppose that $1/30 \in \Phi$ and that $P = (P^+,P^-)$ is a Kantor pair over $\Phi$.
	The pre-Kantor pair associated to $P$ is an operator Kantor pair. 
	\begin{proof}
		Note that we are working with the usual exponentials and one checks that these are necessarily automorphisms using that $1/30 \in \Phi$.
		Note that 
		$$ \mu_{m,2}(x,y)= \nu_{m,2}(x,y) + \begin{cases}
			y_1\nu_{m,1}(x,y) & 1 \le m \le 2 \\
			y_1\nu_{3,1}(x,y) + \nu_{1,1}(x,y)\nu_{2,1}(x,y)&  m = 3 \\
			\nu_{1,1}(x,y)\nu_{3,1}(x,y) + \nu_{0,1}(x,y)\nu_{4,1}(x,y)& m = 4 \\
			\nu_{2,1}(x,y)\nu_{3,1}(x,y) + \nu_{0,1}(x,y)\nu_{5,1}(x,y) + \nu_{1,1}(x,y)\nu_{4,1}(x,y) & m = 5 \\
			0 & \text{otherwise}
		\end{cases} $$
		and that $\nu_{k,1}(x,y) = \text{ad} (x_k \cdot y_1)$, with $x_k$ the part of $\exp(x)$ that is $k$-graded. This immediately ensures that $\nu_{4,1} = \nu_{5,1} = 0$.
		We prove that the sufficient conditions listed in Remark \ref{remark sufficient condition} are satisfied, in order to prove that we have an operator Kantor pair.
		The first, second, and third condition follow immediately.
		
		Similar to Lemma \ref{lemma grouplike}, one can prove that $\exp(o_{i,j}(x,y))$ is an automorphism for all $x,y$, using that $\exp(x)$ and $\exp(y)$ are automorphisms.
		For the fourth, we only need to see that $\nu_{5,2}$ acts trivially on $P^-$, which is the case since it is a derivation $d$ (as it is the first not necessarily zero component of $o_{5,3}$) that acts as $+3 = 5 - 2$ on the grading.
		Namely, for $i$-graded $u$, we compute that
		$$  id(u) = d[\zeta,u] = [\zeta, d(u)] = (3 + i)d(u),$$
		so that $3d(u) = 0$.
		
		For $\nu_{6,3}(x,y)$, recall that 
		$$ \exp(o_{2,1}(x,y)) = \sum_{k = 1}^4 \nu_{2k,k}(x,y)$$
		is an automorphism. We know already that $\nu_{2,1}(x,y)$ and $\nu_{4,2}(x,y)$ correspond to the parts of $\exp(Q^\text{grp}_xy)$ which act as $+1$ and $+2$ on the grading.
		This proves that 
		$$ \nu_{6,3}(x,y) = \nu_{2,1}(x,y)\nu_{4,2}(x,y) - \nu_{2,1}(x,y)^3/3 + d$$
		with $d$ a derivation. Since there are no derivations which act as $+3$ on the grading, we conclude that $\nu_{6,3}(x,y)$ coincides with the part of $\exp(Q^\text{grp}_xy)$ which acts as $+3$ on the grading.
	\end{proof}
\end{theorem}

\begin{remark}
	We did only use $1/5 \in \Phi$ to prove that the exponentials are automorphisms of the Lie algebra.
\end{remark}

\begin{corollary}
	\label{bijection Kp preKP opKP}
	If $1/6 \in \Phi$, then the Kantor pairs and pre-Kantor pairs  stand in a bijective relation given by considering the associated pre-Kantor pair to a Kantor pair.
	If $1/30 \in \Phi$, then each pre-Kantor pair is an operator Kantor pair.
	\begin{proof}
		In the previous theorem and lemma, we established that each Kantor pair $P = (P^+,P^-)$ induces a pre-Kantor pair $G = (G^+,G^-)$ if $1/6 \in \Phi$. 
		So, it is sufficient to prove that $P \mapsto G$ establishes a bijection between Kantor pairs and pre-Kantor pairs if $1/6 \in \Phi$.
		
		We know that we can associate a Kantor pair $P'$ to each pre-Kantor pair $G$ by Lemma \ref{lemma preKantor induces Kantor pair}, from which we can uniquely recover the operators of $G$. Furthermore, we know that $P \cong P'$ since both pairs can be understood as $(G^+/G^+_2, G^-/G^-_2)$ with the induced action of $V$.
		
		If $1/30 \in \Phi$, Theorem \ref{theorem kantor implies operator kantor} proves that $P$ is also an operator Kantor pair.
	\end{proof}
\end{corollary}

\section{Structurable algebras}\label{se:struct}

We will construct operator Kantor systems corresponding to each of the classes of central simple structurable algebras as determined by Allison and Smirnov \cite{ALL78, SMI90Example,smirnov1990}, i.e., associative algebras, structurable algebras associated to a hermitian form, tensor products of composition algebras, Smirnov algebras, skew dimension one structurable algebras, and Jordan algebras. Allison and Faulkner \cite{ALLFLK93} proved that these classes of algebras are still structurable over arbitrary rings of scalars, though not in a manner related to $5$-graded Lie algebras. The operator Kantor pairs will have roughly the same Lie algebras. There are certain small divergences, however. Firstly, for associative algebras and structurable algebras associated to hermitian forms, it will prove useful to embed these in bigger structurable algebras of the same type if $1/2 \notin \Phi$. Secondly, we will not consider Smirnov algebras if $1/2 \notin \Phi$.

\subsection{(Quadratic) Jordan algebras}

Consider a quadratic Jordan algebra $(J,Q)$. Set $G = J \times 0$, $Q^\text{grp} = Q = (Q, 0)$, and $T = 0$. 
We also set $P_x y = Q_xQ_yx$. We remark that $G$ is proper by Lemma \ref{lem: Lie of G functorial}.

Now, we show that this forms an operator Kantor system.
First, observe that $(sx,ty)$ is quasi-invertible in $J[[s,t]]$ for $x,y \in J$ since the Bergmann operator $B(sx,ty)$ is invertible \cite[Proposition 3.2]{Loos75}, and that quasi inverses $(sx)^{ty}, (ty)^{sx}$ are determined by the symmetry principle \cite[Proposition 3.3]{Loos75}:
$$ (sx)^{ty} = sx + s^2Q_{x}(ty)^{sx}, \quad (ty)^{sx} = ty + t^2Q_{y}(sx)^{ty}.$$
Loos \cite[Theorem 1.4]{Loos95} proved that
$$ \exp(sx)\exp(ty) = \exp(ty)^{sx}\beta(sx,ty)\exp(sx)^{ty},$$
holds in the TKK representation.
We conclude that
$$ o_{2i+1,2i}(x,y) = \exp((Q_xQ_y)^ix), \quad o_{2i+2,2i+1}(x,y) = \exp((Q_xQ_y)^iQ_xy), \quad o_{i,j}(x,y) = 1 \text{ if $|i - j| > 1$}$$
hold in the TKK representation. Theorem \ref{thm main 2} shows that $(G,Q^{\text{grp}},T,P)$ forms an operator Kantor system.

\subsection{Associative algebras}

Let $A$ be an associative algebra with involution over $\Phi$.
If $1/2 \notin \Phi$, we set $B = A \otimes \Phi[s]/(s^2 - s)$ with involution determined by $a \otimes s \mapsto \bar{a} \otimes (1 -s) $. If $1/2 \in \Phi$ we also use $B$ to denote $A$.

The algebra $B$ induces an operator Kantor system of the form of Example \ref{example hermitian form}, using $(a,b) \mapsto a\bar{b}$ as the hermitian form. 
The vector group is given by
$$ G_B = \{ (a,b) \in B \times B | b + \bar{b} = a\bar{a}\}.$$
This vector group is proper, since $(G_B)_2 = \{a \otimes s - \bar{a} \otimes (1 - s)\}$ so that $\Lie(G) \otimes K \cong \Lie(\hat{G}_K)$, which is sufficient by Lemma \ref{lem: Lie of G functorial} for $G$ to be proper.
The operators are
\begin{enumerate}
	\item $Q_{(a,b)}c = -a\bar{c}a + bc,$
	\item $T_{(a,b)}c = a\bar{c}\bar{b} - bc\bar{a},$
	\item $P_{(a,b)}(c,d) = a\bar{c}a\bar{c}a - a\bar{c}bc - bda,$
	\item $R_{(a,b)}(c,d) = bd\bar{b} + a\bar{c}b\bar{c}a - a\bar{c}a\bar{c}b$.
\end{enumerate}

The reason we chose to extend $A$ if $1/2 \notin \Phi$, is to guarantee that condition (\ref{Kantor pair equation}) of the definition of a pre-Kantor pair holds (or its equivalent formulation in Example \ref{example hermitian form}), i.e.,
$$ x - \bar{x} \in [G_B,G_B]$$
holds for all $x \in B$. This condition holds in the extension, since
$$ [(x, x\bar{x}s), (1,s)] = (0,x - \bar{x}),$$
for all $x \in B$. 

\begin{remark}
	In order to obtain formulas compatible with the operators of following sections, one must apply the automorphism $((a,b) \mapsto (-a,b), (c,d) \mapsto (-c,d))$ to this operator Kantor system.
\end{remark}

\subsection{Structurable algebras associated to hermitian forms}\label{sss:herm}

Consider an associative algebra $C$ with involution $c \mapsto \bar{c}$ and a right $C$-module $M$ with (right-) hermitian form $h :  M \times M \longrightarrow C,$ i.e., it is the same as a left-hermitian form except that $M$ is a right $C$-module and $h$ should satisfy $h(a,bc) = h(a,b)c$ for $c \in C, a,b \in M$.
Set $A = C \times M$ with operation
$$ (a,m)(b,n) = (ab + h(n,m), na + m\bar{b})$$
and involution
$$ \overline{(a,m)} = (\bar{a},m).$$
This algebra is isomorphic to the structurable algebra associated to a hermitian form constructed by Allison \cite[8.iii]{ALL78}, by still using $h$ as the (now left) hermitian form with $M$ considered as a left module under $a \cdot m = m\bar{a}.$ When we speak of $C$ as a left $A$ module, we consider $M$ with this second module structure (and $A$ with the obvious structure). The reason we work with the right hermitian description, is that there appear less involutions in computations.

We shall consider
\[ G_{A} = \{ ((c,m),(d,mc)) \in A \times A| d + \bar{d} = c\bar{c} + h(m,m)\}.\]
We assume that $G_{A}$ satisfies condition (\ref{Kantor pair equation}), i.e., $x - \bar{x} \in [G_A,G_A]$ for all $x \in A$, and that $G$ is proper.
As for associative algebras, we can assume that both statements hold for $A$ or $A \otimes \Phi[s]/(s^2 - s)$ with $\bar{s} = 1 - s$ (with $ms = \bar{s}m$ for all $m \in M$). 

We shall construct all the operators by making use of a representation. This approach also establishes that $A$ corresponds to a $3$-special Kantor pair in the sense of \cite[section 6]{ALLFLK99}.
Define \[h^\pm : A \times A \longrightarrow C : ((a,m),(b,n)) \mapsto a\bar{b} \pm h(m,n).\] These are (left-)hermitian maps if we consider $A$ as a left $C$-module.
We let $M_2(C)$ act on $A^{2 \times 1}$ using
\[ \begin{pmatrix}
	a & b \\
	c & d
\end{pmatrix} \cdot \begin{pmatrix}
	(e,f) \\ (g,h)
\end{pmatrix} = \begin{pmatrix}
	(ae + bg, af - bh) \\
	(ce + dg, -cf + dh)
\end{pmatrix},\]
i.e., we act normally on $C^{2 \times 1} \subset A^{2 \times 1}$ and as
\[ \begin{pmatrix}
	a & -b \\
	-c & d \\
\end{pmatrix}\]
on $M^{2 \times 1}$. We write $b \cdot_{\epsilon} (g,h)$ for $(bg, - bh)$.
So, 
\[ f: A^{2 \times 1} \otimes A^{1\times2} \longrightarrow M_2(C) : \begin{pmatrix}
	a \\ b
\end{pmatrix} \otimes \begin{pmatrix}
	c & d
\end{pmatrix} \longmapsto \begin{pmatrix}
	h^+(a,c) & h^-(a,d) \\
	h^-(b,c) & h^+(b,d) 
\end{pmatrix}\]
is an $M_2(C)$-bimodule map if we let $M_2(C)$ act on the right on $A^{1 \times 2}$ using $x \cdot M = (M^* \cdot x^*)^*$ with the $a^*$ the hermitian transpose of $a$,
since
\[ h^-(c \cdot_\epsilon a,b) = c h^+(a,b), \quad ch^-(a,b) = h^+(c \cdot_\epsilon a,b),\]
for all $a,b \in A, c \in C$.
Set \[\mathcal{E} = \{ (a,b) \in \text{End}(A^{2\times1}) \times \text{End}(A^{1\times2}) |f(a(x),y) = f(x,b(y))\}.\]
Define
\[ g : A^{1\times2} \otimes_{M_2(C)} A^{2\times1} \longrightarrow \mathcal{E}\]
as
\[ g( x \otimes y)(z) = \begin{cases}
	xf(y,z) & z \in A^{1,2}\\
	f(z,x)y & z \in A^{2,1}
\end{cases}.\]
It is not hard to check that $g$ is a well defined $\mathcal{E}$-bimodule map. This proves that
\[ (M_2(C) \oplus \mathcal{E}) \oplus ( A^{2\times1} \oplus A^{1\times2}),\]
forms an associative algebra if all undefined multiplications are seen as $0$.
Using this definition, we can interpret elements as matrices in
\[ \begin{pmatrix}
	C & A & C \\
	A & \mathcal{E} & A \\
	C & A & C\\
\end{pmatrix},\]
with the obvious embeddings of subspaces.

We have two representations of the associated vector group, namely
\[ ((a,m),(u,ma)) \mapsto \begin{pmatrix}
	1 & (a,m) & -u + h(m,m) \\
	0 & 1 & (-a,-m) \\
	0 & 0 & 1\\
\end{pmatrix}\]
and 
\[ ((a,m),(u,ma)) \mapsto \begin{pmatrix}
	1 & 0 &  0 \\
	(-a,-m) & 1 & 0 \\
	-u + h(m,m) & (a,m) & 1\\
\end{pmatrix}. \]

Given an element $g = ((a,m),(u,ma))$, we write $\tilde{u}$ to represent $u - h(m,m)$ in operators involving $g$.
This representation yields us operators
\begin{enumerate}
	\item $Q_{(x,y)} b= h^+(x,b)\cdot x - \tilde{y} \cdot_\epsilon b,$
	\item $T_{(x,y)} b =  h^-( \tilde{y} \cdot_\epsilon b,x) -h^-( x,\tilde{y} \cdot_\epsilon b),$
	\item $P_{(x,y)} (u,v) = h^+(x,u) \cdot Q_{(x,y)} u - \tilde{y}\tilde{v} \cdot x,$
	\item $\tilde{R}_{(x,y)} (u,v) = \tilde{u}\tilde{v}\overline{\tilde{u}} - h^+(x,b)T_{(x,y)}u,$
\end{enumerate}
with $(Q_xy , \tilde{R}_{x}y)$ the image of $Q^\text{grp}_xy$ under $(a,b) \mapsto (a, \tilde{b})$.
So, we see that these operators form an operator Kantor pair on the structurable algebra $A$, using Theorem \ref{thm main 2}.
It is not hard to compute that
\[ (x\bar{y})x - uy = Q_{(x,u)} y\]
holds for $(x,u) \in G_A$ and $y \in A$, so that
\[ (x\bar{y})z + (z\bar{y})x - (z\bar{x})y = Q^{(1,1)}_{z,x}y.\]
Hence, the constructed operator Kantor pair corresponds precisely to the structurable algebra.

We remark that this construction works more generally for $M_1 \times M_2$ with $M_i$ left $C$-modules with hermitian forms $h_i$.
This corresponds to the Kantor triple system
\[ V_{x,y}(z) = h^+(x,y) \cdot z + h^+(z,y) \cdot x - h^-(z,x) \cdot_\epsilon y,\]
with $h^\pm = h_1 \pm h_2$. Using as underlying vector group
\[ G = \{(a,b) \in (M_1 \times M_2) \times C | b + \bar{b} = h^-(a,a), b - \bar{b} \in \langle h^-(x,y) - h^-(y,x) | x,y \in M_1 \times M_2 \rangle \}\]
contained in $(M_1 \times M_2) \times C$ with $\psi = h^-$, we get an operator Kantor pair if $1/2 \in \Phi$. If $1/2 \notin \Phi$ we need to make the extra assumptions that \[ \langle h^-(x,y) - h^-(y,x) | x,y \in M_1 \times M_2 \rangle = [G,G] \]
and that $G$ is proper.

\subsection{Constructing operator structurable algebras using $\mathbb{Z}$-substructures}

We want to construct subsystems over $\mathbb{Z}$ from some operator Kantor systems over $\mathbb{Q}$. We want to do this specifically for systems associated to structurable algebras. In this setting, the associated operator Kantor system over $\mathbb{Q}$ means the operator Kantor pair associated to the $5$-graded Lie algebra \cite[section 3]{ALL79} corresponding to the structurable algebra over $\mathbb{Q}$.

The advantage of these $\mathbb{Z}$-subsystems is that we do not need to check the axioms of operator Kantor pairs, but only need to check that the operators are internal.
Remark \ref{Remark uniqueness P and R} shows that if $P$ and $R$ exist for an operator Kantor pair, then they are unique. In this section, we start by first constructing $P$ and $R$ as homogeneous maps from $Q, T$ and $V$. After this construction, it becomes sufficient to construct the primary operators $Q$ and $T$ to construct all operators. 
Later we will construct $\mathbb{Z}$-analogues for different classes of structurable algebras. With those analogues we will be able to construct operator Kantor pairs for all classes of central simple structurable algebras (except Smirnov algebras if $1/2 \notin \Phi$).

We consider an algebra $C$ with involution $a \mapsto \bar{a}$.
Consider $\psi : C \times C \longrightarrow C, (a,b) \mapsto a\bar{b}$. This endows 
$$ G_C = \{ (a,b) \in C \times C | b + \bar{b} = a\bar{a}\},$$
with a vector group structure.
We will be interested in defining an operator Kantor pair structure with operator
\[ Q_{(a,b)}(c,d) = (a\bar{c})a - bc.\]
We assume that 
$$ \text{Lie}(G_C) = C \times (G_C)_2,$$
i.e., for each $a \in C$ there exists $b$ such that $b + \bar{b} = a\bar{a}$. The $(G_C)_2$ appearing in this assumption does not impose anything as proved in Lemma \ref{lemma Lie G}. We set $S = (G_C)_2$.
Note that $x - \bar{x} \in S$ for all $x \in C$ and that $2S = \{ s - \bar{s} | s \in S\}$ since $s = - \bar{s}$ for all $s \in S$.
So, the assumption that $ \text{Lie}(G_C) = C \times (G_C)_2$ guarantees that conditions (\ref{Kantor pair equation}), (\ref{Kantor pair equation 2}) for pre-Kantor pairs always hold. This assumption is only important if $1/2 \notin \Phi$ and will, for us, only play a role in structurable algebras that can have an arbitrary associative part.

We want to prove that there exists at most one operator Kantor system with
$$ Q_{(a,b)}(c,d) = (a\bar{c})a - bc,$$
and a given operator $T : G_C \times C \longrightarrow \{ x - \bar{x} | x \in C \} = \{ [a,b] = a\bar{b}- b \bar{a} | a,b \in C\}$ of bidegree $[3,1]$. We remark that we will use $V$ and $\tau$ without making reference to the operator Kantor pair structure on which they are defined. In that case, we use operators
$$V_{x,y} g = ( - Q^{(1,1)}_{g,x} y, - T^{(2,1)}_{g,x} y - \psi(Q^{(1,1)}_{g,x}y,g))$$ and $$\tau_{y,x} g = (P^{(1,2)}_{g,x^{-1}}y^{-1}, R^{(2,2)}_{g,x^{-1}}y^{-1} + \psi(P^{(2,1)}_{g,x^{-1}}y, g) - \psi(Q_gy,Q_{x^{-1}}y)),$$
i.e, we identify the operator Kantor $V$ with  its restriction to $C^+ \times C^- \longrightarrow \text{Nat}(G^+)$ and $\tau$ to its restriction $G^- \times G^+ \longrightarrow \text{Nat}(G^+)$. There is an incongruity in the order of arguments with the purpose of making $V$ correspond to the usual $V$ for structurable algebras, while still taking the version of $\tau$ that is not only the most easy to define, but also the only one which we will need in this section.

\begin{lemma}
	\label{lemma constructing structurable algebras}
	Let $C$ be a unital algebra with involution and consider a (proper) vector group $G \le G_C$ with $\Lie(G) = C \times S$ for some $S$.
	Define $Q_{(a,u)}b = (a\bar{b})a - ub$. Suppose that there exist operators $T,P,R$ such that $(G,Q^\text{grp},T,P)$ forms an operator Kantor system. Use $(V_{x,y})_i$ to denote the action of $V$ on the $i$-graded parts of $\text{Lie}(G_C)$. The operators $P,R$ are uniquely determined by $Q,T$, as indicated by the following formulas, in which $x = (a,b)$, and $g, h \in G$:
	\begin{enumerate}
		\item $P_g (x(-x))= T_g(b) - (T_g1)b - (V_{g,1})_1Q_g b + (V_{g,b})_1 Q_g 1,$
		\item $(\tau_{h,g(-g)} g)_1 = - Q_{g(-g)}Q_h g + V_{g,h} Q_{g(-g)}h ,$
		\item $ P_{g} h = - P_{g} (h(-h)) - Q_{g}Q_hg - \tau_{h,g(-g)}g + 2 Q_{T_{g}h}h - (V_{g,h})_1Q_{g}h,$
		\item $(\tau_{h,T_{g^{-1}}1} g)_1 = [Q_{T_{g^{-1}}1},Q_{h^{-1}}g] - Q_{[g,Q_{T_{g^{-1}}1}h]}h,$
		\item $ R_g h = (Q_gh)^2 - P_g(T^{(2,1)}_{h,1}g) - Q_gQ_hQ_{g^{-1}} 1 + (\tau_{hT_{g^{-1}}1} g)_1 + (V_{g,h})_1 T_gh.$
	\end{enumerate}
	Moreover, if there is no $3$ torsion, then $T$ is uniquely determined by
	$$ 3(T_g h - [g,Q_g h]) = 2((g\bar{h})g)\bar{g} - 2g(\bar{g}(h\bar{g})) - ((g\bar{g})h)\bar{g} + g(\bar{h}(g\bar{g})).$$
	\begin{remark}
		Substituting equations $1$ and $2$ in equation $3$ of the lemma yields that $P_g(a,b)$ must equal to
		\[- T_g(b) + (T_{g}1)b + (V_{g,1})Q_gb - (V_{g,b})Q_g1 - Q_gQ_{(a,b)}g - Q_{g(-g)}Q_{(- a,\bar{b})}g + V_{g,a}Q_{g(-g)}a + 2Q_{T_ga}a - (V_{g,a})_1Q_ga, \] 
		which is a homogeneous map of bidegree $[3,2]$ in $g$ and $(a,b)$.
		Similarly, substituting equation $4$ in equation $5$ yields us a definition of $R_g h$ as a homogeneous map of bidegree $[4,2]$ in terms of $V$, $Q$, $T$ and $P$.
	\end{remark}
	\begin{proof}
		
		We work with the Lie algebra $L$ associated to an operator Kantor pair, and prove that these equalities hold.
		
		We know that $x(-x) = (0,2b - a\bar{a}) = (0,b - \bar{b}) = [b,1]$ for $x = (a,b)$.
		The first equation is obtained by evaluating
		$$ \text{ad} \; g_3 ([b,1]) = \nu_{3,2}(g,x(-x)),$$
		on the grading element $\zeta$ of $L$. Namely,
		using that $\exp(g)$ is necessarily an automorphism so that
		\[\text{ad} \; g_3 ([b,1]) \cdot \zeta = ([T_gb,1] + [b,T_g1] + [Q_gb,V_{g,1}] + [V_{g,b},Q_g1]) \cdot \zeta = [T_gb,1] + [b,T_g1] + [Q_gb,V_{g,1}] + [V_{g,b},Q_g1],\]
		and that
		\[ \nu_{3,2}(x,y) = \text{ad} \; P_x y,\]
		we obtain the first equation.
		The second equation is obtained from $$(\tau_{a,b}c)_1 = - P^{(1,2)}_{c,b} a^{-1} =- P^{(2,1)}_{b,c} a^{-1} = - Q_{b} Q_a c + V_{c,a} Q_{b} a$$
		which holds if $b \in G_2$,
		by first applying the definition of $\tau$, using that $f^{(2,1)}_{b,a} = f^{(1,2)}_{a,b}$ if $b \in G_2$ for all $f$ homogeneous of degree $3$, and the expression of $P^{(2,1)}$ for pre-Kantor pairs.
		
		The third equation follows from evaluating 
		$$ \nu_{3,2}(g,h) = \mu_{3,2}(g,h) - h_1\mu_{3,1}(g,h) - \mu_{1,1}(g,h)\mu_{2,1}(g,h)$$
		on the grading element $\zeta$.
		The last equation follows from evaluating 
		$$ \nu_{4,2}(g,h) = \mu_{4,2}(g,h) - \mu_{1,1}(g,h)\mu_{3,1}(g,h),$$
		on the $1$ contained in the $-1$-graded copy of $C$ in $L$.
		Namely, we know that
		\[ (\nu_{4,2}(g,h) + \mu_{1,1}(g,h)\mu_{3,1}(g,h)) \cdot 1 = (Q_gh)^2 - R_gh + (V_{g,h})_2 T_g h,\]
		while we also know that
		\[ (g_3 h_2 g_1^{-1} + g_2 h_2 g_2^{-1} + g_1 h_2 g_3^{-1}) \cdot 1 =  P_g\left(T^{2,1}_{(h,1)} g\right) + Q_gQ_h Q_{g^{-1}} 1 - \tau_{h,T_{g^{-1}}1} g\]
		since
		$$ g_3 h_2 g_1^{-1} \cdot 1 = g_3 ( h_2 \cdot  V_{1,g}) = P_g(T^{(2,1)}_{(h,1)}g).$$
		The fourth equation follows from evaluating the equation
		$$ (g_1h_2g^{-1}_3 + h^{-1}_2g_1g^{-1}_3 - h_1g_1h_1g^{-1}_3 - (Q_{h^{-1}}g)_1 g^{-1}_3) = 0,$$
		on the $-1$-graded $1$. This last equation holds since
		$$ \mu_{2,1}(h^{-1},g) = (Q_{h^{-1}}g)_1.$$
		
		Note that equations $1,2$ and $4$ uniquely determine their left-hand sides, since these only use $Q,T$ and $V$ (with $V$ itself also being defined in terms of $Q$ and $T$).
		Now, we see that the third equation, uniquely determines $P_g h$.
		Now that $P$ is defined, we see that the last equation determines $R_gh$ uniquely as well.	
		
		We use Equation (\ref{equation homogeneous of degree 3}) to obtain
		$$ 3T_g - 3T^{(1,2)}_{g,g} = - T^{(1,1,1)}_{g,g,g},$$
		as $T$ must be homogeneous of degree $3$ in $g$. Using that $T^{(1,2)}_{g,g}h = [g,Q_gh]$ and evaluating the $(1,1)$-linearisation $T^{(1,1,1)}_{g,g,g}h$ yields the expression for $T$ given in the statement of this lemma, since 
		\[ 	- T^{(1,1,1)}_{g,g,g} h = - [g,Q^{(1,1)}_{g,g} h] = - [g,2(g\bar{h})g - (g\bar{g})h] = - g\overline{(2(g\bar{h})g - (g\bar{g})h)} + (2(g\bar{h})g - (g\bar{g})h	)\bar{g}. \qedhere \]
		
	\end{proof}
\end{lemma}

\begin{remark}
	The previous lemma is not that useful if $1/2 \in \Phi$, since the operators are then uniquely defined by
	\[ 2 P_g h = P^{(3,(1,1))}_g (h,h) + P_g^{(3,(1,1))}(h_2,1) - P_g^{(3,(1,1))}(1,h_2),\]
	and
	\[ 2 R_g h = R^{(4,(1,1))}_g (h,h) + R_g^{(4,(1,1))}(h_2,1) - R_g^{(4,(1,1))}(1,h_2),\]
	using the known linearisations of these operators.
\end{remark}

\subsubsection{Tensor product composition algebras}

Let $O_1$, $O_2$ be two composition algebras with norms $N_1,N_2$.
Consider $A = O_1 \otimes O_2$ with components-wise involution and suppose that there exists $t_i \in O_i$ such that $t_i + \bar{t}_i = 1$.
Note that if $a = \sum_{i = 1}^n \lambda_i a_i \otimes b_i,$ corresponding to a basis formed by pure tensors $\{a_i \otimes b_i | i \in \{1,\ldots,n\}\}$ of $A$, we can define
$$q(a) = \sum_{i = 1}^n \lambda_i^2 N_1(a_i)N_2(b_i) t_1 \otimes 1 + \sum_{i < j} \lambda_i\lambda_j a_i\bar{a_j} \otimes b_i\bar{b_j},$$ so that
$$ q(a) + \overline{q(a)} = a\bar{a}.$$
This map $q$ is clearly quadratic.
Set
$$ G_A = \{ (a,s + q(a)) \in A \times A | s \in (1 \otimes S_{O_2} + S_{O_1} \otimes 1) \; \cup \; \{ x - \bar{x} | x \in O_1 \otimes O_2 \} \},$$
with $S_{O_i} = \{ x \in O_i | x \perp 1,t_i\}$. 
This means that $G_2$ contains all $x - \bar{x}$ and all $1 \otimes s_1 + s_2 \otimes 1$ with $s_i \perp 1,t_i$, i.e., all elements that could reasonably be called skew. These $G$ are clearly proper.

We want to prove that there exists an operator Kantor system associated with $G_A$ and 
$$ Q_{(a,b)} h = (a\bar{h})a - bh.$$

We do this by considering universal variants of $O_1,O_2$ and we will still refer to these variants as $O_1,O_2$.
Consider $R = \mathbb{Z}[a,b,c,d,e,f]$.
Set $O_1 = CD(R[t_1]/(t_1^2 - t_1 + a),b,c)$ with which we mean the algebra obtained by applying the Cayley-Dickson process $2$ times using the parameters $b,c$ starting from $R[t]/(t^2 - t +a)$.
Similarly, we set $O_2 = CD(R[t_2]/(t_2^2 - t_2 + d),e,f)$.
Note that $O_i$ is a subalgebra of an octonion algebra $P_i$ over $\mathbb{Q}(a,b,c,d,e,f)$.

We want to prove that $A = O_1 \otimes O_2 \le P_1 \otimes P_2$ induces an operator Kantor system.
So, we use the similarly defined $G_A$ over $\mathbb{Z}$ and note that this group is proper as well.
We first check whether $T$ restricts to $A$.
We know that $T_gh \in A$ for $h \in A$ and $g$ of the form $(o_1 \otimes o_2, u)$ by using the alternativity of the multiplication for octonion algebras and the expression for $T$ proved in Lemma \ref{lemma constructing structurable algebras}. More precisely, for these elements
we can compute that $$ T_{(a \otimes b, u)} c = [a \otimes b, - uc ] = - (a \otimes b)(\bar{c}\bar{u}) + (uc)(\bar{a} \otimes \bar{b}).$$
Moreover, the linearisations of $T$ restrict nicely to $A$ since they satisfy $T^{(1,2)}_{x,g} y = [x,Q_gy]$ and $T^{(2,1)}_{x,g} = T^{(1,2)}_{g,x} + T^{(1,2)}_{x,[x,g]}$.
Hence $T$ maps $$G_A \times G_A \longrightarrow [G_A,G_A] \subset (G_A)_2.$$
So, the operators $Q,P$ restrict to maps from $G_A \times G_A$ to $A$.
Observe that
$$ \langle (t - \bar{t}) \otimes 1, t \otimes t - \bar{t} \otimes \bar{t} \rangle_{\mathbb{Z}} \subset A \otimes \mathbb{Q} = \langle (t - \bar{t}) \otimes 1, 1/2 ((t - \bar{t}) \otimes 1 + 1 \otimes (t - \bar{t})) \rangle_{\mathbb{Z}}.$$
We also note that $(1 \otimes S_2 + S_1 \otimes 1) \otimes \mathbb{Q} \cap A = (1 \otimes S_2 + S_1 \otimes 1)$. Using that $$ x - \bar{x} \in (1 \otimes S_2 \oplus S_1 \otimes 1 \oplus (1 - 2t_1)\Phi \otimes 1 + 1 \otimes \Phi(1 - 2t_2)) \otimes \mathbb{Q}$$ for $x \in A \otimes \mathbb{Q}$ since $1 - 2t_i = \bar{t}_i - t_i$, and that $A \ni R_gh - q(Q_gh) \in \langle x - \bar{x} | x \in A \otimes \mathbb{Q} \rangle$, we conclude that $Q^\text{grp}$ restricts to $G_A$.
This proves that all operators restrict nicely to $G_A$.

Now, consider $G_A(g)(G_A(k[a,b,c,d,e,f]))$ for $g: k[a,b,c,d,e,f] \longrightarrow k$ corresponding to actually considering octonion algebras $O_1,O_2$ obtained by a Cayley Dickson process over a field $k$. This induces an operator Kantor pair structure associated to the tensor product of octonion algebras.
For general composition algebras, we only need to consider subsystems of the just constructed operator Kantor pair.

\subsubsection{Smirnov algebras}

Allison and Faulkner generalized the Smirnov algebra already to arbitrary rings \cite[example 6.7]{ALLFLK93} combined with \cite{ALLFLK93_2} in a setting suitable for the study of Lie algebras (but not $5$-graded Lie algebras). We show that if $1/2 \in \Phi$, then there is an operator Kantor system that yields the same Lie algebra.

Consider an octonion algebra $O$ with norm $N$. The Smirnov algebra is a quotient of the subalgebra $B$ of $O \otimes O$ generated by $o \otimes o$ for $o \in O$ by an ideal $I$ of hermitian elements.
The ideal $I$ corresponds to the scalar multiplications on $O$ under the embedding $$B \longrightarrow \text{End}(O) : c \otimes c \mapsto (x \mapsto N(c,x)c).$$ If $O$ is split with basis $\{e_1,e_2,e_3,e_4,e^*_1,e^*_2,e^*_3,e^*_4\}$ such that $N(e_i,e_j) = N(e^*_i,e^*_j) = 0$ for all $i,j$ and $N(e_i,e^*_j) = \delta_{ij}$, then $I$ is linearly generated by
$$ \sum_{i = 1}^4 e_i \otimes e^*_i + e^*_i \otimes e_i,$$

We set $A = B/I$.
We can take the fixed elements under the exchange involution of the operator Kantor system associated to the tensor product of composition algebras.
This yields an operator Kantor system associated to $B$ if $1/2 \in \Phi$. 
It is not hard to check that all operators are compatible with the quotient with respect to $I$, as they are formed by multiplications, involutions and the operator $T$.
For the operator $T$ this follows from the formula for $T$ and the linearizations of $T$, if there is no $3$-torsion. It also holds in general, as can be proved using the model of $O \otimes O$, now over $\mathbb{Z}[1/2]$, from the previous subsection.

\begin{remark}
	If $1/2 \notin \Phi$, it is better let the unit go. Let $O$ be the split octonion algebra over $\mathbb{Z}$. If necessary we represent elements as Zorn matrices.	
	Consider \[A = \{ x \in \langle o \otimes o | o \in \mathbb{O} \rangle : n(x) = 0\},\] with $n(o \otimes o) = o\bar{o} = N(o)$, which is a nonunital algebra. 
	Let $G_A$ be the subgroup of $A \times A$ generated by the set $\{ (o \otimes o, 0) | N(o) = 0\}$,
	and elements $ (\tilde{t} - \tilde{e}_i,\tilde{t} - \tilde{e}_i),$
	with $\tilde{t} = t \otimes \bar{t} + \bar{t} \otimes t$, $\tilde{e}_i = e_i \otimes e_i^T + e_i^T\otimes e_i$ for a standard generator $e_i \in \mathbb{Z}^3$ using
	\[ t = \begin{pmatrix}
		1 & 0 \\
		0 & 0
	\end{pmatrix}, \quad e_i = \begin{pmatrix}
		0 & e_i \\
		0 & 0 \\
	\end{pmatrix}.\]
	It is not hard to check that $\text{Lie}((G_{A})_\Phi) = (A \oplus \{ x - \bar{x} | x \in A\}) \otimes \Phi$ using that $A$ is the direct sum of the rank $32$ subspace $\langle o \otimes o| N(o) = 0\rangle$ and the rank $3$ subspace $\langle \tilde{t} - \tilde{v} | v = e_1,e_2,e_3 \rangle$, proving that $G_A$ is a vector group. This immediately shows that $G_2 \otimes K = \hat{G}_2(K)$ and thus that $G$ is proper by Lemma \ref{lem: Lie of G functorial}.
	By embedding $G_A \subset G_{O \otimes O}$ we inherit the operators $Q,T,P$ and $R$ mapping $G_A \times G_A \longrightarrow G_{O \otimes O}$.
	On $G_A(\mathbb{Q})$ the operators are $\mathbb{Q}$-polynomial expressions of $A$-elements, proving that the operators map to $G_A$ instead of $G_{O \otimes O}$.
	This construction differs form $B/I$ since $  \tilde{t} + \tilde{e}_1 + \tilde{e}_2 + \tilde{e}_3 \neq 0$ in $A \otimes \mathbb{F}_2$ while
	\[  \tilde{t} + \tilde{e}_1 + \tilde{e}_2 + \tilde{e}_3 \in I \otimes \mathbb{F}_2. \]
\end{remark}

\subsubsection{Structurable algebras of skew-dimension $1$}

Consider a hermitian cubic norm structure $(J,N,T,\sharp,\times,e)$, as described for example by the second author \cite[Definition 4.1]{DeMedts2019}\footnote{ We will not explicitly use the additional axioms of \cite[Remark 4.4]{DeMedts2019} that are required for fields of characteristic $2$ or $3$. These will still implicitly play a role, however, in constraining the classes of hermitian cubic norm structures that we will consider. }, over a quadratic étale extension $A = E = \Phi[t]/(t^2 - t + \alpha)$ of $\Phi$, for some $\alpha \in \Phi$, and associated algebra
$ E \times J$ with operation
\[ (a,b)(c,d) = (ac + T(b,d), ad + \bar{c}b + b \times d) \]
and involution
\[ (a,c) \mapsto (\bar{a},c). \]
Over fields of characteristic different from $2$ and $3$ each structurable algebra of skew dimension $1$, i.e., a structurable algebra such that the space $\{ x | x + \bar{x} = 0\}$ is $1$-dimensional, can be described in such a way as proved in \cite[Theorem 5.14]{DeMedts2019}\footnote{This theorem says that specific forms of matrix structurable algebras can be described as hermitian cubic norm structures. These specific forms are all skew-dimension $1$-structurable algebras \cite[Theorem 1.13]{ALL90}.}. Recall that an arbitrary cubic norm structure $(J,N,T,\sharp,\times,e)$ over $F$ induces a hermitian cubic norm structure $(J \oplus J, N', T', \sharp', \times', e')$ uniquely determined by $N '(a,b) = (N(a),N(b)), T'((a,b),(c,d))= (T(a,d),T(b,c)), (a,b)^{\sharp'} = (b^\sharp,a^\sharp)$ \cite[Example 3.10]{DeMedts2019}. We call this hermitian cubic norm structure a \textit{split hermitian cubic norm structure}.

We consider the vector group
$$ G_A = \{ ((a,b),((u,b^\sharp + ab)) \in (E \times J)^2| u + \bar{u} = a\bar{a} + T(b,b)\}.$$
This vector group is proper, since $(G_A)_2 = \langle t - \bar{t} | t \in E\rangle$.
Remark that $u,v$ satisfy $u + \bar{u} = v + \bar{v} = a\bar{a} + T(b,b)$ if and only if $(u - v) + \overline{(u - v)} = 0$. If we work over a field $\Phi$ the space of $x \in E$ such that $x + \bar{x} = 0$ coincides with the span of all $y - \bar{y}$ since both spaces are $1$ dimensional and contain $t - \bar{t} \neq 0$. 
One can compute, see Appendix \ref{appendix T}, that
$$ T_{((a,j),(u,j^\sharp + aj))} (b,k) = [(u - T(j,j))(b\bar{a}) + (u - a\bar{a})T(k,j) - T(ak,j^\sharp) + bN(j) + T(j^\sharp \times k,j),1],$$
using $[a,1] = a - \bar{a}$,
if one works over a ring $\Phi$ without $3$-torsion.

Now, we prove the analogon of \cite[Theorem 5.14]{DeMedts2019} in our setting, in which certain forms of matrix structurable algebras are related to hermitian cubic norm structures over fields of characteristic different from $2$ and $3$.

\begin{lemma}
	Suppose that $A$ is an algebra over a field $\Phi$ for which there exist a quadratic extension $K/\Phi$ such that (1) $A \otimes K$ forms a matrix structurable algebra, i.e., $A \otimes K$ can be constructed from a split hermitian cubic norm structure over $K$, (2) there exists a vector group $G_A$ such that $(G_A)_K \cong G_{A \otimes K}$. Then $A$ is an algebra constructed from a hermitian cubic norm structure.
	\begin{proof}We will try to reconstruct all operators uniquely from the assumptions we made in the lemma. All axioms for hermitian cubic norm structures will hold, since we are considering a substructure, since $G_A(\Phi) \subset G_A(K)$ for all $\Phi$-algebras $K$, of a split hermitian cubic norm structure for which these axioms necessarily hold.
		
		Consider $\alpha$ such that $(1,\alpha) \in G_A(\Phi) \subset G_A(K)$ and let $t = Q_{(1,\alpha)} 1$. We can assume that $t \neq 1/2$.
		We set $E$ to be the submodule of $A$ generated by $t$ and $\bar{t} = 1 - t$.
		We observe that $E \otimes K \cong K \oplus K$ (as an algebra) since $\langle Q_{(1,\alpha)} 1 \in A \otimes K | (1,\alpha) \in G(K) \rangle = K \oplus K$, so that $t^2 - t = - a \in \Phi$ and $E \cong \Phi[t]/(t^2 - t + a)$.
		We put
		$$ J = \langle j \in A | ej = j \bar{e}, \forall e \in E \rangle.$$
		We know that $A_K \cong E_K \oplus J_K$. 
		So, for $b \in A$, we write $b = e_k + j_k$ and compute $tb - b\bar{t} = e_k(t - \bar{t})$.
		We obtain that $(tb - b\bar{t})(t - \bar{t}) = e_k(1 - 4a) \in \Phi^\times e_k$ since $4a \neq 1$ ($4a = 1$ implies that $a = 1/4$, $t = 1/2$, $t - \bar{t} = 0$). We see that $e_k \in E$ and $j_k \in J$ and thus $A = J \oplus E$. 
		The $E$-module structure of $J$ is given by left multiplication in $A$. 
		
		We immediately obtain the existence of $\sharp$ restricted from $J_K \longrightarrow J_K$ using the second assumption.
		We remark that $\sharp$ defines all operations uniquely. Firstly, $\times $ is just the linearization of $\sharp$.
		Secondly, if $E$ is a field, then $N$ is uniquely defined from $N(a)a = (a^\sharp)^\sharp$. If it is not a field, then there exists an idempotent $e \in E$ such that $E = \Phi e \oplus \Phi (1 - e)$. In that case, $N(a)e, N(a)(1 -e)$ are uniquely defined from evaluating $N$ on $ea,$ $(1 -e)a$.
		Namely, $eJ$ and $(1 - e)J$ are $\Phi$ vector spaces, $\Phi \cong e\Phi \subset E$ acts on the former by scalar multiplication and acts trivially on the latter while $(1 - e)\Phi$ acts only nontrivially on the latter, so $eN(a), (1-e)N(a)$ are easily recovered if $ea, (1-e)a \neq 0$.
		If $ea = 0$, however, then we know that $eN(a) = N(ea) = 0$.
		
		Lastly, we can recover $T$ using $T(a,b) = ab - a \times b$.
	\end{proof}
\end{lemma}

Using similar observations as for the tensor product of composition algebras, one can prove that $(G,Q,T,P,R)$ forms an operator Kantor system for certain classes of hermitian cubic norm structures. 

Namely, over fields $\Phi$ a non degenerate (split hermitian) cubic norm structure is either of the form $\Phi \oplus J(Q,1)$ for some quadratic form with basepoint $1$, $\mathcal{H}_3(\mathcal{C},\gamma)$ for a composition algebra $\mathcal{C}$, or it is anisotropic \cite[Theorem 4.1.59]{MUL22}. 
In the first two cases, \cite[example 2.2 and 2.3]{PETRAC86} provide versions of cubic norm structures defined over arbitrary rings of scalars, which one can use to prove that the split hermitian cubic norm structures corresponding to isotropic cubic norms induce operator kantor pairs associated to the skew dimension $1$ algebras over arbitrary fields.

For anisotropic cubic norm structures, one can also sometimes construct variants over arbitrary rings. For example, any ring $\Phi$ forms a cubic norm structure over itself with $$N(x) = x^3, \quad T(x,y) = 3xy, \quad x^\sharp = x^2.$$ In particular, $\mathbb{Z}[t_i | i \in I] \subset \mathbb{Q}(t_i | i \in I)$ forms such a cubic norm structure and thus induces an operator Kantor pair for all $I$. The split hermitian cubic norm structure related to an inseparable field extension $E/F$ of degree $3$ is a substructure of the split hermitian cubic norm structure related to $E$, which is an operator Kantor pair as it is the quotient of the operator Kantor pair induced by $\mathbb{Z}[t_i | i \in E]$ (we should pass via $E[t_i | i \in E]$ in order to obtain an operator Kantor pair over $E$). For the other classes of anisotropic cubic norm structures, which can be denoted as 27/F, 27K/F, 9/F, 9K/F, and 3/F using the notation of \cite[Section 15]{Tits2002} and \cite[Theorem 17.6]{Tits2002}, it is easier to identify them with anisotropic substructures of isotropic structures, so that we work with cubic norms on split Albert algebras (27(K)/F), $3 \times 3$ hermitian matrices over $\Phi[t]/(t^2 - t)$ (9(K)/F), or $\Phi \oplus \Phi \oplus \Phi$ (3/F).

\subsection{Structurable algebras as operator Kantor pairs with unit}

Now, we will propose a definition of structurable operator Kantor pairs. First, we mention two results in the literature from which we took our inspiration. Thereafter we state the definition and see that it fits the examples of structurable algebras we constructed earlier in this section. There are two exceptions, however, if $1/2 \notin \Phi$. Namely, the Smirnov algebras since we threw away the unit in our construction and the quadratic Jordan algebras since the skew elements start to play a role in characteristic $2$. For structurable operator Kantor pairs, we will be able to see $G^+_2$ as skew elements with respect to a certain involution and be able to prove that there exists a certain filtration of the associated Lie algebra.

 Allison \cite[Theorem 4]{ALL79} reconstructed a structurable algebra from a $5$-graded Lie algebra that decomposes in a certain way with respect to a certain $\mathfrak{sl}_2$ subalgebra. This subalgebra can, after the reconstruction of the structurable algebra, be seen as being generated by the units of the distinct copies of the structurable algebra.
 
  Later, Allison and Faulkner \cite[Corollary 15]{ALLFLK99} saw that a structurable algebra over a ring $\Phi$ containing $1/6$ is nothing more and nothing less than a Kantor pair over the same ring containing a 1-invertible element $\exp(x)$ for $x$ contained in the Kantor pair (see \cite{ALLFLK99} for the definition of a 1-invertible element). 
  Furthermore, in their setting they saw that $x$ being 1-invertible is the same as asking it to be conjugate invertible, i.e., asking that there exists an $y$ such that $(V_{x,y}, V_{y,x}) = (2 \text{Id}, - 2 \text{Id})$.
  The logical choice for $x$, if one looks at the Kantor pair obtained from a structurable algebra, is $1$.
  The corresponding $y$ is $2$, since $V_{1,1} a= a$.
  Furthermore, $x$ being 1-invertible with corresponding $y$ requires that
  \[ \tau_{x,y} = \exp_-(y) \exp_+(x) \exp_-(y)\]
  reverses the grading of the Lie algebra $L$ associated to the Kantor pair. This $\tau_{x,y}$ is an automorphism with $\tau_{1,2} (a_+) = 2 a_-$ for $a$ in the structurable algebra. Unsurprisingly, annoying factors $2$ play an important role here.

\begin{definition}
	We call an operator Kantor pair $(G^+,G^-)$ \textit{structurable} if there exist $x \in G^+/G^+_2$ and $y \in G^-/G^-_2$ such that
	\begin{enumerate}
		\item[(a)] $\tau = \exp(v) \exp(u) \exp(v)$	induces a bijection $G^- \longrightarrow G^+$ under the conjugation\footnote{If $1/6 \in \Phi$ this is the same as asking that $\tau$ reverses the grading of the Lie algebra, i.e., that $x$ is 1-invertible. However, if $1/6 \notin \Phi$, we prefer the more restrictive version we ask here.} action for $u = (\sqrt{2}x,\psi(x,x)) \in G^+(\Phi[\sqrt{2}])$ and $v = (\sqrt{2} y, \psi(y,y)) \in G^-(\Phi[\sqrt{2}])$ ,
		\item[(b)] $V_{x,y}$ is equal to the grading derivation $\zeta$.
	\end{enumerate}
\end{definition}

\begin{remark}
	One can check that $Q_u y = x$ and that $y$ is the unique element for which this is the case.
	Moreover, $\tau (y) = x$ and $\tau (x) = y$ so that $\tau = \exp(u) \exp(v) \exp(u)$ as well. 
\end{remark}

\begin{remark}
	Over fields of characteristic $2$, the definition of a structurable operator Kantor pair implies that there exist $x$ and $y$ such that $x^{[2]} = (0,\psi(x,x))$ and $y^{[2]} = (0,\psi(y,y))$ induce a bijection $\tau = \exp(y^{[2]}) \exp(x^{[2]}) \exp(y^{[2]}): G^- \longrightarrow G^+$, since we can use a map $\sqrt{2} \mapsto 0$. One can also check that this condition, combined with $V_{x,y} = 1$, is sufficient as well. This indicates that the definition can only work for operator Kantor pairs such that $G^\pm_2 \neq 0$ if one works over a field of characteristic $2$. Hence, quadratic Jordan algebras will not be considered structurable over fields of characteristic $2$.
\end{remark}

So, we can think of the definition of being structurable as a reasonable way to combine
\begin{itemize}
	\item if $1/2 \in \Phi$, there exists a 1-invertible $x$ in the Kantor pair part (such that the corresponding $\tau$ induces a bijection $G^+ \longrightarrow G^-$),
	\item if $\Phi$ is a field of characteristic $2$, there exists an $x$ for which $Q^\text{grp}_{x^{[2]}}$ is invertible\footnote{We will not prove the equivalence between what we stated here and what we expressed before. But one can check that $\tau : G^- \longrightarrow G^+$ is given by  = $Q^\text{grp}_{x^{[2]}}$. } and $V_{x, Q_{x^{[2]}}^{-1} x } = \zeta$,
\end{itemize}
into a single definition. Now, we explain how the $x = 1 = y$ is a natural choice for the previously considered examples of structurable algebras, excluding the quadratic Jordan algebras. We also ignore the Smirnov algebras, since our construction of the associated operator Kantor pair throws the unit away.

\begin{remark}
	If we use $x = 1$ and $y = 1$ for any of the examples of structurable algebras discussed earlier in this section,
	we can guarantee that the conjugation action $\tau : G^- \longrightarrow G^+$ is given by $(a,b) \mapsto (a,b)$ instead of the $\tau_1(a,b) = (2a, 4b)$ we discussed before.
	More precisely, $(\lambda \cdot 1,\mu \cdot 1)$ is 1-invertible in the setting of Allison and Faulkner if and only if $\lambda \mu = 2$. 
	Each
	\[ \tau_{\lambda,\mu} = \exp_-(\mu) \exp_+(\lambda) \exp_-(\mu)\]
	has as action
	\[ \tau_{\lambda,\mu} (a,b)_+ = (\lambda^2/2 \cdot  a, \lambda^4/4  \cdot b)_-,\quad \tau_{\lambda,\mu} (a,b)_- = (\mu^2/2  \cdot a, \mu^4/4  \cdot b)_+.\]
	This makes $(\lambda,\mu) = (\sqrt{2},\sqrt{2})$ a natural choice.
\end{remark}

\begin{remark}
	One can show that $V_{u,v} = 2 \zeta$ for all $u$, $v$ with $u^{-1} = -u$ and $v^{-1} = -v$ such that
	\[ \exp_-(v) \exp_+(u) \exp_-(v)\]
	reverses the grading. 
	So, involving $\sqrt{2}$ in the definition is only relevant if $1/2 \notin \Phi$.
	Hence, if $1/2 \in \Phi$, the definition simply reduces to requiring the existence of a 1-invertible $x$ (with a nicely behaved $\tau$ associated to it if $1/3 \notin \Phi$).
\end{remark}

Besides the fact that the definition of being structurable fits the typical examples, it also guarantees the existence of an involution such that $G^+_2$ can be seen as skew elements.

\begin{lemma}
	\label{lem: strongly struct}
	Suppose that the operator Kantor pair $(G^+,G^-)$ is structurable with corresponding $x \in G^+/G^+_2$ and $y \in G^-/G^-_2$, then there exists an involution $a \mapsto \bar{a} = a - [y,[x,a]]$.
	Moreover, $G^+_2 \longrightarrow G^+/G^+_2 : s \mapsto [s,y] = Q_{s} y $ is injective with $\overline{Q_s y} = - Q_s y$. 
	\begin{proof}
		A direct computation shows that $\bar{\bar{a}} = a$ for all $a$ using that $[x,[y,[x,a]]] = 2[x,a]$ since $V_{x,y} = \zeta$.
		We remark that $[s,y] = 0$ implies that $[x,[y,s]] = 0$ so that $2 s = 0$.
		If $[s,y] = 0$, we know for $a \in G^+/G^+_2$ that
		\[ [s, \tau^{-1} a] = [s,[y,[y,a]] + \sqrt{2} [P_v, [x,a]]] = [s,[y,[y,a]] + 2[y,[y,[y,[x,a]]]]] = [s,[y,[y,a]]] = 0,\]
		using $2s = 0$ to expand $P_v [x,a]$ as $\sqrt{2} [y,[y,[y,[x,a]]]]$.
		This shows that $Q_{s} = 0$ whenever $[s,y] = 0$, which implies that $s = 0$.
		Finally, we compute \[\overline{ Q_s y} = [s,y] - [y,[x,[s,y]]] = [s,y] + 2 [y,s] = - [s,y] = - Q_s y.\qedhere \]
	\end{proof}
\end{lemma}

Lastly, we remark that being structurable yields a generalization of the characterizing decomposition employed in \cite[Theorem 4]{ALL79}.
To obtain this generalization, we first need certain building blocks. After that we prove a proposition generalizing the decomposition. Thereafter, we explain in Remark \ref{rem: explanation all79} in what sense this is a generalization.

Let $V_1$ be the rank $1$ module over $\Phi$ on which $\exp(u)$, $\exp(v)$, $\text{ad}\; x$ and $\text{ad} \; y$ act trivially.
Set $V_3$ to be the module $\langle w, [y,w], [y,[y,w]] \rangle_\Phi$ which is free of rank $3$ over $\Phi$ and let $[x,w] = 0, [x,[y,w]] = w,$ and $[x,[y,[y,w]]] = [y,w]$. Note that $\exp(u)$ and $\exp(v)$ act on $V_3$ as well.
Finally, we define $V_5$ as the free rank $5$ module
\[ \langle s, [y,s], [y,[y,s]], s_3, s_4 \rangle\]
This has actions of $\exp(v)$ and $\text{ad}(y)$ using that
\[ \exp(v) \cdot s = s + \sqrt{2} [y,s] + [y,[y,s]] + \sqrt{2} s_3 + s_4,\]
so that $[y,[y,[y,s]]] = 3 s_3$ and $[y,s_3] = 4 s_4$. To determine the action of $\exp(u)$ and $\text{ad}(x)$ we can use that $\tau = \exp(v)\exp(u)\exp(v)$ acts on the generating set by reversing the order of the elements.

\begin{remark}
	The $\exp(u)$, $\text{ad}\; x$ and similar maps we used in the definition of the $V_i$ cannot simply be identified with the similar elements contained in the endomorphism algebra of the Lie algebra. For example, $\phi \cdot \text{ad}\; x$ might be zero for nonzero $\phi \in \Phi$. The point is that we will be able to recognize quotients of these $V_i$ in the Lie algebra and that taking this quotient preserves these maps.
\end{remark}

\begin{proposition}
	Suppose that $(G^+,G^-)$ is a structurable operator Kantor pair with corresponding $x$ and $y$. 
	There exists a filtration $0 \subset U_1 \subset U_2 \subset L =  L(G^+,G^-)$ of subspaces of the associated Lie algebra closed under $\exp(u),$ $\exp(v)$, $[x,\cdot]$ and $[y,\cdot]$ and there exist spaces $S = G^+_2$, $H = (G^+/G^+_2)/(Q_S y)$, and $D = L_0/\langle V_{t,y} | t \in G^+ \rangle$, such that
	\[ L/U_2 \cong V_1 \otimes_\Phi D, \quad U_2/U_1 \cong V_3 \otimes_\Phi H, \quad \text{and} \quad U_1 \cong V_5 \otimes_\Phi S.\]
	\begin{proof}
		For $s \in G^+_2$ we define a submodule \[ M_s = \langle \tau^{-1}(s) , \tau^{-1} Q_s y, V_{Q_s y, y}, Q_s y, s \rangle \subset L(G^+,G^-)\]corresponding to the image of $V_5 \otimes s$.
		We remark that each of these generators is nonzero if and only if $s$ is nonzero by Lemma \ref{lem: strongly struct} and $[x,y] = V_{x,y}$. So, this shows that $U_1 = M_{G^+_2} \cong V_5 \otimes_\Phi S$.
		
		For $U_2$ we take the whole Lie algebra, where we replace the $0$-graded part by $\langle V_{t,y} | t \in G^+ \rangle$. This immediately shows that $L/U_2 \cong V_1 \otimes_\Phi D$.
		Lastly, we observe that in $U_2/U_1$ we are looking at elements $t$ in $L_1/(Q_S y)$.
		For such a $t$, we obtain $\tau^{-1} (t), [t,y]$ and $t$ as the logical generators corresponding to the image of $V_3 \otimes t$. It is easy to check that $\tau^{-1} (t) = [y,[y,t]] \mod [x,G^-_2]$ so that these generators are well defined for $t \mod U_2$ contained in $L_1$ and that these generators are $0$ if and only if $t$ is zero. 	
	\end{proof}
\end{proposition}

\begin{remark}
	\label{rem: explanation all79}
	
	In \cite[Theorem 4]{ALL79} the requirement was that $L \cong S \otimes V_5 \oplus H \otimes V_3 \oplus D$ with respect to $\text{ad} \; x$ and $\text{ad} \; y$. The direct sum is, however, too strong to be obtainable if $1/2 \notin \Phi$.
	On the other hand, if $1/6 \in \Phi$, we can fully recover the direct sum of \cite[Theorem 4]{ALL79}. Namely, let $H \oplus S$ be the eigenspaces of $x \mapsto \bar{x}$ in $L_1$ with eigenvalues $\pm 1$. We can identify $U_2/U_1$ with $[y,[y,H]] \oplus [y,H] \oplus H \subset L$. Moreover, we can set $D$ to be $\{ d \in L_0 | [d,x] = 0\}$ using that $V_{h + s/3,y} x = h + s$ for all $h + s \in H \oplus S$.
\end{remark}
	
	\begin{remark}
		One can change the definition we used, by not putting the emphasis on what we can say about the elements $1$ in the Lie algebra and by simply requiring the existence of one-invertible elements with maybe some additional assumptions. Doing so has some inherent difficulties.
		
		For example, $(a,u) \in G_B$ for a (not necessarily unital) associative algebra $B$ with involution is one-invertible if and only if $u$ is invertible. So, there exists a one-invertible element in $G_B$ if and only if there exists an element $1$ in $B$, which means that $B$ should be considered structurable. However, we learn that there exists an invertible element in the wrong way. Namely, by looking at the second coordinate.
		If we assume that $B$ is the unital hull of some unital $C$ and look at $H = \{(g,h) \in C \times B | (g,h) \in G_B\}$, then $H$ is an operator Kantor system if the characteristic of $\Phi$ is $2$.
		This operator Kantor system should not be considered structurable (and is not structurable under the definition we employed). 
		
		On the other hand, we know that $(a,u)$ in $H$ is 1-invertible if and only if $u$ is invertible.
		Some assumptions we could put on $a$, such as requiring the existence of a $b$ such that $V_{a,b} = \zeta$, are satisfied for $a = 1_C$. Other ones, such as $(0,s) \mapsto Q_{(0,s)} a$ being invertible, do fail.
		However, even if we require that the aforementioned assumptions hold, we still run into the problem that $(a,u)^{-1} = (-a, -u + a\bar{a}) \neq (-a , u)$ since the assumptions imply that $a\bar{a} \neq 0 = 2u$. This seems a minor problem. However, because of this ``minor problem" the action of the $\tau$ associated to $(a,u)$ on the Lie algebra is not guaranteed to be of order $2$.
	\end{remark}

	\appendix
	
	\section{Proof of Theorems \ref{thm main} and \ref{thm main 2}}
\label{section proof}

We shall prove Theorems \ref{thm main} and \ref{thm main 2} using a universal construction.
Specifically, let $G^\pm$ be vector groups and $U^\pm$ be their universal representations $\mathcal{U}(G^\pm)$.
Set $F$ equal to the free product $U^+*U^-$ of associative algebras where we identify the units of $U^+, U^-$ so that these units correspond to the unit in $F$. Recall that $\mathcal{H}(G^+,G^-,F)$ is the subalgebra of $F$ generated by all the $\nu_{i,i}(x,y)$ for $x \in G^+$ and $y \in G^-$ in $F$ and that the universal representation is given by $\exp_{[s]}(x) = \sum_{i = 0}^\infty s^i x_i$ for $x \in G^\pm$. We also write $\exp(sx)$ for $\exp_{[s]}(x)$.

We will mod an ideal $I$ out that corresponds to imposing the conditions of Theorem \ref{thm main} on $F$ and nothing more.
In doing so, we will also lift the functions $o_{2,1}, o_{3,1}, \nu_{3,2}$ which are maps that depend on the representation in $A$, to maps $o_{2,1}, o_{3,1}, \nu_{3,2}$ to $F/I$.

We consider the ideal $I$ corresponding to equalities
$$\nu_{(i_1,\ldots,i_n),(j_1,\ldots,j_m)}(x_1,\ldots,x_n,y_1,\ldots,y_m) = u_{(i_1,\ldots,i_n),(j_1,\ldots,j_m),x_1,\ldots,x_n,y_1,\ldots,y_m} \in U^\pm_{i - j}$$ (where we still need to define the $u$) for  $$\quad \sum_{k = 1}^n i_k \ge 2 \sum_{k =1}^m j_k \text{ or } \left(\sum i_k,\sum j_k\right) = (3,2).$$ Note that the $(i_1,\ldots,i_n)$ in $\nu$ denote linearisations; we use it just for indexing purposes in $u$.
Mostly, it does not matter what the $u$ are, except that they lie in the right space $U^\pm$.
However, concrete forms for the $u$ are given by
\begin{itemize}
	\item $u_{2i,i,x,y} = (Q^\text{grp}_xy)_i$,
	\item $u_{3i,i,x,y} = (T_xy)_{2i}$,
	\item $u_{i,j,x,y} = 0$ if $i > 2j$ and $i \neq 3j$,
	\item $u_{3,2,x,y} = (P_x y)_1$,
\end{itemize}
in which $Q^\text{grp}_xy$ denotes an element in $G^\pm$ that maps to $o_{2,1}(x,y)$ in $A$ and in which $T$ and $P$ are defined similarly.
In the case that $Q^\text{grp}, T,P$ are defined from an injective representation we know that $(x,y) \mapsto u_{i,j,x,y} \in U^\pm$ is a homogeneous map\footnote{This requires, technically, an argument since $Q, T$, and $P$ are not necessarily uniquely determined over non flat $\Phi$-algebras $K$. However, we are only interested in the linearisations of $u$, which can be recovered from flat $\Phi$-algebras alone. After we determined the linearisations of $u$, we can recover what $Q$, $T$ and $P$ should be over arbitrary $K$.}, so we can linearize $u_{i,j,\cdot,\cdot}$ to obtain the $$ u_{(i_1,\ldots,i_n),(j_1,\ldots,j_m),x_1,\ldots,x_n,y_1,\ldots,y_m}.$$
If there is no injective representation, then one can take random elements in $U^\pm$ which map to \[\nu_{(i_1,\ldots,i_n),(j_1,\ldots,j_m)}(x_1,\ldots,x_n,y_1,\ldots,y_m) \in A \] under the non-injective representation. We assume that these random elements have the right grading, i.e., $u_{(i_1,\ldots,i_n),(j_1,\ldots,j_m),x_1,\ldots,x_n,y_1,\ldots,y_m}$ is $\sum i_k - \sum j_k$ graded.

\begin{lemma}
	For injective representations, the algebra $F/I$ is a Hopf algebra.
	\begin{proof}
		One easily sees that the generators of the ideal $I$ form a coideal using the fact that
		$$ \Delta(u_{(ki_1,\ldots,ki_n),(kj_1,\ldots,kj_m),\mathbf{v}}) = \sum_{m + n = k}u_{(mi_1,\ldots,mi_n),(mj_1,\ldots,mj_m),\mathbf{v}} \otimes u_{(ni_1,\ldots,ki_n),(nj_1,\ldots,nj_m),\mathbf{v}}$$
		if $\gcd(i_1,\ldots,j_m) = 1$ and that the same holds for $\nu_{(i_1,\ldots,i_n),(j_1,\ldots,j_m)}(\mathbf{v})$, where $\mathbf{v}$ stands for $x_1,\dots,x_n,y_1,\dots,y_m$ with all $x_i \in G^+$ and all $y_i \in G^-$.
		These equalities for $u$ are easily obtained by observing that $U_{2,1,t,x,y} = \sum t^i u_{2i,i,x,y}$ and the similarly defined $U_{3,1,t,x,y}$ are group-like and by using comparison of scalars for the linearisations. For the primitive element $u_{3,2,x,y}$ and its linearisations, we can argue similarly.
		For $\nu$ these equations follow from Lemma \ref{lemma grouplike}.
	\end{proof}
\end{lemma}

We argue along the lines of Faulkner \cite[section 6]{FLK00} to prove Theorem \ref{thm main}.

For a monomial $m = \prod_{i=1}^k {(u_i)}_{n_i} \in F/I$ with $u_i \in G^\pm$ and $n_i \in \mathbb{N}$ and $k \in \mathbb{N}$, we define the \textit{$\sigma$-degree} as
$$ \deg_\sigma{m} = \sum_{u_i \in G^\sigma} n_i. $$
Additionally, the \textit{level} of $m$ is defined as
$$ \lambda(m) = \sum_{(i,j) \in L} n_in_j, $$
where
$$ L = \{ (i,j) : i < j, u_i \in G^+, u_j \in G^- \}.$$

Something useful to note is, if $f(m_2) \le f(m_2)'$ for $f = \deg_\pm$ and $f = \lambda$, then
\begin{equation}
	\label{equation level}
	\lambda(m_1m_2m_3) \le \lambda(m_1m_2'm_3),
\end{equation}
holds, for all $m_1,m_3,$
since $\lambda(m_1m_2) = \lambda(m_1) + \lambda(m_2) + \deg_+(m_1)\deg_-(m_2)$ holds for all $m_1,m_2$.

Let $\mathcal{M}_{pq}(r)$ be the span of all the monomials $m$ with $\deg_+(m) \le p$, $\deg_-(m) \le q$ and $\lambda(m) \le r$.
Define $\mathcal{H}$ as the image of $\mathcal{H}(G^+,G^-,F)$ under the projection $F \longrightarrow F/I$, i.e., it is the subalgebra of $F/I$ generated by all $\nu_{i,i}(x,y)$ for $x \in G^+$ and $y \in G^-$.

Consider the statement $I_n$ for $F/I$: ``all the $\nu_{i,j}(g^+,g^-)$ with $\min(i,j) \le n$ are contained in $\mathcal{U}^+ \mathcal{H} \mathcal{U}^-$".
We know that $I_1$ holds.
We will be using induction to prove $I_n$, as part of a proof of Theorem \ref{thm main}.

Recall that we use $\exp(sx) = \exp_{[s]}(x)$ to denote the image of $x \in G^\pm$ under the universal representation in $F$. 
We work in $F/I$ from now onward.

\begin{lemma}
	\label{lemma xpyq equiv wpq}
	For all $p,q \in \mathbb{N}$ and all $x \in G^+, y \in G^-$ we have
	$$ x_py_q \equiv \nu_{p,q}(x,y) \mod \mathcal{M}_{pq}(pq-1).$$
	\begin{proof}
		We look at the terms with coefficient $s^pt^q$ in $$\exp(sx) \cdot \exp(ty) \cdot \exp(sx^{-1}) = \prod \exp( o_{a,b}(x,y), s^at^b),$$
		and realize that $x_py_q$ is the only term that is not necessarily $0$ mod $\mathcal{M}_{pq}(pq-1)$ on the left-hand side. If we look at the right-hand side, $\nu_{pq}(x,y)$ is the only term which is not a product of multiple terms.
		Each of these products has a lower level since the last contributing term has nonzero $\deg_+$.	
	\end{proof}
\end{lemma}

\begin{lemma}
	\label{lemma Faulkner}
	Consider $n \in \mathbb{N}$ and $d = \text{gcd}\{ \binom{n}{i} | 1 < i < n\}$. In that case $d \neq 1$ if and only if $n = p^e$ for a prime $p$ and $d = p$.
	\begin{proof}
		Faulkner \cite[Lemma 21]{FLK00} proved this.
	\end{proof}
\end{lemma}

\begin{lemma}
	\label{lemma gcds}
	Consider even $n > 2$ and $d = \text{gcd}\{ \binom{n}{i} | 1 < i < n, 2i \neq n\},$ then $d = 2$ or $d = 4$ if $n = 2^e$, $d = p$ if $n = 2p^e$ for an odd prime $p$, and $d = 1$ in all other cases.
	\begin{proof}
		We write $2k = n$.
		
		Set $S = \{ \binom{2k}{i} | 1 < i < 2k, i \neq k\}$.
		Suppose that there is an odd prime $p$ such that $p | S$ and write $2k = p^e l$ with $p \nmid l$.
		We have
		$$ 1 + \binom{2k}{k} t^k + t^{2k} \equiv (1 + t)^{2k} \equiv 1 + \binom{l}{1}t^{p^e} + \binom{l}{2}t^{2p^e} + \ldots \mod p.$$
		So, $l = 2$, $k = p^e$ and $p^i \mid S$ for some $1 \le i \le e$.
		We know that there exists $0 < j < k$ such that
		$$(1 + t)^k = 1 + a_j t^j + O(t^{j+1}) \mod p^2$$ with $a_j \neq 0 \mod p^2$, by Lemma \ref{lemma Faulkner}.
		Hence, we obtain $$(1 + t)^{2k} \equiv 1 + 2 a_j t^j + O(t^{j+1}) \mod p^2.$$
		So, we conclude that $p^2 \nmid S$ and thus $d = p$.
		
		Now, suppose that $n = 2^i = 2k.$
		Take $j$ such that
		$$ (1 + t)^{k} \equiv 1 + a_jt^j + O(t^{j+1}) \mod 4,$$
		with $a_j \equiv 2 \mod 4$, which exists by Lemma \ref{lemma Faulkner}.
		We compute that
		$$ (1 + t)^{2k} \equiv (1 + 2 a_j t^j + O(t^{j+1})) \equiv 1 + 4 t^{j} + O(t^{j+1}) \mod 8.$$
		Hence, the maximal power of $2$ that could divide all of $S$ is $4$.
	\end{proof}
\end{lemma}

\begin{lemma}
	\label{lemma xpyq in Mpqpq-1}
	If $I_n$ holds and if we consider $p \neq q, \min(p,q) \le n + 1$, $x \in G^+,$ and  $y \in G^-$, then we know
	$$ x_py_q \in \mathcal{M}_{pq}(pq - 1).$$
	\begin{proof}
		If $\min(p,q) \le n$, this is true since $I_n$ holds.
		So, suppose that $q = n + 1$ (the case $p = n + 1 $ is similar).
		We can assume that $n + 1 < p < 2n + 2$, as $ p \ge 2q$, either learns us that we imposed $\nu_{p,q}(x,y) = (T_xy)_q$ if $p = 3q$, $\nu_{p,q}(x,y) = (Q^{\text{grp}}_{x}y)_q$ if $p = 2q$, or $\nu_{p,q} = 0$ if $p \neq 3q$, which implies that $x_py_q \in \mathcal{M}_{pq}(pq-1)$ by Lemma \ref{lemma xpyq equiv wpq}.
		
		Suppose first that $x \in G^+_2$ and $y \in G^-_2$.
		We can assume that $p = 2l, q = 2k$ since $x_py_q = 0$ otherwise.
		Using the fact that
		\[ \binom{l}{a} x_p = x_{2a}x_{p-2a},\]
		which holds by Construction \ref{construction universal representation}.\ref{cons:g2 scalar},
		and the fact that $x_{2a}x_{p-2a}y_q = x_{p-2a}x_{2a}y_q \in \mathcal{M}_{pq}(pq-1),$ which holds by induction since $\min(a,p-1) \le n$, we see that all $\binom{l}{i}$-multiples of $x_py_q$ with $0 < i < l$ and all $\binom{k}{j}$-multiples with $0<j<k$ are contained in $\mathcal{M}_{pq}(pq-1)$. The greatest common divisor of all those binomial coefficients is $1$ by Lemma \ref{lemma Faulkner} since there are no powers $a,b$ of the same prime such that $a < b < 2a$. So, we conclude that $x_py_q \in \mathcal{M}_{pq}(pq-1)$, since $\gcd(a_1,\ldots,a_n) = k$, for $a_i \in \mathbb{Z}$, implies that there exist integers $b_1,\ldots,b_n$ such that
		\[ \sum_{i = 1}^k a_ib_i = k.\]	
		Suppose now that $x \in G^+$ and $y \in G^-_2$.
		We will use that
		$$\binom{m}{k} x_m= \sum_{\substack{b + c = k\\ a +c = m}} x_a x_b (x(-x))_{2c}, \quad \binom{m}{k} y_{m} = \sum_{\substack{b + c = k\\ a +c = m}} y_a y_b (2 \cdot_2 y)_{2c},$$
		i.e., \ref{construction universal representation}.\ref{cons:g scalar}.
		Analoguous to the previous case, one argues that all $\binom{p}{i}, 0 < i < p, \binom{q}{j}, 0 < j < q$ multiplies of $x_py_q$ are contained in  $\mathcal{M}_{pq}(pq-1)$, except possibly the $\binom{q}{q/2}$ multiples since the expression for $\binom{q}{q/2} y_mq$ contains $(2 \cdot_2 y)_q$.
		If $q = 2$, and thus $p = 3$, we can use that $x_3y_2 \equiv \nu_{3,2}(x,y) \mod \mathcal{M}_{32}(5)$ and $\nu_{3,2}(x,y) \in \mathcal{M}_{10}(0) \subset \mathcal{M}_{32}(5)$.
		If $q > 2$, then the set $S = \{ \binom{q}{a} | a \neq 1,q,q/2\}$ is nonempty. Since $p$ and $q$ cannot both be powers of the same prime, we can use Lemma \ref{lemma gcds} and Lemma \ref{lemma Faulkner} to conclude that $x_py_q \in \mathcal{M}_{pq}(pq-1)$.

		Lastly, suppose that $x \in G^+$ and $y \in G^-$.
		Using the previous case and analogous argumentation to the first case one proves that $x_py_q \in \mathcal{M}_{pq}(pq-1)$.
	\end{proof}
\end{lemma}

\begin{lemma}
	\label{Lemma decomposition}
	Each element of $F/I$ can be written as an element in $U^-\mathcal{H}U^+ \le F$.
	\begin{proof}
		We will show that $I_n$ implies that the statement $J_r$: ``$\mathcal{M}_{pq}(r)$ is contained in $U^-\mathcal{H}U^+$ for all $p,q$ such that $\min(p,q) \le n + 1$", holds for all $r$.
		Since $I_1$ holds, proving that all $J_r$ holds, lets us conclude that the lemma holds, as it proves by induction (i) that all $I_n$ hold and (ii) that the lemma holds if all $I_n$ hold since $\bigcup_{p,q,r} \mathcal{M}_{p,q}(r) = F$.
		
		We use induction on $r$ to prove $J_r$.
		If $r = 0$, this is trivial.
		So, suppose that $r > 0$.
		Take a monomial $m \in \mathcal{M}_{pq}(r)$ but not in $\mathcal{M}_{pq}(r-1)$.
		This monomial cannot factor as
		$ m = m_1 x_a y_b m_2,$
		with $0 \neq a \neq b \neq 0$, $x \in G^+$ and $y \in G^-$ by Lemma \ref{lemma xpyq in Mpqpq-1} and (\ref{equation level}).
		Now, consider a factorization $m_1 x_iz_jy_cm_2,$ with $i,j,c \neq 0$ with $x,z \in G^+$ and $y \in G^-$. Using the previous case, we conclude that $j = c$. Now, we use that $i \neq 0$, to obtain that
		$$ m_1x_iz_cy_cm_2 = m_1 \left((xz)_{i+c} - \sum_{k+l = i+c, l \neq c} x_kz_l\right)y_c m_2 \in \mathcal{M}_{pq}(r-1)$$
		using the previous case.
		Similarly, it cannot factor as $m_1 x_a y_c w_dm_2$
		with $cd \neq 0$ or $0 \neq c + d \neq a$ for $x \in G^+$ and $y,w \in G^-$. 
		
		Hence, $m$ decomposes as a product
		$$ m_1 \cdot \prod (x_i)_{p_i}(y_i)_{p_i} \cdot m_2,$$
		with $m_1 \in U^-$, $m_2 \in U^+$, $x_i \in G^+$ and $y_i \in G^-$.
		Applying Lemma \ref{lemma xpyq equiv wpq} learns us that this product can be rewritten as 
		$$ m_1 \cdot \prod (x_i)_{p_i}(y_i)_{p_i} \cdot m_2 \equiv m_1 \cdot \prod \nu_{p_i,p_i}(x,y) \cdot m_2 \mod \mathcal{M}_{pq}(r - 1).$$
		So, $m$ is an element of $U^- \mathcal{H}U^+$ by the induction hypothesis.
	\end{proof}
\end{lemma}

\begin{proof}[Proof of Theorem \ref{thm main}]
	The non-moreover part is Lemma \ref{Lemma decomposition}.
	For the moreover part, we first prove that the image of $V^+$ of $U^+ = \mathcal{U}(G^+)$ is a direct summand of $F/I$ (without needing the projectivity).
	Suppose that we have an element 
	$$ \sum_{i = 1}^k y_i h_i x_i,$$
	with $y_i \in V^-$, $x_i \in V^+$ and $h_i \in \mathcal{H}$ and assume that all the $y_i,x_i$ are nicely graded (i.e., homogeneous). We want to prove that 
	$$ \sum_{i = 1}^k y_i h_i x_i \in V^+,$$
	implies that $y_ih_i\in \Phi$ for all $i$.
	First of all, we can assume that $\epsilon(x_i) = 0$ for all $i$, since $\sum y_ih_i \in V^+$ implies, using the fact that this element is necessarily $0$-graded, that $\epsilon (\sum y_ih_i) = \sum y_ih_i$.
	Take the terms with the $y_i$ with minimal grading in $F/I$ (or equivalently with maximal grading in $V^-$), which we assume to be $-n$, and take of those terms the ones with $x_i$ maximally graded, which we assume to be $m$ graded. Note that $m > 0$.
	We call the set of those terms $S$.
	We compute if $-n < 0$, using $\pi_i$ the projection on the $i$th grading component of $F/I$, that
	$$ 0 = \mu \left((\pi_{-n} \otimes S \otimes \pi_m)\Delta^3\left(\sum_{i = 1}^k y_i h_i x_i\right)\right) = \sum_{yhx \in S} yh_{(1)}S(h_{(2)})h_{(3)}x = \sum_{t \in S} t$$
	using that the function of the left-hand side must evaluate to $0$ on $U^+$, the fact that $\Delta(h) = h \otimes 1 + 1 \otimes h \mod \ker \epsilon \otimes \epsilon$, and where we used Sweedler summation notation so that
	$$ h_{(1)}S(h_{(2)})h_{(3)} = h,$$
	for all $h \in \mathcal{H}$.
	
	Now, to resolve the $n = 0$ case, consider $\sum_i h_ix_i$ with $h_i \notin \Phi$. We can assume that $\epsilon(h_i) = 0$ for all $i$. We use
	$$ 0 = \mu \left( S(1 - \epsilon)\pi_0 \otimes \pi_m) \Delta\left(\sum_i h_ix_i\right)\right) = \sum_i S(h_{i,(1)})h_{i,(2)}x_i - \sum_i h_ix_i = -\sum_i h_ix_i. $$
	This proves that $V^+$ is a direct summand. Note that the projection $\pi_{V^+}$ acts as
	$$ \pi_{V^+}(yhx) = \epsilon(yh)x$$
	for $y \in V^-, h \in \mathcal{H}, x \in V^+$.
	
	For projective $G^+$, we can see that Hopf algebra map $f : \mathcal{U}(G^+) \longrightarrow F/I$ is an embedding because
	$$  (f(1 - \epsilon))^{\otimes n+1}\Delta^n(y) = (1 - \epsilon)^{\otimes n+1}\Delta^n(f(y)),$$
	for all $y \in \mathcal{U}(G^\pm)$ and since $f$ induces an embedding of $\text{Lie}(G^+) \longrightarrow F/I$ (the map $U^+ \longrightarrow A$ factors through $F/I$). More precisely, using the well-behavedness of projective vectorgroups this can be used to prove that 
	$ (1 - \epsilon)^{\otimes n+1}\Delta^nf$ corresponds to an injective map $Y_{n+1}/Y_n \longrightarrow (F/I)^{\otimes n+1}$ for all $n$ (namely, one can use the argumentation of Lemma \ref{Lemma well behaved}) which can be used to prove the injectivity $Y_{n+1} \longrightarrow F/I$ using induction on $n$. Thus, we see that $V^\pm \cong \mathcal{U}(G^\pm)$.
	
	Now, we only need to prove that $F/I \cong V^- \otimes \mathcal{H} \otimes V^+$ as modules since this implies that both spaces are isomorphic as coalgebras.
	It is easy to check that $$\mu \circ (S \otimes \text{Id} \otimes S) \circ (\pi_{V_-} \otimes \text{Id} \otimes \pi_{V_+}) \circ \Delta^2 (yhx)  = \epsilon(x)\epsilon(y)h$$
	for all $y \in U^-, h \in \mathcal{H}, x \in U^+$. This yields a projection $\pi_h : F/I \longrightarrow \mathcal{H}$.
	
	Note that $yhx = \mu(\pi_{V^-} \otimes \pi_h \otimes \pi_{V^+})\Delta^2(yhx)$ for all $y \in V^-, h \in \mathcal{H}, x \in V^+$. So, we conclude that 
	$$ yhx \mapsto (\pi_{V^-} \otimes \pi_h \otimes \pi_{V^+})\Delta^2(yhx)$$
	forms a coalgebra isomorphism $F/I \longrightarrow V^- \otimes \mathcal{H} \otimes V^+$.
\end{proof}

Now, we assume that we are working with the vector groups $\rho^+(G^+)$ and $\rho^-(G^-)$ in order to prove Theorem \ref{thm main 2}.
So, we use the previously considered $F/I$ with $U^+$ the universal representation of $\rho^+(G^+)$ and similarly defined $U^-$.

Consider the Lie subalgebra $L$ of $F/I$ generated by the $g_1,h_1$ for $g \in G^+$ and $h \in G^-$.
We add a grading element of $F/I$, namely take $F/I \otimes \Phi[\zeta]$ with associative product defined from
$$ k \otimes \zeta \cdot l \otimes p(\zeta) = kl \otimes \zeta p(\zeta) + i kl \otimes p(\zeta),$$
for $l$ in the $i$-th grading component of $F/I$ and $p(\zeta) \in \Phi[\zeta]$.
We effectively added a single generator $\zeta$ and relations $[\zeta , l] = i l$ for $i$-graded $l$.
Given this interpretation, we do not write the $\otimes$-sign anymore.
We use $J$ to denote the ideal $$Z(L \oplus \Phi\zeta) \cap \{ z \in (F/I)_0 | \sum_{i + j = 2} g_izg^{-1}_{j} = 0, \forall g \in G^\pm \},$$ i.e., the elements $e$ of the center of the Lie algebra $L \oplus \Phi\zeta$ which are $0$-graded so that $\exp(g) e \exp(g^{-1}) = e$ for all $g \in G^\pm$, with $\exp(g) = \sum_{i = 0}^\infty g_i$ (one can prove that only $g_i$ with $i \le 2$ contribute in the conjugation action).

\begin{lemma}
	The Lie algebra $(L \oplus \Phi \zeta)/J$ is isomorphic to the Lie algebra associated to the pre-Kantor pair $(\rho^+(G^+),\rho^-(G^-))$.
	\begin{proof}
		We first prove that the underlying modules of both Lie algebras are the same. We see that $\text{Lie}(\rho^+(G^+))$ embeds into $F/I$ and the same holds for $G^-$. The $0$-graded part of the Lie algebra can be identified with certain maps $G^\pm \longrightarrow \text{Lie}(G^\pm)$ using $[\delta,g_1] = \delta(g)_1$ and $g_2\delta - g_1\delta g_1 + \delta \times (g^{-1})_2 = - \delta(g)_2 + \delta(g)_1g_1$, and this identification is bijective because we divided the ideal $J$ out.
		To prove that all the $V_{x,y}$ are contained in $L_0$, one uses Lemma \ref{Lemma equations} to prove that $[x,y]$ corresponds exactly to $V_{x,y}$.
		
		One easily checks that the other brackets coincide, proving the isomorphism.
	\end{proof}
\end{lemma}

We call the action
$$ x \mapsto (a \mapsto x \cdot a = x_{(1)}aS(x_{(2)}))$$
of $F/I$ on itself, using Sweedler summation notation, the \textit{adjoint action}.

\begin{lemma}
	The adjoint action of $F/I$ on $L \oplus \Phi \zeta$ induces an action of the pre-Kantor pair $(\rho^+(G^+),\rho^-(G^-))$ on $(L \oplus \Phi \zeta)/J$ which coincides with the usual action of a pre-Kantor pair on its associated Lie algebra.
	\begin{proof}
		We only need to check that certain elements act as expected on $L \oplus \Phi \zeta$ under the adjoint action.
		For the elements $g_1$, this is the case since the Lie algebras are isomorphic.
		For elements $g_3,g_4$ this is easily checked using that $T,P,R$ coincide, on the Lie algebra, with $\mu_{i,j}(g,h)$ for $(i,j) = (3,1), (3,2),(4,2)$ respectively.
		For $g_2$ one notes that the map $\tau_{g,v}$ acts as the same as $g_2v_2 - g_1v_2g_1 + v_2(g^{-1})_2$ for $v_2 \in [G,G]$ and that $g_2 \cdot \delta  = - \delta(g)_2 + (\delta(g)_1)g_1$ is used as part of the definition of $\delta(g)$ for $0$-graded primitive $\delta$.
	\end{proof}
\end{lemma}

\begin{proof}[Proof of Theorem \ref{thm main 2}]	
	We only need to prove the axioms for operator Kantor pairs, as the rest was already proved below the statement of Theorem \ref{thm main 2}.
	All the axioms for operator Kantor pairs follow from the fact that $\nu_{3,2}(x,y) = (P_xy)_1$, $\nu_{2i,i}(x,y) = (Q^\text{grp}_xy)_i$ for $i = 2,3$ and $\nu_{5,2}(x,y) = 0$, using Remark \ref{remark sufficient condition}, since the adjoint action satisfies 
	\[ x \cdot ab = (x_{(1)} \cdot a)(x_{(2)} \cdot b). \qedhere \]
\end{proof}
	
\section{Computations}

\subsection{Lemma 1}

For a pre-Kantor pair $(G^+,G^-)$ we have an operator $V$.
We will often write $V_{x,y}z$ with $z \in A^+ \cong G^+/(G^+_2)$, i.e., the space of possible first coordinates.
In that case, the definition of $V$ becomes
\[ V_{x,y }z = - Q^{(1,1)}_{z,x} y.\]

\begin{lemma}
	\label{Lemma further linearizations P,R}
	Let $(G^+,G^-)$ be a pre-Kantor pair. The following equations hold for all $g,h,x,a \in G^+, y,y',b,c \in G^-$:
	\begin{align*}
		P^{((1,2),(1,1))}_{a,g}(y,y') = &  V_{Q_g y, y'} a + V_{g,y'} V_{g,y} a , \\
		R^{((2,2),(1,1))}_{g,h}(c,b) = & - V_{g,b}T^{(1,2)}_{g,h} c - V_{h,b} T^{(2,1)}_{g,h} c  + \psi(Q_g c,Q_h b) + \psi(Q^{(1,1)}_{g,h} c, Q^{(1,1)}_{g,h} b) + \psi(Q_h c,Q_g b), \\
		P^{(1,1,1)}_{g,x,a}y = & \; Q^{(1,1)}_{g,x}Q_{y^{-1}} a - V_{a,y} Q^{(1,1)}_{g,x} y, \\
		R^{(2,1,1)}_{g,x,a} y = &T^{(2,1)}_{g,x}Q_{y^{-1}} a - V_{a,y} T^{(2,1)}_{g,x} y + \psi(Q_g y, Q^{(1,1)}_{x,a} y) + \psi(Q^{(1,1)}_{g,x}y,Q^{(1,1)}_{g,a} y).
	\end{align*}
	\begin{proof}
		These equations are obtained by linearizing the expressions in the Definition \ref{Definition pre-Kantor pair} for $P^{(3,(1,1))}$, $R^{(4,(1,1))},$  $P^{(2,1)}$ and $R^{(3,1)}$ respectively.
		For the first equation, linearizing only yields
		\[ Q_{T^{(1,2)}_{a,g}y} y' - V_{a,y'} Q_g y - V_{g,y'} Q^{(1,1)}_{a,g} y.  \]
		Applying $T^{(1,2)}_{a,g} y = [a,Q_gy],$ $Q_{[u,v]} w = Q^{(1,1)}_{u,v} w - Q^{(1,1)}_{v,u} w$, and $Q^{(1,1)}_{a,g} y = - V_{g,y} a$ yields the desired result. The other expressions are obtained more easily.
	\end{proof}
\end{lemma}

\begin{lemma}
	Let $(G^+,G^-)$ be a pre-Kantor pair. For $x \in G^+, a,b \in G^-$ we have
	\label{Lemma linerizations tau}
	$$ \tau^{(2,(1,1))}_{x,a,b}  = - V_{b,Q_{x}a} + V_{a,x}V_{b,x},$$
	$$ \tau^{((1,1),2)}_{b,a,x^{-1}} =  - V_{Q_{x}a,b} + V_{b,x}V_{a,x}.$$
	\begin{proof}
		In order to avoid subtleties with the second coordinate, we assume that $d(g_1,g_2) = (d_1(g_1),d_2(g_2))$ with $d_i$ a linear function for derivations $d$. This can be assumed if we work with the first $2$ coordinates of the universal representation.
		Consequently, we can say that
		$$ \delta(\psi(a,b)) = \psi(\delta(a),b) + \psi(a,\delta(b)),$$
		without ever running into identification issues.

		Both equations that we want to prove are equivalent. To observe that this is the case set $g = (t \cdot_1 a)b \in G(\Phi[t])$, and consider the equation
		$$ \tau_{x,g} + \tau_{g^{-1},x^{-1}} - V_{x,g}^2 = 0$$
		used in the definition of $\tau$, cfr., Definition \ref{definition V tau}.
		Comparing the terms belonging to $t$ shows that both equations are equivalent.
		
		So, suppose that $x \in G^+, a,b \in G^-$.
		We show that the first equality holds on $G^-$ and the second on $G^+$, in order to prove that both equations hold in $\text{InStr}(G)$.
		
		We first evaluate the $\tau^{(2,(1,1))}_{x,a,b}$ at $g \in G^-$. 
		The first coordinate is
		$$ P^{(1,1,1)}_{g,b,a}x^{-1} = (V_{Q_{x}a,b} + V_{a,x}V_{b,x})_1 g$$
		by Lemma \ref{Lemma further linearizations P,R}. Hence, the left and right-hand side agree on the first coordinate.
		The second coordinate is
		$$ R^{(2,1,1)}_{g,b,a} x^{-1} - \psi(Q_g x, Q^{(1,1)}_{b,a} x) + \psi(P^{(1,1,1)}_{g,b,a}x^{-1},g)$$
		which equals
		$$ T^{(2,1)}_{g,b}Q_{x} a - V_{a,x} T^{(2,1)}_{g,b} x^{-1} + \psi(Q^{(1,1)}_{g,b}x,Q^{(1,1)}_{g,a} x) + \psi(P^{(1,1,1)}_{g,b,a}x^{-1},g).$$ 
		We can factor $V_{a,x}$, using the expression for $P^{(1,1,1)}$ of Lemma \ref{Lemma further linearizations P,R} and the fact that $\delta(\psi(a,b))$ equals $\psi(\delta a,b) + \psi(a, \delta b)$ to obtain
		$$  T^{(2,1)}_{g,b}Q_{x} a + \psi(V_{b,Q_xa}g,g) + V_{a,x}(T^{(2,1)}_{g,b}x + \psi(V_{b,x} g,g)).$$
		So, by evaluating the right-hand side of the equation we want to prove, we conclude that the left-hand side and the right-hand side have the same action on $G^-$.
		
		Now, we prove that the second equality holds as functions on $G^+$.
		The first coordinate of $\tau^{(1,1)}_{b,a,x^{-1}} g$ equals
		$$ P^{((1,2),(1,1))}_{g,x}(a,b) = V_{Q_{x}a,b}g + V_{x,b}V_{x,a}g.$$
		This proves the equality on the first coordinate.
		The second coordinate equals
		$$ R^{((2,2),(1,1))}_{g,x}(a,b) - \psi(Q_ga,Q_xb) - \psi(Q_gb,Q_xa) + \psi(P^{((1,2),(1,1))}_{g,x}(a,b),g).$$
		By substituting the linearisation of $R$ appearing in Lemma \ref{Lemma further linearizations P,R}, we can conclude that the second coordinate equals
		\begin{equation}
			\label{eq1}
			- V_{g,b}T^{(1,2)}_{g,x} a - V_{x,b} T^{(2,1)}_{g,x} a  + \psi(Q^{(1,1)}_{g,x} a, Q^{(1,1)}_{g,x} b) + [Q_x a,Q_g b] +  \psi(P^{((1,2),(1,1))}_{g,x}(a,b),g).
		\end{equation}
		We want to prove that this expression equals the second coordinate of $(V_{Q_{x}a,b} + V_{b,x}V_{a,x})(g),$ i.e., that it equals
		$$ T^{(2,1)}_{g,Q_xa} b + \psi(V_{Q_xa,b}g,g) + V_{b,x}(T^{(2,1)}_{g,x}a + \psi(V_{a,x}g,g))) .$$
		We use the equality
		$$ T^{(2,1)}_{g,Q_xa} b = T^{(1,2)}_{Q_xa,g} b + T^{(2,1)}_{[Q_xa,g],g} b  = [Q_xa,Q_gb] - V_{g,b} T^{(1,2)}_{g,x} a, $$
		to rewrite the expression above as
		$$ - V_{g,b}T^{(1,2)}_{g,x} a + [Q_x a,Q_g b]  + \psi(V_{Q_xa,b}g,g) + V_{b,x}(- T^{(2,1)}_{g,x}a + \psi(V_{a,x}g,g))).$$
		So, to conclude the equality between that and Equation (\ref{eq1}), we must check
		$$ V_{x,b}T^{(2,1)}_{g,x} a -  \psi(V_{x,a}g, V_{x,b}g) - \psi(V_{Q_xa,b}g + V_{x,b}V_{x,a}g,g) + \psi(V_{Q_xa,b}g,g ) + V_{b,x}(T^{(2,1)}_{g,x}a + \psi(V_{a,x}g,g))) = 0,$$
		as it is the previous expression minus Expression (\ref{eq1}) in which we substituted the known expression for the $P$ and replaced all the $Q^{(1,1)}$'s with $V$'s. We see that this is the case, using the fact that $V_{x,b}$ acts as a derivation on $\psi$.
	\end{proof}                                 
\end{lemma}

\subsection{Computation of $T$}
\label{appendix T}
We work with the structurable algebra and corresponding vector group associated to a hermitian cubic norm structure.
Set $g = ((a,j),(u,aj + j^\sharp)) \in G$.
First, we compute 
\begin{align*}
	Q_g(b,k)  = & ((a,j)(\bar{b},k))(a,j) - (u,aj + j^\sharp)(b,k) \\
	= & (a\bar{b} + T(j,k), ak + bj + j \times k)(a,j) \\ & - (ub + aT(j,k)+ T(j^\sharp,k), uk + a\bar{b}j + \bar{b}j^\sharp + \bar{a}\cdot j \times k + j^\sharp \times k) \\
	= & (a^2\bar{b} - ub + aT(k,j) + bT(j,j) + T(j^\sharp,k),\\ & a\bar{a}k - uk + T(j,k)j + b\bar{a}j - \bar{a} \cdot j \times k + \bar{b} j^\sharp + (j \times k) \times j - j^\sharp \times k).	
\end{align*}
By linearising, we obtain that $l =  Q^{(1,1)}_{(a,j),(a,j)}(b,k)$ equals
\begin{align*}
	(& 2a^2\bar{b} - a\bar{a}b + 2aT(k,j) + bT(j,j) + T(j\times j,k), \\ & a\bar{a}k - T(j,j)k + 2T(j,k)j + 2b\bar{a}j - 2\bar{a}\cdot j\times k + 2\bar{b} j^\sharp + 2(j\times k)\times j - (j\times j)\times k).
\end{align*}
We want to compute $[l,(a,j)]$.
This is easily computed by computing the first coordinate of $l(\bar{a},j)$ and calling it $t$, and then using $[l,(a,j)] = (t - \bar{t},0)$.
We get that
\begin{align*}
	t = & \; 2a^2\bar{b}\bar{a} - ab\bar{a}^2 + 3a\bar{a}T(k,j) + b\bar{a}T(j,j) + \bar{a}T(j \times j,k) - T(j,j)T(k,j) + 2T(j,k)T(j,j)\\ & \;  + 2b\bar{a}T(j,j) + 2\bar{a}T(j \times k,j) + \bar{b}T(j\times j,j) + 2T((j \times k) \times j,j) - T((j\times j)\times k,j).
\end{align*}
Using that $$T(j \times k,j) = T(j \times j,k), \quad T((j\times k)\times j,j) = T(j,(j \times j) \times k), \quad 2 j^\sharp = j \times j, \quad 3N(j) = T(j,j^\sharp)$$ we obtain 
$$ t - \bar{t}= 3(s - \bar{s})$$
with
$$ s =  -ab\bar{a}^2 + a\bar{a}T(k,j) + \bar{a}bT(j,j) - aT(k,j \times j) - T(j,j)T(k,j) - 2 bN(j) - T(j\times(j\times j),k).$$
We can rewrite this $s$ as
$$ s= (a\bar{a} - T(j,j))(T(k,j) - b\bar{a}) - 2 \left( T(ak,j^\sharp) + bN(j) + T(j^\sharp \times k,j)\right).$$
Finally, we can compute $T$ if there is no $3$-torsion using the formula
$$ T_g(b,k) = [(a,j),Q_g(b,k)] + s - \bar{s},$$
proved in Lemma \ref{lemma constructing structurable algebras}, using that $3(\bar{s} - s)$ coincides with the right-hand side of the defining expression of $T$. To see that this $s$ coincides with that expression, observe that the mentioned expression in Lemma \ref{lemma constructing structurable algebras}, is obtained by computing $[(a,j),Q^{(1,1)}_{(a,j),(a,j)}(b,k)] = 3(\bar{s} - s)$.
Using the same technique as before, we compute that
$$ [Q_g(b,k),(a,j)] = v - \bar{v},$$
with
$$ v = b\bar{a}(2T(j,j) - u - a\bar{a}) + T(k,j)(2a\bar{a} - T(j,j) - u) - T(ak,j^\sharp) - 3\left(bN(j) + T(j^\sharp \times k,j)\right).$$
So, we conclude that
$$ T_g(b,k) = w - \bar{w},$$
with $w = s - v = (u - T(j,j))(b\bar{a}) + (u - a\bar{a})T(k,j) - T(ak,j^\sharp) + bN(j) + T(j^\sharp \times k,j)$. 
	
	\paragraph{Acknowledgements}
	
	Sigiswald Barbier is supported by a FWO postdoctoral junior fellowship
	from the Research Foundation Flanders (1269821N).
	
	\paragraph{Conflict of Interest}
	
	The authors have no conflict of interest to declare that are relevant to this article.

\end{document}